\documentclass[12pt]{amsart}
\usepackage{lmodern}
\usepackage{ifthen}
\usepackage{amsfonts}
\usepackage{amsxtra}
\usepackage{amssymb}
\usepackage{array}
\usepackage{eucal}
\usepackage[margin=1.2in]{geometry}
\usepackage{xcolor}
\definecolor{cite}{rgb}{0.80,0.00,1.00}
\definecolor{url}{rgb}{0.00,0.00,0.80}
\definecolor{link}{rgb}{0.10,0.30,0.20}
\usepackage[colorlinks,linkcolor=link,urlcolor=url,citecolor=cite,pagebackref,breaklinks]{hyperref}
\usepackage{colonequals}
\usepackage{bbm}
\usepackage{mathtools}
\usepackage{mathrsfs}
\usepackage{appendix}
\usepackage[all]{xy}
\usepackage[lite,abbrev,msc-links,alphabetic]{amsrefs}

\usepackage[OT2,T1]{fontenc}
\DeclareSymbolFont{cyrletters}{OT2}{wncyr}{m}{n}
\DeclareMathSymbol{\Sha}{\mathalpha}{cyrletters}{"58}

\numberwithin{equation}{section}

\theoremstyle{plain}
\newtheorem{proposition}{Proposition}[section]
\newtheorem{conjecture}[proposition]{Conjecture}
\newtheorem{corollary}[proposition]{Corollary}
\newtheorem{lem}[proposition]{Lemma}
\newtheorem{theorem}[proposition]{Theorem}

\theoremstyle{definition}
\newtheorem{definition}[proposition]{Definition}
\newtheorem{notation}[proposition]{Notation}
\newtheorem{assumption}[proposition]{Assumption}

\theoremstyle{remark}
\newtheorem{remark}[proposition]{Remark}

\renewcommand{\b}[1]{\mathbf{#1}}
\renewcommand{\c}[1]{\mathcal{#1}}
\renewcommand{\d}[1]{\mathbb{#1}}
\newcommand{\f}[1]{\mathfrak{#1}}
\renewcommand{\r}[1]{\mathrm{#1}}
\newcommand{\s}[1]{\mathscr{#1}}
\renewcommand{\sf}[1]{\mathsf{#1}}
\renewcommand{\(}{\left(}
\renewcommand{\)}{\right)}
\newcommand{\res}{\mathbin{|}}

\newcommand{\Sec}{\S}

\newcommand{\bG}{\b G}

\newcommand{\bK}{\b K}

\newcommand{\bP}{\b P}

\newcommand{\cD}{\c D}

\newcommand{\cI}{\c I}

\newcommand{\cO}{\c O}

\newcommand{\cS}{\c S}
\newcommand{\cT}{\c T}

\newcommand{\cX}{\c X}
\newcommand{\cY}{\c Y}
\newcommand{\cZ}{\c Z}

\newcommand{\dA}{\d A}

\newcommand{\dC}{\d C}

\newcommand{\dF}{\d F}

\newcommand{\dQ}{\d Q}
\newcommand{\dR}{\d R}
\newcommand{\dS}{\d S}
\newcommand{\dT}{\d T}

\newcommand{\dZ}{\d Z}

\newcommand{\fD}{\f D}

\newcommand{\fM}{\f M}

\newcommand{\fP}{\f P}

\newcommand{\fa}{\f a}

\newcommand{\fd}{\f d}

\newcommand{\ff}{\f f}

\newcommand{\fl}{\f l}
\newcommand{\fm}{\f m}
\newcommand{\fn}{\f n}

\newcommand{\fv}{\f v}

\newcommand{\rB}{\r B}

\newcommand{\rE}{\r E}

\newcommand{\rG}{\r G}
\newcommand{\rH}{\r H}
\newcommand{\rI}{\r I}
\newcommand{\rJ}{\r J}

\newcommand{\rL}{\r L}

\newcommand{\rO}{\r O}
\newcommand{\rP}{\r P}

\newcommand{\rR}{\r R}

\newcommand{\rT}{\r T}

\newcommand{\rV}{\r V}
\newcommand{\rW}{\r W}

\newcommand{\rZ}{\r Z}

\newcommand{\rd}{\r d}

\newcommand{\rh}{\r h}

\newcommand{\rs}{\r s}

\newcommand{\sF}{\s F}

\newcommand{\sO}{\s O}
\newcommand{\sP}{\s P}
\newcommand{\sQ}{\s Q}

\newcommand{\sT}{\s T}

\newcommand{\sfM}{\sf M}

\newcommand{\sfh}{\sf h}

\newcommand{\tc}{\mathtt{c}}

\newcommand{\ts}{\mathtt{s}}
\newcommand{\tu}{\mathtt{u}}

\newcommand{\pres}[2]{\prescript{#1}{}{#2}}
\newcommand{\ab}{\r{ab}}
\newcommand{\ac}{\r{ac}}

\newcommand{\cl}{\r{cl}}
\newcommand{\Cl}{\r{Cl}}
\newcommand{\cor}{\r{cor}}
\newcommand{\cusp}{\r{cusp}}

\newcommand{\dr}{\r{dR}}
\newcommand{\et}{\acute{\r{e}}\r{t}}

\newcommand{\loc}{\r{loc}}
\newcommand{\Mot}{\sf{Mot}}
\newcommand{\op}{\r{op}}
\newcommand{\rat}{\r{rat}}
\newcommand{\sing}{\r{sing}}
\newcommand{\sph}{\r{sph}}
\newcommand{\ssl}{\r{ss}}
\newcommand{\ssp}{\r{ssp}}
\newcommand{\symm}{\r{symm}}
\newcommand{\tor}{\r{tor}}
\newcommand{\unr}{\r{unr}}

\DeclareMathOperator{\AJ}{AJ}

\DeclareMathOperator{\As}{As}

\DeclareMathOperator{\CH}{CH}

\DeclareMathOperator{\Corr}{Corr}

\DeclareMathOperator{\diag}{diag}
\DeclareMathOperator{\disc}{disc}

\DeclareMathOperator{\End}{End}
\DeclareMathOperator{\Ext}{Ext}

\DeclareMathOperator{\Fr}{Fr}

\DeclareMathOperator{\Gal}{Gal}

\DeclareMathOperator{\GL}{GL}

\DeclareMathOperator{\Hom}{Hom}
\DeclareMathOperator{\IM}{Im}
\DeclareMathOperator{\ind}{ind}
\DeclareMathOperator{\Ind}{Ind}

\DeclareMathOperator{\Ker}{Ker}
\DeclareMathOperator{\Lie}{Lie}

\DeclareMathOperator{\Mat}{Mat}

\DeclareMathOperator{\Nm}{Nm}

\DeclareMathOperator{\ord}{ord}
\DeclareMathOperator{\PGL}{PGL}
\DeclareMathOperator{\Pic}{Pic}
\DeclareMathOperator{\RE}{Re}
\DeclareMathOperator{\Res}{Res}

\DeclareMathOperator{\Sh}{Sh}
\DeclareMathOperator{\SL}{SL}
\DeclareMathOperator{\SO}{SO}

\DeclareMathOperator{\Spec}{Spec}

\DeclareMathOperator{\Sym}{Sym}
\DeclareMathOperator{\Tr}{Tr}
\DeclareMathOperator{\tr}{tr}

\begin{document}

\title{Hirzebruch--Zagier cycles and twisted triple product Selmer groups}

\author{Yifeng Liu}
\address{Department of Mathematics, Northwestern University, Evanston, IL 60208}
\email{liuyf@math.northwestern.edu}

\date{August 28, 2015}

\begin{abstract}
Let $E$ be an elliptic curve over $\dQ$ and $A$ be another elliptic curve
over a real quadratic number field. We construct a $\dQ$-motive of rank $8$,
together with a distinguished class in the associated Bloch--Kato Selmer
group, using Hirzebruch--Zagier cycles, that is, graphs of Hirzebruch--Zagier
morphisms. We show that, under certain assumptions on $E$ and $A$, the
non-vanishing of the central critical value of the (twisted) triple product
$L$-function attached to $(E,A)$ implies that the dimension of the associated
Bloch--Kato Selmer group of the motive is $0$; and the non-vanishing of the
distinguished class implies that the dimension of the associated Bloch--Kato
Selmer group of the motive is $1$. This can be viewed as the triple product
version of Kolyvagin's work on bounding Selmer groups of a single elliptic
curve using Heegner points.
\end{abstract}

\maketitle

\tableofcontents

\section{Introduction}
\label{s1}

Let $E$ be an elliptic curve over $\dQ$ and $A$ be another elliptic curve
over a real quadratic number field. We construct a $\dQ$-motive of rank $8$,
together with a distinguished class in the associated Bloch--Kato Selmer
group, using Hirzebruch--Zagier cycles, that is, graphs of Hirzebruch--Zagier
morphisms. We show that, under certain assumptions on $E$ and $A$, the
non-vanishing of the central critical value of the (twisted) triple product
$L$-function attached to $(E,A)$ implies that the dimension of the associated
Bloch--Kato Selmer group of $\sfM_{E,A}$ is $0$; and the non-vanishing of the
distinguished class implies that the dimension of the associated Bloch--Kato
Selmer group of $\sfM_{E,A}$ is $1$. These results are consequences of the
Beilinson--Bloch conjecture, the Bloch--Kato conjecture, and the
(conjectural) triple product version of the Gross--Zagier formula, connected
by the $L$-function as a bridge. Our results can be viewed as the triple
product version of the famous work of Kolyvagin \cite{Kol90} on bounding
Selmer groups of a single elliptic curve using Heegner points. In the
framework of Gan--Gross--Prasad conjecture, the motive in Kolyvagin's case
comes from automorphic representations of the pair $\SO(2)\times\SO(3)$,
while ours is from the pair $\SO(3)\times\SO(4)$ in which the group $\SO(4)$
is quasi-split but not split.

Our results certainly provide new evidence toward above-mentioned
conjectures, which are among central problems in number theory. They provide
the first example in the central critical case with both large Hodge--Tate
weights (of range $>1$) and large rank ($>4$).

\subsection{Main results}
\label{ss:main_results}

Denote by $\dQ^\ac$ the algebraic closure of $\dQ$ in the field $\dC$ of
complex numbers, $\dA$ the ring of ad\`{e}le of $\dQ$, and
$\Gamma_\dQ=\Gal(\dQ^\ac/\dQ)$ the absolute Galois group of $\dQ$. We fix
throughout the whole article a \emph{real} quadratic field $F\subset\dQ^\ac$
(except in the proof of Corollary \ref{co:symmetric}), with the ring of
integers $O_F$, and the absolute Galois group
$\Gamma_F\colonequals\Gal(\dQ^\ac/F)\subset\Gamma_\dQ$. Denote by $\theta$
the non-trivial involution in $\Gal(F/\dQ)$ or its local avatar.

\subsubsection{Motives and $L$-functions}

We consider an elliptic curve $E$ over $\dQ$ of conductor $N$. For each prime
$p$, denote by $\rT_p(E)$ the $p$-adic Tate module
$\varprojlim_nE[p^n](\dQ^\ac)$  of $E$, which is equipped with a
$\dZ_p$-linear continuous action of $\Gamma_\dQ$. Put
$\rV_p(E)=\rT_p(E)\otimes_{\dZ_p}\dQ_p$. Let $\sigma$ be the irreducible
cuspidal automorphic representation of $\GL_2(\dA)$ associated to $E$, whose
existence is guaranteed by the famous works on the modularity of rational
elliptic curves in \cites{Wil95,TW95,BCDT01}.

We consider another elliptic curve $A$ over $F$ whose conductor has norm $M$.
Put $A^\theta=A\otimes_{F,\theta}F$. For each prime $p$, we similarly define
$\rV_p(A)$, which is a $p$-adic representation of $\Gamma_F$. The tensor
product representation $\rV_p(A)\otimes_{\dQ_p}\rV_p(A^\theta)$ has a natural
extension to the larger group $\Gamma_\dQ$, known as the Asai representation,
denoted by $\As\rV_p(A)$. Let $\pi$ be the irreducible cuspidal automorphic
representation of $\Res_{F/\dQ}\GL_2(\dA)$ associated to $A$, whose existence
is guaranteed by the very recent works of Le Hung \cite{LH14} and Freitas--Le
Hung--Siksek \cite{FLS14} on the modularity of elliptic curves over real
quadratic fields.

Similarly to the Galois representation, the $F$-motive
$\sfh^1(A)\otimes\sfh^1(A^\theta)$ has a natural descent to a $\dQ$-motive,
known as the Asai motive, denoted by $\As\sfh^1(A)$, whose $p$-adic
realization is the Galois representation $\As(\rV_p(A)(-1))$. We define the
$\dQ$-motive
\[\sfM_{E,A}=\sfh^1(E)(1)\otimes\As\sfh^1(A)(1).\]
In particular, its $p$-adic realization $(\sfM_{E,A})_p$ is
$\rV_p(E)\otimes_{\dQ_p}(\As\rV_p(A))(-1)$ whose weight is $-1$. The motive
$\sfM_{E,A}$ is equipped with a canonical polarization
$\sfM_{E,A}\xrightarrow\sim\sfM_{E,A}^\vee(1)$ from the Poincar\'{e} duality
pairing $\sfM_{E,A}\times\sfM_{E,A}\to\dQ(1)$. The polarization is of
symplectic type.

The $L$-function $L(s,\sfM_{E,A})$ associated to the motive $\sfM_{E,A}$
satisfies the relation
\begin{align}\label{eq:l_function}
L(s,\sfM_{E,A})=L(s+1/2,\sigma\times\pi),
\end{align}
where the latter is the (twisted) triple product $L$-function, which is the
same as the tensor product $L$-function of $\sigma\times\As\pi$. By the work
of Garrett \cite{Gar87} and Piatetski-Shapiro--Rallis \cite{PSR87}, it has a
meromorphic continuation to the entire complex plane and is holomorphic at
$s=0$, satisfying the following functional equation
\[L(s,\sfM_{E,A})=\epsilon(\sfM_{E,A}) c(\sfM_{E,A})^{-s} L(-s,\sfM_{E,A})\] for the root number
$\epsilon(\sfM_{E,A})\in\{\pm1\}$ and some positive integer $c(\sfM_{E,A})$.

Consider the following assumptions.

\begin{assumption}[Group \textbf{E}]\label{as:group_e}
We assume that
\begin{description}
  \item[(E1)] Neither $E$ nor $A$ has complex multiplication over
      $\dQ^\ac$;

  \item[(E2)] $N$ and $M\disc(F)$ are coprime;

  \item[(E3)] If we write $N=N^+N^-$ such that $N^+$ (resp.\ $N^-$) is
      the product of prime factors that are split (resp.\ inert) in $F$,
      then $N^-$ is square-free.
\end{description}
\end{assumption}

\begin{definition}\label{de:parity}
Denote by $\wp(N^-)\geq 0$ the number of distinct prime factors of $N^-$. We
say that the pair $(E,A)$ is \emph{of even (resp.\ odd) type} if Assumption
\ref{as:group_e} is satisfied and $\wp(N^-)$ is odd (resp.\ even).
\end{definition}

By the result of Prasad \cite{Pra92}*{Theorems B \& D, Remark 4.1.1}, we have
$\epsilon(\sfM_{E,A})=(-1)^{\wp(N^-)+1}$ for $(E,A)$ satisfying Assumption
\ref{as:group_e}. In particular, if $(E,A)$ is of odd type, then we have
$L(0,\sfM_{E,A})=0$.

\begin{theorem}\label{th:main_even}
Suppose that $(E,A)$ satisfy Assumption \ref{as:group_e}. If
$L(0,\sfM_{E,A})\neq 0$, which in particular implies that $(E,A)$ is of even
type, then for all but finitely many primes $p$, we have
\[\dim_{\dQ_p}\rH^1_f(\dQ,(\sfM_{E,A})_p)=0.\]
\end{theorem}

From the above theorem, we can deduce the following one, whose statement has
nothing to do with the quadratic field $F$.

\begin{theorem}\label{th:main_sym}
Let $E_1$ and $E_2$ be two rational elliptic curves of conductors $N_1$ and
$N_2$, respectively. Suppose that $N_1$ and $N_2$ are coprime; $E_1$ has
multiplicative reduction at at least one finite place; and $E_2$ has no
complex multiplication over $\dQ^\ac$. If the central critical value
$L(2,E_1\times\Sym^2 E_2)\neq0$, then for all but finitely many primes $p$,
we have
\[\dim_{\dQ_p}\rH^1_f(\dQ,\rV_p(E_1)\otimes_{\dQ_p}(\Sym^2\rV_p(E_2))(-1))=0.\]
\end{theorem}

We will give the proof in \Sec \ref{ss:base_change}, actually for its
equivalent form Corollary \ref{co:symmetric}.

\subsubsection{A candidate in the Selmer group}

Let $X=X_{N^+M}$ be the minimal resolution of the Baily--Borel
compactification of the Hilbert modular surface associated to $F$ of
level-$N^+M$ structure, which is a smooth projective surface over $\dQ$ (see
\Sec \ref{ss:hilbert_modular} for details). In this article, our level
structures are always of $\Gamma_0$-type in the classical sense, unless
otherwise specified.

Suppose that $(E,A)$ is of odd type. Let $B_{N^-}$ be the indefinite
quaternion algebra over $\dQ$, ramified exactly at primes dividing $N^-$,
which exists since $\wp(N^-)$ is even. Denote by $Y=Y_{N^+M,N^-}$ the
(compactified) Shimura curve associated to $B_{N^-}$ of level-$N^+M$
structure. By modular interpretation, we have a natural finite morphism
$\zeta\colon Y\to X$ (see \Sec \ref{ss:special_morphisms}), which was first
studied by Hirzebruch--Zagier in this context. Put $Z=Y\times_{\Spec\dQ}X$.
Then the graph of $\zeta$ defines a Chow cycle $\Delta$ in $\CH^2(Z)$, which
we call a \emph{Hirzebruch--Zagier cycle}. This cycle is analogous to the
Gross--Kudla--Schoen cycle (see \cite{GS95} and \cite{GK92}) on triple
product of (modular) curves, and is an example of special cycles on Shimura
varieties in general.

We denote by $\rH^1_\sigma(Y_{\dQ^\ac})$ (resp.\ $\rH^2_\pi(X_{\dQ^\ac})$)
the $\sigma$-isotypic (resp.\ $\pi$-isotypic) subspace of
$\rH^1_{\et}(Y_{\dQ^\ac},\dQ_p(1))$ (resp.\
$\rH^2_{\et}(X_{\dQ^\ac},\dQ_p(1))$) under the Hecke action. Then it is a
direct summand as a $p$-adic representation of $\Gamma_\dQ$. Choose a Hecke
projector $\sP_{\sigma,\pi}\in\CH^3(Z\times_{\Spec\dQ}Z)\otimes_\dZ\dQ$ for
$(\sigma,\pi)$, viewed as a Chow correspondence (see \Sec
\ref{sss:construction_gross} for details). We have the Chow cycle
\[\Delta_{E,A}\colonequals\sP_{\sigma,\pi}^*\Delta\in\CH^2(Z)\otimes_\dZ\dQ,\]
which has the trivial image under the geometric cycle class map
$\cl_Z^0\colon\CH^2(Z)\otimes_\dZ\dQ\to\rH^4_{\et}(Z_{\dQ^\ac},\dQ_p(2))$.
Therefore, it induces an element
\[\pres{p}{\Delta}_{E,A}\colonequals\AJ_p(\Delta_{E,A})
\in\rH^1(\dQ,\rH^3_{\et}(Z_{\dQ^\ac},\dQ_p(2)))\] via the ($p$-adic)
\'{e}tale Abel--Jacobi map $\AJ_p$. As we will see in Proposition
\ref{pr:belong_selmer}, the element $\pres{p}{\Delta}_{E,A}$ does not depend
on the choice of $\sP_{\sigma,\pi}$, and moreover belongs to the
\emph{Bloch--Kato Selmer group}
$\rH^1_f(\dQ,\rH^1_\sigma(Y_{\dQ^\ac})\otimes_{\dQ_p}\rH^2_\pi(X_{\dQ^\ac}))$;
see \Sec \ref{ss:bloch_kato} for the general definition. Note that by
Conjecture \ref{co:abel_injective}, even the Chow cycle $\Delta_{E,A}$ itself
is independent of the choice of $\sP_{\sigma,\pi}$.

It is well-known that as $p$-adic representations of $\Gamma_\dQ$, we have
$\rH^1_\sigma(Y_{\dQ^\ac})\simeq\rV_p(E)^{\oplus b}$ for some $b\geq 1$. We
will show that the $p$-adic representation $\rH^2_\pi(X_{\dQ^\ac})$ of
$\Gamma_\dQ$ is isomorphic to $(\As\rV_p(A))(-1)^{\oplus a}$ for some $a\geq
1$. In particular, the class $\pres{p}{\Delta}_{E,A}$ may be regarded as an
element in
\[\rH^1_f(\dQ,\rV_p(E)\otimes_{\dQ_p}(\As\rV_p(A))(-1))^{\oplus ab}
=\rH^1_f(\dQ,(\sfM_{E,A})_p)^{\oplus ab}.\]

\begin{theorem}\label{th:main_odd}
Suppose that $(E,A)$ is of odd type. For all but finitely many primes $p$, if
$\pres{p}{\Delta}_{E,A}\neq0$, then
\[\dim_{\dQ_p}\rH^1_f(\dQ,(\sfM_{E,A})_p)=1.\]
\end{theorem}

The reader may consult \Sec \ref{sss:ggp} to see why we consider the Chow
cycle $\Delta_{E,A}$.

\begin{remark}\leavevmode
\begin{enumerate}
  \item In Theorems \ref{th:main_even}, \ref{th:main_sym} and
      \ref{th:main_odd}, we have explicit conditions on $p$ for which
      those theorems hold; see Definition \ref{de:group_p}.

  \item The Galois representation $\As\rV_p(A)$, and hence
      $(\sfM_{E,A})_p$, are decomposable if and only if $A^\theta$ is
      isogenous to a (possibly trivial) quadratic twist of $A$. We will
      take a closer look at a special case where $A^\theta$ is isogenous
      to $A$ in \Sec \ref{ss:base_change}.

  \item Assumption \ref{as:group_e} (E2, E3) are not essentially
      necessary. For each place $v$ of $\dQ$, denote by
      $\epsilon(1/2,\sigma_v\times\pi_v)\in\{\pm1\}$ the (twisted) triple
      product $\epsilon$-factor. Put
      \[\Sigma_{\sigma,\pi}=\{v<\infty\res\epsilon(1/2,\sigma_v\times\pi_v)\eta_v(-1)=-1\},\]
      where $\eta_v\colon\dQ_v^\times\to\{\pm1\}$ is the character
      associated to $F_v/\dQ_v$. Our method should be applicable under
      the only requirement that $v\|N$ if $v\in\Sigma_{\sigma,\pi}$
      (which already implies $v\mid N$), and moreover be able to prove
      Theorems \ref{th:main_even}, \ref{th:main_sym} and
      \ref{th:main_odd} for \emph{all} primes $p$. But these may
      tremendously increase the technical complication of the proof, and
      hence we decide not to proceed with that generality.
\end{enumerate}
\end{remark}

%\begin{remark}
%Note that in the triple product setup, we have a, say totally real, cubic
%\'{e}tale algebra $\dF$ over $\dQ$. There are three cases: (1) \emph{split
%case}, that is, $\dF=\dQ\oplus\dQ\oplus\dQ$; (2) \emph{quasi-split case},
%that is, $\dF=\dQ\oplus F$ with $F$ a real quadratic field; (3) \emph{cubic
%case}, that is, $\dF$ itself is a field.

%In this article, we deal with the quasi-split case. For the other two cases,
%at this moment we only know how to treat the rank-$0$ part (that is, in the
%flavor of Theorem \ref{th:main_even}), and new techniques are required. In
%the work in progress \cite{Liu}, we will prove the analogue of Theorem
%\ref{th:main_even} in the cubic case, under certain assumptions. As one can
%see from that article, the techniques developed for the cubic case are
%sufficient for the split case.
%\end{remark}

\begin{remark}
The idea of exploiting diagonal special cycles in Euler System-style
arguments originated in the work of Darmon and Rotger \cite{DR1,DR2}, who
exploited them to obtain a result for the motive of an elliptic curve twisted
by two $2$-dimensional Artin representations, which is also of rank $8$. They
showed that the Mordell--Weil group (as a twisted version of the classical
Mordell--Weil group of an elliptic curve) is finite if the corresponding
$L$-function has non-vanishing central value. Both their work and ours
provide evidence of the Beilinson--Bloch conjecture and the Bloch--Kato
conjecture, for the pair $\SO(3)\times\SO(4)$ in the framework of
Gan--Gross--Prasad, both using diagonal special cycles. But there are some
differences:
\begin{itemize}
      \item They deal with the split case (in particular, the geometric
          object is a triple product of modular curves), while we deal
          with the quasi-split case.

      \item The (range of) Hodge--Tate weight of their motive is $1$,
          while ours is $3$ as the motive involves two elliptic curves.

      \item The method they use is $p$-adic, especially via Hida
          families, while ours is more geometric by exploiting ``tame''
          deformations at auxiliary primes $\ell$ not equal to $p$ in
          bounding the pro-$p$-Selmer group.

      \item More importantly, in the rank-$0$ case, our proof actually
          indicates that the $p^\infty$-part of the Tate--Shafarevich
          group of the motive, defined in a suitable way, is trivial for
          all but finitely many primes $p$; while their method is only
          able to control the image of the pro-$p$-Selmer group in the
          local cohomology at $p$, from which it seems very hard to
          deduce bounds on the relevant Tate--Shafarevich groups in
          settings other than Mordell--Weil groups.
\end{itemize}
\end{remark}

\subsection{Connection with $L$-function and Beilinson--Bloch--Kato conjecture}
\label{ss:connection_l}

For a number field $K$, denote by $\Mot_K^\rat$ the pseudo-abelian category
of Chow motives over $K$ with coefficient $\dQ$ (see \cite{Sch94} for
example). For a Chow motive $\sfM$, we have the $p$-adic realization
$\sfM_p$, which is a $p$-adic representation of $\Gal(K^\ac/K)$, for each
prime $p$. Recall that for each smooth projective scheme $X$ over $K$ purely
of dimension $d$, we have the Picard motive $\sfh^{2d-1}(X)$ and the Albanese
motive $\sfh^1(X)$. For each Chow motive $\sfM$, we have an associated
$L$-function $L(s,\sfM)$ which is defined for $\RE s$ sufficiently large. A
polarization of $\sfM$ is an isomorphism $\sfM\xrightarrow\sim\sfM^\vee(1)$
in the category $\Mot_K^\rat$.

\subsubsection{Beilinson--Bloch / Bloch--Kato conjecture}

\begin{definition}
Let $\sfM$ be a Chow motive. We define the \emph{Chow group} of $\sfM$ (of
degree $0$) to be
\[\CH(\sfM)=\Hom_{\Mot_K^\rat}(\mathbbm{1}_K,\sfM),\]
where $\mathbbm{1}_K$ is the unit motive over $K$.
\end{definition}

Denote by $\CH(\sfM)^0$ the kernel of the geometric cycle class map
\[\cl_{\sfM}^0\colon\CH(\sfM)\to\rH^0(K,\sfM_p)\]
for some prime $p$, which is known to be independent of $p$. We have the
following \emph{($p$-adic) \'{e}tale Abel--Jacobi map}
\[\AJ_p\colon\CH(\sfM)^0\to\rH^1(K,\sfM_p)\]
for each prime $p$. The following conjecture is a part of the
Beilinson--Bloch conjecture; see for example \cite{Jan90}*{9.15}.

\begin{conjecture}\label{co:abel_injective}
Let notation be as above. The \'{e}tale Abel--Jacobi map $\AJ_p$ is injective
for every prime $p$.
\end{conjecture}

The following conjecture is a combination of the Beilinson--Bloch conjecture
and the Bloch--Kato conjecture.

\begin{conjecture}\label{co:bloch}
Let $\sfM$ be a polarized Chow motive in $\Mot_K^\rat$.
\begin{enumerate}
  \item The $\dQ_p$-linearization
        \[\AJ_p\colon\CH(\sfM)^0\otimes_\dQ\dQ_p\to\rH^1(K,\sfM_p)\]
        is an isomorphism onto the Bloch--Kato Selmer group
        $\rH^1_f(K,\sfM_p)$, which is a subspace of $\rH^1(K,\sfM_p)$,
        for all $p$.

  \item The $L$-function $L(s,\sfM)$ associated to $\sfM$ has a
      meromorphic continuation to the entire complex plane, satisfying
      the following functional equation
      \[L(s,\sfM)=\epsilon(\sfM) c(\sfM)^{-s} L(-s,\sfM)\] for the root number
      $\epsilon(\sfM)\in\{\pm1\}$ and some positive integer $c(\sfM)$.

  \item For \emph{all} $p$, we have
      \[\ord_{s=0}L(s,\sfM)=\dim_{\dQ_p}\rH^1_f(K,\sfM_p)-\dim_{\dQ_p}\rH^0(K,\sfM_p).\]
\end{enumerate}
\end{conjecture}

Recall from \Sec \ref{ss:main_results} that we have the surface $X$, and when
$(E,A)$ is of odd type, the curve $Y$ and hence the three-fold
$Z=Y\times_{\Spec\dQ}X$, all being smooth projective over $\dQ$. We have the
Chow motive $\sfh^2(X)$, defined uniquely up to isomorphism, by Murre
\cite{Mur90}. The motive $\sfh^1(Y)(1)\otimes\sfh^2(X)(1)$ is canonically
polarized by the Poincar\'{e} duality.

\begin{conjecture}\label{co:oda}
The motive $\As\sfh^1(A)$ is a direct summand of $\sfh^2(X)$ in the
pseudo-abelian category $\Mot_\dQ^\rat$.
\end{conjecture}

Note that for our motive $\sfM_{E,A}$, we have $\rH^0(\dQ,(\sfM_{E,A})_p)=0$
for all $p$ and hence $\CH(\sfM_{E,A})^0=\CH(\sfM_{E,A})$. The following is a
consequence of \eqref{eq:l_function}, Theorems \ref{th:main_even} and
\ref{th:main_odd}.

\begin{corollary}\label{co:main}
Let $(E,A)$ be a pair satisfying Assumption \ref{as:group_e}.
\begin{enumerate}
  \item Conjecture \ref{co:bloch} (2) holds for $\sfM_{E,A}$.

  \item Suppose that $\ord_{s=0}L(s,\sfM_{E,A})=0$. Then Conjecture
      \ref{co:bloch} (3) holds for $\sfM_{E,A}$ and all but finitely many
      $p$.

  \item Suppose that the pair $(E,A)$ is of odd type. We further assume
      Conjecture \ref{co:abel_injective} for
      $\sfh^1(Y)(1)\otimes\sfh^2(X)(1)$, and Conjecture \ref{co:oda}. If
      $\Delta_{E,A}\neq 0$, then Conjecture \ref{co:bloch} (1) holds for
      $\sfM_{E,A}$ and all but finitely many primes $p$.
\end{enumerate}
\end{corollary}

\begin{remark}
Since currently we do not know Conjecture \ref{co:abel_injective} for
$\sfh^1(Y)(1)\otimes\sfh^2(X)(1)$, it is not clear that
$\pres{p}{\Delta}_{E,A}=\AJ_p(\Delta_{E,A})\neq 0$ for one prime $p$ implies
the same for others. We also do not even know if
$\dim_{\dQ_p}\rH^1_f(\dQ,(\sfM_{E,A})_p)$ are same for all (but finitely
many) $p$. However, see Corollary \ref{co:base_change}.
\end{remark}

\subsubsection{Arithmetic Gan--Gross--Prasad conjecture}
\label{sss:ggp}

The Chow cycle $\Delta_{E,A}$ is closely related to the Gross--Prasad
conjecture \cite{GP92}. In fact, the product $\sigma\times\pi$ can be viewed
as a cuspidal automorphic representation of $H\times G$, where
$H\simeq\SO(3)$ is the split special orthogonal group over $\dQ$ of rank $3$
and $G\simeq\SO(4)$ is the quasi-split special orthogonal group over $\dQ$ of
rank $4$ and with discriminant field $F$. The global (Gan--)Gross--Prasad
conjecture \cite{GGP12} predicts that the central value
$L(1/2,\sigma\times\pi)$ is non-vanishing if and only if there is a
non-trivial diagonal automorphic period integral for some member
$\sigma'\times\pi'$, as an automorphic representation of some pure inner form
$H'\times G'$ of $H\times G$, in the Vogan packet of $\sigma\times\pi$. If
such $\sigma'\times\pi'$ exists, then it is unique, and the group $H'\times
G'$ is determined by all local epsilon factors
$\epsilon(1/2,\sigma_v\times\pi_v)\in\{\pm1\}$. Denote by
$\Sigma_{\sigma,\pi}$ the (finite) set of places $v$ of $\dQ$ such that
$\epsilon(1/2,\sigma_v\times\pi_v)\eta_v(-1)=-1$, where
$\eta_v\colon\dQ_v^\times\to\{\pm1\}$ is the character associated to
$F_v/\dQ_v$.

In our case, the set $\Sigma_{\sigma,\pi}$ consists of $\infty$ and all prime
factors of $N^-$, by \cite{Pra92}*{Theorems B \& D, Remark 4.1.1}. In
particular, it has even (resp.\ odd) cardinality if $(E,A)$ is of even
(resp.\ odd) type. Suppose that $(E,A)$ is of odd type, which implies that
the (global) root number
$\epsilon(\sfM_{E,A})=\epsilon(1/2,\sigma\times\pi)=-1$ and
$\ord_{s=0}L(s,\sfM_{E,A})$ is odd by \eqref{eq:l_function}. This is usually
referred as the \emph{arithmetic case}. As a consequence, the central
$L$-value $L(0,\sfM_{E,A})=L(1/2,\sigma\times\pi)$ vanishes. Moreover, all
diagonal automorphic period integrals for the entire Vogan packet of
$\sigma\times\pi$ vanish simply by the local epsilon dichotomy
\cite{Pra92}*{Theorems B \& D}. In the arithmetic case, we consider
(modified) diagonal cycles instead of diagonal period integrals. The
following conjecture can be regarded as an arithmetic version of the global
Gan--Gross--Prasad conjecture.

\begin{conjecture}[Arithmetic GGP for $\SO(3)\times\SO(4)$, unrefined version]
\label{co:arithmetic_ggp}

Suppose that $(E,A)$ is of odd type. Then the Chow cycle $\Delta_{E,A}\neq0$
if and only if $L'(0,\sfM_{E,A})\neq 0$.
\end{conjecture}

\begin{remark}
There is no local obstruction for $\Delta_{E,A}$ being non-zero. In fact, the
local invariant functionals \cite{Pra92} can be realized at the level
structure defining the Shimura varieties $Y$ and $X$, as told by Lemma
\ref{le:test_vector}.
\end{remark}

In the case where the group $G$ is the \emph{split} special orthogonal group
over $\dQ$ of rank $4$, we can construct a similar Chow motive $\sfM$ but
from three elliptic curves defined over $\dQ$. Then the above conjecture was
proposed by Gross and Kudla \cite{GK92}, in terms of the Beilinson--Bloch
height. Later, Gross and Keating \cite{GK93} did crucial local computation
toward this conjecture, actually applicable to both the split case and the
quasi-split case as in Conjecture \ref{co:arithmetic_ggp}. In a recent
preprint \cite{YZZ}, X.~Yuan, S.~Zhang and W.~Zhang made important progress
toward the global height formula in the split case under certain conditions.
On the other hand, in \cite{Zha12}, the author proposed the Arithmetic
Gan--Gross--Prasad conjecture for close pairs of unitary groups of arbitrary
ranks, together with the approach via relative trace formulae, and obtained
some local results as evidence toward the conjecture.

In view of the Arithmetic Gan--Gross--Prasad conjecture, Theorem
\ref{th:main_odd} has the following corollary.

\begin{corollary}
Let $(E,A)$ be of odd type such that $\ord_{s=0}L(s,\sfM_{E,A})=1$. If we
assume Conjecture \ref{co:abel_injective} for the motive
$\sfh^1(Y)(1)\otimes\sfh^2(X)(1)$, and Conjecture \ref{co:arithmetic_ggp},
then Conjecture \ref{co:bloch} (3) holds for $\sfM_{E,A}$ and all but
finitely many primes $p$.
\end{corollary}

In view of Conjectures \ref{co:bloch} and \ref{co:arithmetic_ggp}, we propose
the following.

\begin{conjecture}
Let $(E,A)$ be an arbitrary pair.
\begin{enumerate}
  \item Suppose that either one of the following holds
        \begin{itemize}
          \item $\dim_{\dQ_p}\rH^1_f(\dQ,(\sfM_{E,A})_p)=0$ for
              \emph{some} prime $p$;
          \item $\dim_\dQ\CH(\sfM_{E,A})=0$;
          \item $\ord_{s=0}L(s,\sfM_{E,A})=0$.
        \end{itemize}
        Then for \emph{all} primes $p$, we have
        \[\dim_\dQ\CH(\sfM_{E,A})=\ord_{s=0}L(s,\sfM_{E,A})=\dim_{\dQ_p}\rH^1_f(\dQ,(\sfM_{E,A})_p)=0.\]

  \item Suppose that either one of the following holds
        \begin{itemize}
          \item $\dim_{\dQ_p}\rH^1_f(\dQ,(\sfM_{E,A})_p)=1$ for
              \emph{some} $p$;
          \item $\dim_\dQ\CH(\sfM_{E,A})=1$;
          \item $\ord_{s=0}L(s,\sfM_{E,A})=1$.
        \end{itemize}
        Then $\Delta_{E,A}\neq 0$, and for \emph{all} primes $p$, we have
        \[\dim_\dQ\CH(\sfM_{E,A})=\ord_{s=0}L(s,\sfM_{E,A})=\dim_{\dQ_p}\rH^1_f(\dQ,(\sfM_{E,A})_p)=1.\]
        Here, the class $\Delta_{E,A}$ may be constructed via a similar
        way.
\end{enumerate}
\end{conjecture}

\subsection{The case of base change}
\label{ss:base_change}

Let us consider a special but illustrative case where $A$ is isogenous to
$A^\flat\otimes_\dQ F$ for an elliptic curve $A^\flat$ over $\dQ$. This
immediately implies that $A^\theta$ is isogenous to $A$, but the latter is
slightly weaker. In this case, we have a decomposition
\[(\As\rV_p(A))(-1)=(\Sym^2\rV_p(A^\flat))(-1)\oplus\dQ_p(\eta)\]
of $p$-adic representations of $\Gamma_\dQ$, where
$\eta\colon\Gamma_\dQ\to\{\pm1\}$ is the quadratic character associated to
$F/\dQ$ via global class field theory. Note that
$\rV_p(E)\otimes_{\dQ_p}\dQ_p(\eta)\simeq\rV_p(E^F)$, where $E^F$ is the
$F$-twist of $E$.

Suppose that $(E,A)$ is of odd type. Denote by $\pres{p}{\Delta}_{E,A^\flat}$
(resp.\ $\pres{p}{\Delta}_{E^F}$) the projection of $\pres{p}{\Delta}_{E,A}$
into $\rH^1_f(\dQ,\rV_p(E)\otimes_{\dQ_p}(\Sym^2\rV_p(A^\flat))(-1))^{\oplus
ab}$ (resp.\ $\rH^1_f(\dQ,\rV_p(E^F))^{\oplus ab}$). Then Theorem
\ref{th:main_odd} asserts that for all but finitely many primes $p$, there is
at most one member between $\pres{p}{\Delta}_{E,A^\flat}$ and
$\pres{p}{\Delta}_{E^F}$ that could be non-zero, and moreover
\begin{enumerate}
  \item if $\pres{p}{\Delta}_{E,A^\flat}\neq 0$, then
      \[\dim_{\dQ_p}\rH^1_f(\dQ,\rV_p(E)\otimes_{\dQ_p}(\Sym^2\rV_p(A^\flat))(-1))=1,\quad\dim_{\dQ_p}\rH^1_f(\dQ,\rV_p(E^F))=0;\]

  \item if $\pres{p}{\Delta}_{E^F}\neq 0$, then
      \[\dim_{\dQ_p}\rH^1_f(\dQ,\rV_p(E)\otimes_{\dQ_p}(\Sym^2\rV_p(A^\flat))(-1))=0,\quad\dim_{\dQ_p}\rH^1_f(\dQ,\rV_p(E^F))=1.\]
\end{enumerate}

Suppose that $\pres{p}{\Delta}_{E^F}\neq 0$. By the main theorem of
\cite{Ski14} or \cite{Zha14}, the rank of $E^F(\dQ)$ is $1$, or equivalently,
the rank of $E(F)^{\theta=-1}$ is $1$. A natural question would be: \emph{How
to construct a $\dQ$-generator of $E(F)^{\theta=-1}$, from the information
that $\pres{p}{\Delta}_{E^F}\neq 0$ ?}

Note that we have the geometric cycle class map
\[\cl_{X_F}^0\colon\CH^1(X_F)\otimes_\dZ\dQ\to\rH^2_{\et}(X_{\dQ^\ac},\dQ_p(1))^{\Gamma_F}.\]
By the Tate conjecture for the surface $X_F$ in the version proved in
\cite{HLR86}, we have
\[(\IM(\cl_{X_F}^0)\cap\rH^2_\pi(X_{\dQ^\ac}))\otimes_\dZ\dQ_p=\rH^2_\pi(X_{\dQ^\ac})^{\Gamma_F}\simeq\dQ_p^{\oplus a}.\]
Let $\CH^1_\pi(X_F)$ be the subspace of $\CH^1(X_F)\otimes_\dZ\dQ$ whose
image under $\cl_{X_F}^0$ is contained in $\rH^2_\pi(X_{\dQ^\ac})$. If we
regard $\Delta_{E,A}$ as a correspondence from $Y$ to $X$, then
$\phi_*\Delta_{E,A}^*\CH^1_\pi(X_F)$ is contained in
$(E(F)\otimes_\dZ\dQ)^{\theta=-1}$ for every modular parametrization
$\phi\colon Y\to E$. One can check that $\pres{p}{\Delta}_{E^F}\neq 0$
implies the existence of $\phi$ such that
$\phi_*\Delta_{E,A}^*\CH^1_\pi(X_F)$ is non-zero. Combining with Theorem
\ref{th:main_odd}, we have
\[\phi_*\Delta_{E,A}^*\CH^1_\pi(X_F)=(E(F)\otimes_\dZ\dQ)^{\theta=-1}
=E^F(\dQ)\otimes_\dZ\dQ,\] which has dimension $1$.

\begin{remark}
The above discussion is logically unhelpful for the proof of Theorem
\ref{th:main_odd}, since it is very special that $A$ comes from the base
change. Nevertheless, this observation, on the contrary, is a starting point
of our strategy for the proofs of \emph{both} Theorem \ref{th:main_odd} and
Theorem \ref{th:main_even}: The key theorems -- Theorems \ref{th:congruence}
and \ref{th:congruence_bis}, which are explicit congruence formulae for
various Hirzebruch--Zagier classes, involve computation of their
localizations at sophisticatedly chosen primes. We realize such computation
as an application of the Tate conjecture for Hilbert modular surfaces over
\emph{finite fields}, together with some geometric arguments that are very
close to the above pullback along correspondences.
\end{remark}

The following corollary is a natural consequence of Theorem \ref{th:main_odd}
and the above discussion, since we know Conjecture \ref{co:abel_injective}
for the motive $\sfh^1(E^F)(1)$.

\begin{corollary}\label{co:base_change}
Suppose that the pair $(E,A)$ is of odd type, and $A$ is isogenous to
$A^\flat\otimes_\dQ F$ for an elliptic curve $A^\flat$ over $\dQ$. If
$\pres{p}{\Delta}_{E^F}$ is non-zero for \emph{some} $p$, then
$\pres{p}{\Delta}_{E^F}$ and hence $\pres{p}\Delta_{E,A}$ are non-zero for
\emph{every} $p$. Moreover, we have
$\dim_{\dQ_p}\rH^1_f(\dQ,(\sfM_{E,A})_p)=1$ for \emph{all but finitely many}
$p$.
\end{corollary}

We also have the following result, which is a corollary of Theorem
\ref{th:main_even} and is equivalent to Theorem \ref{th:main_sym}. It
provides evidence of Conjecture \ref{co:bloch} for the motive
$\sfh^1(E_1)(1)\otimes(\Sym^2\sfh^1(E_2))(1)$ over $\dQ$ of rank $6$, which
is also canonically polarized of symplectic type.

\begin{corollary}\label{co:symmetric}
Let $E_1$ and $E_2$ be two rational elliptic curves of conductors $N_1$ and
$N_2$, respectively. Suppose that $N_1$ and $N_2$ are coprime; $E_1$ has
multiplicative reduction at at least one finite place; and $E_2$ has no
complex multiplication over $\dQ^\ac$. If the central critical value
$L(0,\sfh^1(E_1)(1)\otimes(\Sym^2\sfh^1(E_2))(1))\neq0$, then for all but
finitely many primes $p$, we have
\[\dim_{\dQ_p}\rH^1_f(\dQ,\rV_p(E_1)\otimes_{\dQ_p}(\Sym^2\rV_p(E_2))(-1))=0.\]
\end{corollary}

\begin{proof}
First, note that the assumptions in this corollary force $E_1$ to have root
number $-1$, that is, $\epsilon(\sfh(E_1)(1))=-1$. In fact, we have
\[\epsilon(\sfh(E_1)(1)\otimes\sfh(E_2)(1)\otimes\sfh(E_2)(1))=
\epsilon(\sfh(E_1)(1))\times\epsilon(\sfh^1(E_1)(1)\otimes(\Sym^2\sfh^1(E_2))(1)),\]
in which the first root number is $-1$ since $N_1$ and $N_2$ are coprime, and
the last root number is $1$ since
$L(0,\sfh^1(E_1)(1)\otimes(\Sym^2\sfh^1(E_2))(1))\neq0$.

Choose a prime $\ell\mid N_1$ such that $E_1$ has multiplicative reduction at
$\ell$. In particular, we have $\ell\| N_1$. Choose a quadratic Dirichlet
character $\eta_1$ of conductor coprime to $N_1$ such that $\eta_1(-1)=1$,
$\eta_1(\ell)=-1$, and $\eta_1(\ell')=1$ for all other prime factors $\ell'$
of $N_1$. Then the quadratically twisted elliptic curve $E_1^{\eta_1}$ has
root number $1$. By the main theorem of \cite{BFH90}, we may choose another
quadratic character $\eta_2$ of conductor coprime to $N_1$ such that
$\eta_2(-1)=1$, $\eta_2(\ell')=1$ for all prime factors $\ell'$ of $N_1$
(including $\ell$), and moreover $L(0,\sfh(E_1^\eta)(1))\neq 0$ where
$\eta=\eta_1\eta_2$. Now let $F$ be the (real) quadratic field determined by
$\eta$. Applying Theorem \ref{th:main_even} with $E=E_1$ and
$A=E_2\otimes_\dQ F$, we have that for all but finitely many primes $p$,
\[\dim_{\dQ_p}\rH^1_f(\dQ,\rV_p(E_1)\otimes_{\dQ_p}(\Sym^2\rV_p(E_2))(-1))+
\dim_{\dQ_p}\rH^1_f(\dQ,\rV_p(E_1^\eta))=0,\] which implies the corollary.
\end{proof}

\subsection{Analogy with Kolyvagin's work}
\label{ss:analogy_kolyvagin}

The current article is mostly influenced by the pioneer work of Kolyvagin. We
explain this in the case where $(E,A)$ is of odd type.

Let $E$ be an elliptic curve over $\dQ$ of conductor $N$ with the associated
cuspidal automorphic representation $\sigma$ of $\GL_2(\dA)$. Let
$K\subset\dQ^\ac$ be an imaginary quadratic field satisfying the (strong)
Heegner condition for $E$: every prime factor of $N$ is split in $K$. Denote
by $Y_N$ the compactified modular curve over $\dQ$ of $\Gamma_0(N)$ level
structure. Pick up a Heegner point $P_K$ with respect to the maximal order of
$K$ on $Y_N(\dC)$, which in fact belongs to $Y_N(H_K)$ where $H_K$ is the
Hilbert class field of $K$. Take a modular parametrization $\phi\colon Y_N\to
E$ that sends the cusp to $0_E$. Define
\[P_{K,E}=\Tr_{H_K/K}(\phi(P_K)-0_E)\in\CH^1(E_K)^0=\CH(\sfh^1(E_K)(1)).\]
In \cite{Kol90}, Kolyvagin proved that if $P_{K,E}\neq 0$, then
$\dim_{\dQ_p}\rH^1_f(K,\sfh^1(E_{K^\ac})(1)_p)=1$ for all primes $p$. By the
Gross--Zagier formula \cite{GZ86}, we have that $P_{K,E}\neq 0$ if and only
if $L'(1/2,\sigma_K)\neq 0$. Note that
$L(s,\sfh^1(E_K)(1))=L(s+1/2,\sigma_K)$, with root number
$\epsilon(\sfh^1(E_K)(1))=-1$.

\begin{remark}
Although we use the same notation for $E$, the reader should not think it is
the same thing as in the pair $(E,A)$ previously. In fact, the role $E$ plays
in Kolyvagin's case is more or less the role of the other elliptic curve $A$
in our case.
\end{remark}

We summarize the analogy between Kolyvagin's setup and ours in Table
\ref{ta:analogy}.

\begin{table}[h]
\centering
\begin{tabular}{c||c|c}
  \hline
  & Kolyvagin's case & our case \\ \hline
  motive & $\sfh^1(E_K)(1)$ & $\sfM_{E,A}$ \\ \hline
  Hodge--Tate weights & $\{-1,0\}$ & $\{-2,-1^3,0^3,1\}$  \\ \hline
  automorphic $L$-function & $L(s,\sigma_K)$  & $L(s,\sigma\times\pi)$ \\ \hline
  special cycle & $P_{K,E}$  & $\Delta_{E,A}$ \\ \hline
  GGP pair $H\times G$ & $\SO(2)\times\SO(3)$ & $\SO(3)\times\SO(4)$ \\ \hline
  ambient variety $Y\times X$ & CM $0$-fold $\times$ modular curve &  Shimura curve $\times$ HMS  \\\hline
\end{tabular}
\caption{Analogy} \label{ta:analogy}
\end{table}

In fact, the key idea of our proof is similar to Kolyvagin's as well. Namely,
we use variants of $\Delta_{E,A}$ ($P_{K,E}$ in Kolyvagin's case) to produce
sufficiently many annihilators, under the Tate pairing, of the Selmer group
of the Galois cohomology with torsion coefficients. For $\sfM$ equal to
$\sfh^1(E_K)(1)$ or $\sfM_{E,A}$, we have a natural Galois stable lattice
$\bar\sfM_p^\infty$ of $\sfM_p$. For an integer $n\geq 1$, the ``free rank''
$r_n$ of the torsion Selmer group $\rH^1_f(K,\bar\sfM_p^n)$, where
$\bar\sfM_p^n=\bar\sfM_p^\infty\otimes_{\dZ_p}\dZ/p^n$, will be stabilized to
$r_\infty$ as $n\to\infty$. Under both situations, the goal is to show that
$r_\infty=1$.

In Kolyvagin's case, the cycle $P_{K,E}$ is constructed via the diagonal
embedding $Y\to Y\times X$ where $Y$ (resp.\ $X$) is the union of CM points
(resp.\ modular curve), which is also the Shimura $0$-fold (resp.\ $1$-fold)
attached to the group $H=\SO(2)$ defined by $K$ (resp.\
$G=\SO(3)=\PGL_{2,\dQ}$). Here, the scheme $Y$ maps naturally to $X$ through
an optimal embedding $H\hookrightarrow G$ of underlying groups. Besides
$P_{K,E}$, he defined other classes $\bar{P}_{K,E\res S}^n$ with torsion
coefficient $\dZ/p^n$ for $S$ a finite set of so-called \emph{Kolyvagin
primes} (with respect to $p^n$), by modifying the embedding $H\hookrightarrow
G$ and hence $Y\to X$. These variant classes $\bar{P}_{K,E\res S}^n$ are
closely related to $P_{K,E}$. In fact, once $P_{K,E}\neq 0$, they will
produce sufficiently many annihilators for the Selmer group
$\rH^1_f(K,\bar\sfM_p^n)$. In fact, it is enough to consider those classes
with $\# S\leq 2$.

In our case, the cycle $\Delta_{E,A}$ is constructed via the diagonal
embedding $Y\to Y\times X$ where $Y$ (resp.\ $X$) is a Shimura curve (resp.\
Hilbert modular surface), which is also the Shimura $1$-fold (resp.\
$2$-fold) attached to the group $H=\SO(3)$ defined by the quaternion algebra
$B_{N^-}$ (resp.\ $G=\SO(4)\subset\Res_{F/\dQ}\GL_{2,F}/\bG_{m,\dQ}$).
Besides $\Delta_{E,A}$, we also define other classes
$\bar{\Delta}_{E,A\res\ell_1,\ell_2}^n$ with torsion coefficient $\dZ/p^n$
for $\ell_1,\ell_2$ being two distinct \emph{strongly $n$-admissible primes},
as we call. For doing this, instead of changing the embedding
$H\hookrightarrow G$ as in Kolyvagin's case, we change the algebra $B_{N^-}$
and hence the group $H$. Such idea first appeared in the work of
Bertolini--Darmon \cite{BD05}. The most difficult part is to relate these
variant classes $\bar{\Delta}_{E,A\res\ell_1,\ell_2}^n$ to the original one
in an explicit way. Once this is achieved and $\pres{p}{\Delta}_{E,A}\neq 0$,
they will produce sufficiently many annihilators for the Selmer group
$\rH^1_f(K,\bar\sfM_p^n)$, by some sophisticated Galois-theoretical
arguments.

\subsection{Strategy of proof}
\label{ss:strategy_proof}

We explain in more details about our strategy of the proof, and take the
chance to introduce the main structure of the article. The proofs for Theorem
\ref{th:main_even} and Theorem \ref{th:main_odd} will be carried out
simultaneously. To avoid confusion, here we will restrict ourselves to the
case where $(E,A)$ is of odd type. In fact, the ingredients we need for the
proof of Theorem \ref{th:main_odd} strictly contain those of Theorem
\ref{th:main_even}.

In \Sec \ref{s2}, we introduce and study some properties of our fundamental
geometric objects, namely, Shimura curves and the Hilbert modular surface. We
give \emph{canonical} parametrization of the following objects:
\begin{itemize}
  \item supersingular locus of a Shimura curve at a good prime,
  \item reduction graph of a Shimura curve in the \v{C}erednik--Drinfeld
      reduction,
  \item reduction graph of the supersingular locus of a Hilbert modular
      surface at a good inert prime.
\end{itemize}
Note that the first two are known by the work of Ribet \cite{Rib89}. We will
use a similar method to study the third one. Here, the canonical
parametrization is realized by oriented orders in certain definite quaternion
algebras, \emph{without} picking up a base point. More importantly, we will
study the behavior of canonical parametrization under special morphisms from
different Shimura curves to the Hilbert modular surface (which is fixed in
application) at \emph{different} primes. The canonicality of the
parametrization is crucial for doing this (see Propositions
\ref{pr:special_supersingular} and \ref{pr:special_superspecial}). All these
considerations will be used in the computation of localization maps later.

We start \Sec \ref{s3} by recalling the definition of Bloch--Kato Selmer
groups in various setting, and proving some preparatory results. Then we
switch the language from elliptic curves to automorphic representations,
followed by the construction of the cycle $\Delta_{E,A}$, or rather
$\Delta_{\sigma,\pi}$ in the automorphic setting. After recalling certain
results from \cite{BD05} concerning level raising of torsion Galois modules
for $\GL_{2,\dQ}$, we introduce those classes
$\bar{\Delta}_{\sigma,\pi\res\ell_1,\ell_2}^n$. The most technical part is to
choose appropriate primes, which we call \emph{strongly admissible primes},
in the sense that there should be enough of them for those classes to serve
as annihilators, and meanwhile they should be special enough to make the
computation of localization possible.

The next chapter \Sec \ref{s4} is the technical core of the article, where we
formulate and prove the explicit congruence formulae in Theorems
\ref{th:congruence}, for $\Delta_{\sigma,\pi}$ and various
$\bar{\Delta}_{\sigma,\pi\res\ell_1,\ell_2}^n$ via localization. The
computation relies on the existence of integral Tate cycles on special fibers
of the Hilbert modular surface at good inert primes, together with certain
specialization arguments from algebraic geometry. In fact, the localization
of both $\Delta_{\sigma,\pi}$ and
$\bar{\Delta}_{\sigma,\pi\res\ell_1,\ell_2}^n$ can be realized as period
integrals of Gan--Gross--Prasad type, of some torsion algebraic automorphic
forms for certain pair $\SO(3)\times\SO(4)$ that are compact at infinity.

We arrive at the final stage of the proof in \Sec \ref{s5}. The remaining
argument is a combination of Theorem \ref{th:congruence}, Tate pairings,
Serre's theorems on large Galois image for elliptic curves, and some
Galois-theoretical techniques, in particular, the heavy use of the Chebotarev
Density Theorem. The final part of the proof is much more involved than
Kolyvagin's case, essentially due to the absence of Kummer maps and the
inapplicability of Mordell--Weil group (Chow group).

\subsection{Notation and conventions}

\subsubsection{Generalities}

\begin{itemize}
  \item The Greek letter $\epsilon$ always represents a sign, either $+$
      or $-$.

  \item All rings and algebras are unital and homomorphisms preserve
      units. For a \emph{ring}, it is always assumed to be commutative.

  \item Throughout the article, a \emph{prime} means a rational prime
      number; a \emph{prime power} means a positive integer power of a
      prime. Starting from the next chapter, the prime $p$ is always
      \emph{odd}.

  \item For an integer $N$, denote by $\wp(N)\geq 0$ the number of its
      distinct prime factors.

  \item For a ring $R$ of prime characteristic, denote by $\rW(R)$ the
      ring of infinite Witt vectors with values in $R$.

  \item For every prime $\ell$, we fix an algebraic closure
      $\dF_\ell^\ac$ of $\dF_\ell$. If $\ell^f$ is a prime power, then we
      denote by $\dF_{\ell^f}$ the unique subfield of $\dF_\ell^\ac$ of
      cardinality $\ell^f$. If $v$ is a prime power, then we put
      $\dZ_v=\rW(\dF_v)$ and $\dQ_v=\dZ_v\otimes_\dZ\dQ$.

  \item Denote by $\dA$ (resp.\ $\dA^\infty$) the ring of ad\`{e}les
      (resp.\ finite ad\`{e}les) of $\dQ$.

  \item If $S$ is a set and $R$ is a ring, then we denote by $R[S]$ the
      $R$-module of functions from $S$ to $R$ with finite support. We
      have a degree map $\deg\colon R[S]\to R$ sending $f$ to $\sum_{s\in
      S}f(s)$, and denote by $R[S]^0$ the kernel of the degree map.
\end{itemize}

\subsubsection{Algebraic geometry}

\begin{itemize}
  \item Denote by $\bP^1$ the projective line scheme over $\dZ$, and
      $\bG_m=\Spec\dZ[T,T^{-1}]$ the multiplicative group scheme.

  \item If $S$ is a scheme and $A$ is an $S$-abelian scheme, then we
      denote by $A^\vee$ its dual $S$-abelian scheme.

  \item If $X/S$ is a relative scheme and $S'$ is an $S$-scheme, then we
      denote by $X_{S'}=X\times_SS'$ for the base change. If $S'=\Spec R$
      is affine, then we write $X_R$ instead of $X_{S'}$. If the notation
      for a scheme has already carried a subscript such as $X_{N^+M}$,
      then we denote by $X_{N^+M;S'}$ or $X_{N^+M;R}$ for the base
      change.

  \item Let $S$ be a scheme, and $j\colon Y\to X$ be a closed immersion
      of $S$-schemes. For a sheaf $\sF$ in the \'{e}tale topos $X_{\et}$,
      denote by
      \[\rH^i_Y(X_{S'},\sF)=\rh^i\rR\Gamma((Y_{S'})_{\et},b^*j^!\sF)\] the
      \emph{cohomology with support} for an $S$-scheme $S'$ and the base
      change morphism $b\colon Y_{S'}\to Y$. If $S'$ is a (pro-)Galois
      cover of $S$, then $\rH^i_Y(X_{S'};F)$ is equipped with a
      (continuous) action of the (pro-)finite Galois group $\Gal(S'/S)$.

  \item Let $X$ be a smooth proper scheme over a field $k$ of
      characteristic zero. For each integer $r\geq 0$, denote by
      $\CH^r(X)^0$ the subspace of $\CH^r(X)\otimes_\dZ\dQ$ of
      cohomologically trivial cycles, that is, the kernel of the
      geometric cycle class map
      \[\cl_X^0\colon\CH^r(X)\otimes_\dZ\dQ\to\rH^{2r}_{\et}(X_{k^\ac},\dQ_p(r))\]
      for some and hence all primes $p$.

  \item Let $X$ be a smooth proper scheme over a field $k$ purely of
      dimension $d$. Denote by $\Corr(X)=\CH^d(X\times_{\Spec
      k}X)\otimes_\dZ\dQ$ the $\dQ$-algebra of correspondences (of degree
      $0$) of $X$, with multiplication given by composition of
      correspondences. It acts on $\CH^*(X)\otimes_\dZ\dQ$ via pullbacks.
      If the characteristic of $k$ is $0$, then we denote by
      \[\rH^*_\dr(X)=\bigoplus_{r=0}^{2d}\rH^i_\dr(X)\]
      the total de Rham cohomology of $X$, which is a graded $k$-vector
      space. We have a natural homomorphism
      \[\cor_\dr\colon\Corr(X)\to\End_k^{\r{gr}}(\rH^*_\dr(X))\]
      of $\dQ$-algebras via pullbacks.
\end{itemize}

\subsubsection{Galois modules}
\label{ss:galois_modules}

\begin{itemize}
  \item For a field $k$, denote by $\Gamma_k=\Gal(k^\ac/k)$ the absolute
      Galois group of $k$. If $k=K$ is a subfield of $\dQ^\ac$, then we
      take $K^\ac=\dQ^\ac$ and hence $\Gamma_K\subset\Gamma_\dQ$.

  \item For an abelian variety $A$ over a field $k$ of characteristic not
      $p$, we put $\rT_p(A)=\varprojlim_m A[p^m](k^\ac)$ which is a
      $\dZ_p$-module with a continuous $\Gamma_k$-action.

  \item We simply write $\rH^\bullet(k,-)$ instead of
      $\rH^\bullet(\Gamma_k,-)$ for group cohomology of the absolute
      Galois group $\Gamma_k$.

  \item If a group $G$ acts on a set $V$, then $V^G$ is the subset of
      $G$-fixed elements. We write $V^k$ instead of $V^{\Gamma_k}$.

  \item If $k$ is a local field, then we denote by $\rI_k\subset\Gamma_k$
      the inertia subgroup, and $\Fr_k$ the arithmetic Frobenius element
      of $\Gamma_k/\rI_k$. If $v$ is a prime power, then we simply write
      $\Gamma_v$ for $\Gamma_{\dQ_v}$, $\rI_v$ for $\rI_{\dQ_v}$, and
      $\Fr_v$ for $\Fr_{\dQ_v}$.

  \item Let $R$ be a ring and $\rT$ be an $R[\Gamma_\dQ]$-module. For
      each prime power $v$, we have the following localization map
      \[\loc_v\colon\rH^1(\dQ,\rT)\to\rH^1(\dQ_v,\rT)\]
      of $R$-modules obtained by restriction. For $r\in R^\times$, we
      denote by $\rT[v|r]$ the largest $R$-submodule of $\rT$ on which
      $\rI_v$ acts trivially and $\Fr_v$ acts by multiplication by $r$.

  \item If $v$ is a prime power and $\rT$ is an $R[\Gamma_v]$-module,
      then we put
      \[\rH^1_\unr(\dQ_v,\rT)=\Ker[\rH^1(\dQ_v,\rT)\to\rH^1(\rI_v,\rT)]\]
      as an $R$-submodule of $\rH^1(\dQ_v,\rT)$. Put
      $\rH^1_\sing(\dQ_v,\rT)=\rH^1(\dQ_v,\rT)/\rH^1_\unr(\dQ_v,\rT)$
      with the quotient map
      $\partial_v\colon\rH^1(\dQ_v,\rT)\to\rH^1_\sing(\dQ_v,\rT)$.

  \item Let $\rG$ be a topological group. By a \emph{p-adic
      representation}, we mean a pair $(\rho,\rV)$ where $\rV$ is a
      finite dimensional $\dQ_p$-vector space, and
      $\rho\colon\rG\to\GL(\rV)$ is a continuous homomorphism. A
      \emph{stable lattice} in $(\rho,\rV)$ is a finitely generated
      $\dZ_p$-submodule $\rT\subset\rV$ such that
      $\rT\otimes_{\dZ_p}\dQ_p=\rV$ and $\rT$ is preserved under the
      action of $\rG$. Sometimes we will omit $\rho$ if it is irrelevant.

  \item Let $\rT$ be a stable lattice in a $p$-adic representation
      $(\rho,\rV)$ of $\rG$. For $n\geq 1$, we denote by
      \[\bar\rho^n\colon\rG\to\GL(\bar\rT^n)\]
      the induced residue representation, where
      $\bar\rT^n=\rT\otimes_{\dZ_p}\dZ/p^n$. Usually we omit the
      superscript $n$ when it is $1$. In certain cases, we also allow
      $n=\infty$, then $\bar\rT^\infty$ is nothing but $\rT$, and
      $\dZ/p^\infty$ is understood as $\dZ_p$.
\end{itemize}

\subsubsection*{Acknowledgements}

The author would like to thank Henri~Darmon, Wee-Tech~Gan, Benedict~Gross,
Atsushi~Ichino, Ye~Tian, Yichao~Tian, Shouwu~Zhang, and Wei~Zhang for helpful
discussions and comments. He appreciates Yichao~Tian and Liang~Xiao for
sharing their preprint \cite{TX14} at the early stage. He also thank the
anonymous referees for very careful reading and useful comments. The author
is partially supported by NSF grant DMS--1302000.

\section{Geometry of some modular schemes}
\label{s2}

In this chapter, we study geometry of some modular schemes which we will use.
In \Sec \ref{ss:algebraic_preparation}, we collect some notation and facts
about certain quaternion algebras and their orders. In \Sec
\ref{ss:preparations_abelian}, we study the canonical parametrization of
supersingular abelian surfaces using quaternion orders. In \Sec
\ref{ss:shimura_curves}, we study geometry of Shimura curves, focusing on the
special fibers of both smooth and semistable reductions. In \Sec
\ref{ss:hilbert_modular}, we study geometry of Hilbert modular surfaces and
their reduction at good inert primes. In \Sec \ref{ss:special_morphisms}, we
consider the connection between Shimura curves and Hilbert modular surfaces
through Hirzebruch--Zagier morphisms, particularly the interaction of their
reductions under such morphisms.

\subsection{Algebraic preparation}
\label{ss:algebraic_preparation}

We first introduce a bunch of notation.

\begin{notation}\label{no:algebra}
Recall that $F$ is a real quadratic number field.
\begin{enumerate}
  \item Denote by $\Cl(F)^+$ the strict ideal class group of $F$, whose
      elements are represented by projective $O_F$-modules of rank $1$
      \emph{with a notion of positivity}. Denote by $\fD$ the different
      of $F$. Then for example, the pair $(\fD^{-1},(\fD^{-1})^+)$ is an
      element in $\Cl(F)^+$ of order at most $2$. See
      \cite{Gor02}*{Chapter 2, \Sec 1} for more details.

  \item For each prime $\ell$ that is inert in $F$, we fix an embedding
      $\tau^\bullet_\ell\colon O_{F,\ell}\colonequals
      O_F\otimes_\dZ\dZ_\ell\hookrightarrow \rW(\dF_\ell^\ac)$ of
      $\dZ_\ell$-algebras, and put
      $\tau^\circ_\ell=\tau^\bullet_\ell\circ\theta$. They induce (by the
      same notation) two homomorphisms
      $\tau^\bullet_\ell,\tau^\circ_\ell\colon O_{F,\ell}\to
      O_F/\ell\hookrightarrow\dF_\ell^\ac$.
\end{enumerate}
\end{notation}

\begin{notation}\label{no:order_rational}
Let $D$ be a square-free positive integer.
\begin{enumerate}
  \item We denote by $B_D$ the unique indefinite (resp.\ definite)
      quaternion algebra over $\dQ$ which ramifies exactly at primes
      dividing $D$ when $D$ has even (resp.\ odd) number of prime
      factors. Denote by $B_D^+\subset B_D$ the subset of elements with
      positive reduced norm.

  \item Suppose that $D$ has \emph{even} number of prime factors. Denote
      by $\iota$ the canonical involution on $B_D$. Let $\cO_D$ be a
      maximal order of $B_D$. Fix $\delta\in\cO_D$ with $\delta^2=-D$.
      Define another involution $\ast$ on $B_D$ by the formula
      $b^\ast=\delta^{-1}b^\iota\delta$. It is easy to see that $\cO_D$
      is preserved under $\ast$.

  \item Suppose that $D$ has \emph{even} number of prime factors, all
      being \emph{inert} in $F$. We will fix an \emph{oriented} maximal
      order $(\cO_D,o_v)$ in $B_D$, which by Strong Approximation
      \cite{Vig80}*{Chaptire III, Th\'{e}or\`{e}me 4.3} is unique up to
      conjugation in $B_D$, together with an embedding
      $O_F\hookrightarrow\cO_D$ of $\dZ$-algebras. Recall that an
      orientation of $\cO_D$ is a homomorphism
      $o_v\colon\cO_D\to\dF_{v^2}$ for each prime $v\mid D$. We require
      that the restriction $o_v\res_{O_F}$ coincides with
      $\tau^\bullet_v$. Put $\fM_D=\{b\in\cO_D\res xb=b^\ast
      x^\ast,\forall x\in O_F\}$ and $\fM_D^+=\fM_D\cap B_D^+$. Then
      $(\fM_D,\fM_D^+)\simeq (\fD^{-1},(\fD^{-1})^+)$. We fix such an
      isomorphism $\psi_D$ once and for all.

  \item Suppose that $D$ has \emph{odd} number of prime factors. If $M$
      is another positive integer coprime to $D$, then we denote by
      $\cT_{M,D}$ the set of isomorphism classes of oriented Eichler
      orders of level $M$ in $B_D$.

  \item For positive integers $d'\mid d\mid M$, we have a degeneracy map
      \[\delta_{M,D}^{(d,d')}\colon\cT_{M,D}\to\cT_{M/d,D}.\]
\end{enumerate}
\end{notation}

Recall that an oriented Eichler order of level $M$ in $B_D$ (see
\cite{Rib89}*{\Sec 2}) is a tuple $\vec{R}=\{R,R',o_v\}$ where
\begin{itemize}
  \item $R$ is an Eichler order of level $M$ in $B_D$;

  \item $R'$ is a maximal Eichler order such that $R=R'\cap R''$ for
      another (unique) maximal Eichler order $R''$;

  \item $o_v\colon R\to\dF_{v^2}$ is a homomorphism of $\dZ$-algebras for
      \emph{every} prime $v\mid D$.
\end{itemize}

For the degeneracy map, suppose that $\vec{R}=(R,R',o_v)$ is an oriented
Eichler order of level $M$ in $B_D$. Let $R_1$ be the unique Eichler order of
level $M/d$ containing $R$ such that $\#(R'/R'\cap R_1)=Md'/d$; and $R'_1$ be
the unique maximal Eichler order containing $R_1$ such that $\#(R'/R'\cap
R'_1)=d'$. Then we define $\delta_{M,D}^{(d,d')}(\vec{R})$ to be
$(R_1,R'_1,o_v)$.

For each prime $\ell$ such that $\ell^s\|M$ with $s\geq 1$, we have a unique
involution \emph{switching the orientation at $\ell$}
\[\op_\ell\colon\cT_{M,D}\to\cT_{M,D},\]
such that
$\delta_{M,D}^{(\ell^s,\ell^s)}=\delta_{M,D}^{(\ell^s,1)}\circ\op_\ell$. If
$\ell\mid D$, then we also have an involution \emph{switching the orientation
at $\ell$}
\[\op_\ell\colon\cT_{M,D}\to\cT_{M,D},\]
sending $\vec{R}=(R,R',o_v)$ to the one obtained by only switching $o_\ell$
to its conjugate.

\begin{notation}\label{no:order_quadratic}
Let $\fM$ be an ideal of $O_F$.
\begin{enumerate}
  \item Denote by $Q$ the unique totally definite quaternion algebra over
      $F$ that is unramified at all finite places. Denote by $\cS_\fM$
      the set of isomorphism classes of oriented ($O_F$-linear) Eichler
      orders $\vec{R}=(R,R')$ of level $\fM$ in $Q$.

  \item For ideals $\fd'\mid \fd\mid\fM$, we have a similar degeneracy
      map
      \[\gamma^{(\fd,\fd')}_\fM\colon\cS_\fM\to\cS_{\fM/\fd}.\]

  \item For each prime ideal $\fl$ of $O_F$ such that $\fl\mid\fM$, we
      have a similar involution \emph{switching the orientation at $\fl$}
      \[\op_\fl\colon\cS_\fM\to\cS_\fM.\]
\end{enumerate}
\end{notation}

\begin{definition}\label{de:special}
Suppose that $D$ has \emph{even} number of prime factors, all being
\emph{inert} in $F$. We would like to construct a canonical \emph{special
map}
\[\zeta_{M,D}\colon\cT_{M,D}\to\cS_M\]
as follows. Take $\vec{R}=(R,R',o_v)\in\cT_{M,D}$. There is a unique
($O_F$-linear) Eichler order $R_\sharp$ of level $M$ in the $F$-quaternion
algebra $B_D\otimes_\dQ F\simeq Q$ such that it contains $R\otimes_\dZ O_F$,
and for every prime $v\mid D$, the $(R/v,O_F/v)$-bimodule
$R_\sharp\otimes_\dZ\dZ_v/R\otimes_\dZ O_{F,v}\simeq\dF_{v^2}$ is isomorphic
to $o_v\otimes\tau^\bullet_v$. Similarly, one defines $R'_\sharp$. Then we
define $\zeta_{M,D}(\vec{R})$ to be $(R_\sharp,R'_\sharp)$. The special maps
$\zeta_{M,D}$ are compatible with degeneracy maps in the sense that the
following diagram
\[\xymatrix{
\cT_{M,D} \ar[rr]^-{\delta^{(d,d')}_{M,D}} \ar[d]_-{\zeta_{M,D}} && \cT_{M/d,D} \ar[d]^-{\zeta_{M/d,D}} \\
\cS_M \ar[rr]^-{\gamma^{(d,d')}_M} && \cS_{M/d} }\] commutes for all $d'\mid
d\mid M$.
\end{definition}

Let $\ell$ be a prime that is inert in $F$. Fix an oriented maximal order
$(\cO_\ell,o_\ell)$ of $B_\ell$. For each prime $v$, put
$\cO_{\ell,v}=\cO_\ell\otimes_\dZ\dZ_v$. For $?=\bullet,\circ$, fix an
embedding $\varsigma_\ell^?\colon O_{F,\ell}\hookrightarrow\cO_{\ell,\ell}$
of $\dZ_\ell$-algebras such that its composition with $o_\ell$ coincides with
$\tau^?_\ell$.

\begin{lem}
Let $n$ be a positive integer. Let $f\colon
O_{F,\ell}\to\Mat_n(\cO_{\ell,\ell})$ be a homomorphism of
$\dZ_\ell$-algebras.
\begin{enumerate}
  \item Then $f$ is $\GL_n(\cO_{\ell,\ell})$-conjugate to a homomorphism
      of the form
      \[x\mapsto
      \diag[\varsigma^\bullet_\ell(x),\cdots,\varsigma^\bullet_\ell(x),
      \varsigma^\circ_\ell(x),\cdots,\varsigma^\circ_\ell(x)].\] We say
      that $f$ is of type $(r,n-r)$ if $\varsigma^\bullet_\ell$ appears
      with $r$ times.

  \item If $f$ is of type $(r,n-r)$, then the commutator of
      $f(O_{F,\ell})$ in $\Mat_n(\cO_{\ell,\ell})$ is isomorphic to
      $\End(O_{F,\ell}^{\oplus r}\oplus \ell O_{F,\ell}^{\oplus
      n-r})\cap\Mat_n(O_{F,\ell})$ as an $O_{F,\ell}$-algebra.
\end{enumerate}
\end{lem}

\begin{proof}
The homomorphism $f$ corresponds to an
$O_{F,\ell}\otimes_{\dZ_\ell}\cO_{\ell,\ell}$-module $M_f$, free of finite
rank over $\cO_{\ell,\ell}$. Since
$O_{F,\ell}\otimes_{\dZ_\ell}\cO_{\ell,\ell}$ is an Eichler order of level
$\ell$ in $\Mat_2(F_\ell)$, the
$O_{F,\ell}\otimes_{\dZ_\ell}\cO_{\ell,\ell}$-module $M_f$ is isomorphic to a
direct sum of $\cO_{\ell,\ell}$ on which $\cO_{\ell,\ell}$ acts by right
multiplication, and $O_{F,\ell}$ acts from left via \emph{either}
$\varsigma^\bullet_\ell$ \emph{or} $\varsigma^\circ_\ell$. This proves (1),
of which (2) is a corollary.
\end{proof}

Now we study $(O_F,\cO_\ell)$-bimodules of ($\dZ$-)rank $8$. For such a
bimodule $M$, consider the $(O_{F,\ell},\cO_{\ell,\ell})$-bimodule
$M_\ell\colonequals M\otimes_\dZ\dZ_\ell$ which corresponds to a homomorphism
$f_M\colon O_{F,\ell}\to\Mat_2(\cO_{\ell,\ell})$. We say that $M$ is
\emph{pure} (resp.\ \emph{mixed}) if $f_M$ is of type $(0,2)$ or $(2,0)$
(resp.\ $(1,1)$). We deduce the following facts from the above lemma: If $M$
is pure, then $\End_{(O_F,\cO_\ell)}(M)$, as an $O_F$-module, is isomorphic
to a maximal order in $Q$, that is, an element in $\cS_1$. If $M$ is mixed,
then there is a unique sub-bimodule $M_\bullet$ of index $\ell^2$ that is of
type $(0,2)$, and the inclusion
$\End_{(O_F,\cO_\ell)}(M)\subset\End_{(O_F,\cO_\ell)}(M_\bullet)$ is
naturally an element in $\cS_\ell$. In all cases, we simply denote by
$\End(M)$ the endomorphism $O_F$-algebra with the above orientation.

\begin{proposition}\label{pr:order_bijection}
The assignment $M\mapsto \End(M)$ induces a bijection between the set of
isomorphism classes of pure of a fixed type (resp.\ mixed)
$(O_F,\cO_\ell)$-bimodules of rank $8$, and $\cS_1$ (resp.\ $\cS_\ell$).
\end{proposition}

This proposition is the analogue of \cite{Rib89}*{Theorem 2.4} for
$(O_F,\cO_\ell)$-bimodules. Note that in our case, the notion of
admissibility is unnecessary. The proof of this proposition is very similar
to that of \cite{Rib89}*{Theorem 2.4}. Therefore, we only indicate necessary
changes in the following proof, instead of explicating all the details.

\begin{proof}
We first prove that the assignment $M\mapsto \End(M)$ is injective. Choose an
$(O_F,\cO_\ell)$-bimodule $M$ as a reference point, and put
$\Lambda=\End_{(O_F,\cO_\ell)}(M)$. For each $(O_F,\cO_\ell)$-bimodule $N$
that is locally isomorphic to $M$, put $J(N)=\Hom(M,N)$ as a right
$\Lambda$-module which is locally free of rank $1$. The same proof for
\cite{Rib89}*{Theorem 2.3} implies that the assignment $N\mapsto J(N)$
establishes a bijection between the sets of isomorphism classes of
\begin{itemize}
  \item $(O_F,\cO_\ell)$-bimodules that are locally isomorphic to $M$;

  \item locally free right $\Lambda$-modules of rank $1$.
\end{itemize}
Therefore, the injectivity amounts to say that for $M$ and $M'$ with
isomorphic oriented Eichler order $\End(M)\simeq\End(M')$, the right
$\Lambda$-module $\Hom(M,M')$ is \emph{free} of rank $1$. The rest of the
argument is the same as the one for \cite{Rib89}*{Theorem 2.4}, which works
for an arbitrary Dedekind ring, in particular $O_F$, instead of $\dZ$.

We then prove that the assignment $M\mapsto \End(M)$ is surjective. The first
step is to show that there exists a mixed $(O_F,\cO_\ell)$-bimodule of rank
$8$, which has been constructed in the proof of \cite{Rib89}*{Theorem 2.4}.
The rest of the proof is a counting argument, which apparently works for $Q$
as well.
\end{proof}

\subsection{Some preparations on abelian surfaces}
\label{ss:preparations_abelian}

Fix an odd prime $\ell$ that is inert in $F$, and an ideal $\fM$ of $O_F$
coprime to $\ell$. Let $k$ be an algebraically closed field containing
$\dF_{\ell^2}$. Consider triples $(A,\iota,C)$ where $A$ is a $k$-abelian
surface, $\iota\colon O_{F}\to\End(A)$ is a homomorphism of $\dZ$-algebras,
and $C$ is a level-$\fM$ structure, that is, a subgroup of $A(k)$ of order
$\Nm_{F/\dQ}\fM$ that is stable and cyclic under the action of $O_{F}$.

\begin{definition}
We say that the triple $(A,\iota,C)$ is
\begin{enumerate}
  \item \emph{$\bullet$-pure} if $A$ is superspecial and $O_F/\ell$ acts
      on the $k$-vector space $\Lie A$ by $\tau^\bullet_\ell$;

  \item \emph{$\circ$-pure} if $A$ is superspecial and $O_F/\ell$ acts on
      the $k$-vector space $\Lie A$ by $\tau^\circ_\ell$;

  \item \emph{mixed} if both $\tau^\bullet_\ell$ and $\tau^\circ_\ell$
      appear in the action of $O_F/\ell$ on $\Lie A$;

  \item \emph{superspecial mixed} if it is mixed and $A$ is superspecial.
\end{enumerate}
\end{definition}

It is clear that the Frobenius element $\Fr_\ell$ acts on the set of all
isomorphism classes of triples by sending $(A,\iota,C)$ to its Frobenius
twist $(A^{(\ell)},\iota^{(\ell)},C^{(\ell)})$.

For each triple $(A,\iota,C)$, denote by $A[\ell]_\alpha$ the maximal
$\alpha_\ell$-elementary subgroup of $A[\ell]$, which is isomorphic to
$\alpha_\ell^{\oplus a(A)}$ for some number $a(A)\in\{0,1,2\}$. Then $A$ is
superspecial if and only if $a(A)=2$. For each mixed triple $(A,\iota,C)$
with $a(A)=1$, we say that it is \emph{$\bullet$-mixed} (resp.\
\emph{$\circ$-mixed}) if $O_F/\ell$ acts on the Dieudonn\'{e} module of
$A[\ell]_\alpha$ via $\tau^\bullet_\ell$ (resp.\ $\tau^\circ_\ell$). Given a
$\bullet$-pure (resp.\ $\circ$-pure) triple $(A,\iota,C)$ and a subgroup
$H\subset A[\ell]_\alpha$ that is isomorphic to $\alpha_\ell$ (in particular,
$H$ is automatically stable under the action of $\iota(O_F)$), we have an
induced triple $(A/H,\iota,C)$. It is either $\bullet$-mixed (resp.\
$\circ$-mixed) or superspecial mixed. The set of such subgroups $H$ is
isomorphic to $\bP^1(k)$. In fact, one can construct a family, called
\emph{Moret-Bailly family} (see, for example, \cite{Nic00}*{\Sec 4}), of
triples over $\bP^1_k$ parameterizing $\{(A/H,\iota,C)\res H\}$. There are
exactly $\ell^2+1=\#\bP^1(\dF_{\ell^2})$ subgroups $H$ such that
$(A/H,\iota,C)$ is superspecial mixed.

\begin{remark}\label{re:surface_superspecial}
If we start with a superspecial mixed triple $(A,\iota,C)$, then there is a
\emph{unique} $O_F$-stable subgroup $H_\bullet$ (resp.\ $H_\circ$) of
$A[\ell]_\alpha$ isomorphic to $\alpha_\ell$ such that $O_F/\ell$ acts on its
Dieudonn\'{e} module by $\tau_\ell^\bullet$ (resp.\ $\tau_\ell^\circ$).
Dividing it, we obtain a triple $(A/H_\bullet,\iota,C)$ (resp.\
$(A/H_\circ,\iota,C)$) that is $\circ$-pure (resp.\ $\bullet$-pure).
Moreover, the triple $(A^{(\ell)},\iota^{(\ell)},C^{(\ell)})$ appears in the
Moret-Bailly family associated to $(A/H_\bullet,\iota,C)$ (resp.\
$(A/H_\circ,\iota,C)$), as obtained by dividing $A[\ell]_\alpha/H_\bullet$
(resp.\ $A[\ell]_\alpha/H_\circ$).
\end{remark}

For each $(A,\iota,C)$, denote by $\End(A,\iota,C)$ the $O_F$-algebra of
endomorphisms of $A$ that commute with $\iota(O_F)$ and preserve $C$. If
$(A,\iota,C)$ is pure (resp.\ superspecial mixed), then $\End(A,\iota,C)$ is
an Eichler order of level $\fM$ (resp.\ $\fM\ell$) in $Q$. Moreover, we endow
$\End(A,\iota,C)$ with the orientation given by the maximal order
$\End(A,\iota)$ (resp.\ $\End(A/H_\bullet,\iota)$).

\begin{proposition}\label{pr:surface_superspecial}
The assignment $(A,\iota,C)\mapsto\End(A,\iota,C)$, with the above
orientation, induces a bijection between
\begin{enumerate}
  \item the set of isomorphism classes of $\bullet$-pure triples and
      $\cS_\fM$;
  \item the set of isomorphism classes of $\circ$-pure triples and
      $\cS_\fM$;
  \item the set of isomorphism classes of superspecial mixed triples and
      $\cS_{\fM\ell}$.
\end{enumerate}
Moreover, the Frobenius element $\Fr_\ell$
\begin{itemize}
  \item switches the types of $\bullet$-pure and $\circ$-pure and does
      not change values in $\cS_\fM$;

  \item preserves the set of superspecial mixed triples and coincides
      with $\op_\ell$ on $\cS_{\fM\ell}$.
\end{itemize}
\end{proposition}

\begin{proof}
The proof of this proposition follows exactly as for \cite{Rib89}*{Theorems
4.15 \& 4.16}, in which we only need to replace \cite{Rib89}*{Theorem 2.14}
by Proposition \ref{pr:order_bijection}. The action of $\Fr_\ell$ follows
from the construction.
\end{proof}

\subsection{Shimura curves}
\label{ss:shimura_curves}

Let $N^-$ be a square-free positive integer such that $\wp(N^-)$ is even, and
$N^+$ be another integer coprime to $N^-$. Denote by
$\underline{\cY}^\natural_{N^+,N^-}$ the functor from the category of schemes
over $\dZ[1/N^+]$ to the category of groupoids such that
$\underline{\cY}^\natural_{N^+,N^-}(S)$ is the groupoid of triples
$(A,\iota,C)$, where
\begin{itemize}
  \item $A$ is an $S$-abelian scheme of relative dimension $2$;
  \item $\iota\colon\cO_{N^-}\to\End(A)$ is a special action (see
      \cite{BC91}*{\Sec III.3}) of $\cO_{N^-}$ on $A$;
  \item $C$ is a level-$N^+$ structure, that is, a subgroup scheme of
      $A[N^+]$ of order $(N^+)^2$ that is stable and cyclic under the
      action of $\cO_{N^-}$.
\end{itemize}
Suppose that $N^+$ is \emph{neat} in the sense that $(A,\iota,C)$ has no
non-trivial automorphism. Then the functor
$\underline{\cY}^\natural_{N^+,N^-}$ is represented by a scheme
$\cY^\natural_{N^+,N^-}$, quasi-projective over $\dZ[1/N^+]$, projective if
and only if $N^->1$. We denote by $\cY_{N^+,N^-}$ the canonical
compactification of $\cY^\natural_{N^+,N^-}$ (only necessary when $N^-=1$).
It is a regular scheme and smooth over $\dZ[1/N^+N^-]$.

\begin{notation}\label{no:neron}
Put $Y_{N^+,N^-}=\cY_{N^+,N^-;\dQ}$ and denote its Jacobian by $J_{N^+,N^-}$.
For each prime $\ell\mid N^-$, we denote by $\Phi_{N^+,N^-}^{(\ell)}$ the
group of connected components of the special fiber of the N\'{e}ron model of
$J_{N^+,N^-;\dQ_{\ell^2}}$ over $\dZ_{\ell^2}$.
\end{notation}

\subsubsection{Hecke algebra and cohomology}

Let $d$ be a factor of $N^+$, we have degeneracy morphisms
\[\delta_{N^+,N^-}^{(d,d')}\colon \cY_{N^+,N^-}\to\cY_{N^+/d,N^-}\]
indexed by positive integers $d'\mid d$ such that $\delta_{N^+,N^-}^{(d,1)}$
is the natural projection by forgetting the level structure. For each prime
$v\nmid N^+N^-$, we have the Hecke correspondence
$T_{N^+,N^-}^{(v)}\in\Corr(Y_{N^+,N^-})$ as
\[\xymatrix{
& Y_{N^+v,N^-} \ar[ld]_-{\delta_{N^+v,N^-}^{(v,1)}} \ar[rd]^-{\delta_{N^+v,N^-}^{(v,v)}} \\
Y_{N^+,N^-} && Y_{N^+,N^-}.
}\]

Denote by $\dT_{N^+,N^-}^{\sph,\rat}$ the (commutative) $\dQ$-subalgebra of
$\Corr(Y_{N^+,N^-})$ generated by $T_{N^+,N^-}^{(v)}$ for all primes $v\nmid
N^+N^-$. Denote by $\dT_{N^+,N^-}^{\sph,\dr}$ the image of
$\dT_{N^+,N^-}^{\sph,\rat}$ under the homomorphism $\cor_\dr$, which is a
finite-dimensional commutative $\dQ$-algebra. For each
$\Gal(\dC/\dQ)$-conjugacy class $\sigma$ of irreducible cuspidal automorphic
representations of $\GL_2(\dA)$ that occur in $\rH^1_\dr(Y_{N^+,N^-})$, we
have a unique idempotent $\sP^\dr_\sigma\in\dT_{N^+,N^-}^{\sph,\dr}$ such
that it induces the natural projection
$\bigoplus\rH^*_\dr(Y_{N^+,N^-})\to\rH^1_\dr(Y_{N^+,N^-})[\sigma]$ onto the
$\sigma$-isotypic subspace $\rH^1_\dr(Y_{N^+,N^-})[\sigma]$ of
$\rH^1_\dr(Y_{N^+,N^-})$.

\begin{definition}\label{de:projector_curve}
A \emph{$\sigma$-projector} (of level $N^+$) is an element $\sP_\sigma$ in
$\dT_{N^+,N^-}^{\sph,\rat}$ such that $\cor_\dr(\sP_\sigma)=\sP^\dr_\sigma$.
\end{definition}

Consider the $p$-adic \'{e}tale cohomology
$\rH^1_{\et}(Y_{N^+,N^-;\dQ^\ac},\dQ_p(1))$ as a $p$-adic representation of
$\Gamma_\dQ$. Denote by $\rH^1_\sigma(Y_{N^+,N^-;\dQ^\ac})$ its
$\sigma$-isotypic subspace which is preserved under the action of
$\Gamma_\dQ$. Then $\sP^\dr_\sigma$ induces a natural projection
\[\sP^{\et}_\sigma\colon\rH^1_{\et}(Y_{N^+,N^-;\dQ^\ac},\dQ_p(1))\to\rH^1_\sigma(Y_{N^+,N^-;\dQ^\ac}),\]
by the Comparison Theorem between de Rham and \'{e}tale cohomology. It is
$\Gamma_\dQ$-equivariant.

\subsubsection{Good reduction at $\ell\nmid N^-$}

Suppose that $\ell$ is a prime not dividing $N^+N^-$. Denote by
$\cY^\ssl_{N^+,N^-;\dF_{\ell^2}}$ the supersingular locus of
$\cY^\natural_{N^+,N^-;\dF_{\ell^2}}$, which is preserved under the action of
$\Fr_\ell$.

\begin{proposition}\label{pr:curve_supersingular}
The scheme $\cY^\ssl_{N^+,N^-;\dF_{\ell^2}}$ is a disjoint union of
$\Spec\dF_{\ell^2}$. There is a \emph{canonical} bijection between
$\cY^\ssl_{N^+,N^-;\dF_{\ell^2}}(\dF_{\ell^2})$ and $\cT_{N^+,N^-\ell}$,
under which $\Fr_\ell$ acts by $\op_\ell$.
\end{proposition}

\begin{proof}
The first statement is well-known. The second statement is proved in
\cite{Rib89}*{Theorem 3.4}.

For later use, we briefly recall how to construct the map
$\cY^\ssl_{N^+,N^-;\dF_{\ell^2}}(\dF_{\ell^2})\to\cT_{N^+,N^-\ell}$ in the
case $N^+=1$. Suppose that $(A,\iota)$ is an element in
$\cY^\ssl_{1,N^-;\dF_{\ell^2}}(\dF_{\ell^2})$. The (maximal) Eichler order in
$B_{N^-\ell}$ would be $\End(A,\iota)$, that is, the algebra of all
endomorphisms commuting with $\iota(\cO_{N^-})$. We only need to specify the
orientation $o_v\colon\End(A,\iota)\to\dF_{v^2}$ for every prime $v$ dividing
$N^-\ell$. There are two cases.
\begin{itemize}
  \item When $v\neq\ell$, we let $\fm_v$ be the maximal ideal of
      $\End(A,\iota)$ of residue characteristic $v$. Then the action of
      $\End(A,\iota)/\fm_v$ on $A[\fm_v](\dF_\ell^\ac)$, which is an
      $\dF_{v^2}$-vectors space of dimension $1$, induces a homomorphism
      $\End(A,\iota)\to\dF_{v^2}$. We set this homomorphism to be $o_v$.

  \item When $v=\ell$, consider the action of $\End(A,\iota)$ on the
      $\dF_{\ell^2}$-vector space $\Lie A$, which is in fact a scalar
      action, and hence induces an orientation
      $o_\ell\colon\End(A,\iota)\to\dF_{\ell^2}$.
\end{itemize}
\end{proof}

\begin{remark}\label{re:degeneracy_supersingular}
By construction, the map
$\cY^\ssl_{N^+,N^-;\dF_{\ell^2}}(\dF_{\ell^2})\to\cY^\ssl_{N^+/d,N^-;\dF_{\ell^2}}(\dF_{\ell^2})$
induced by the degeneracy morphism $\delta_{N^+,N^-}^{(d,d')}\colon
\cY_{N^+,N^-}\to\cY_{N^+/d,N^-}$ coincides with
$\delta_{N^+,N^-\ell}^{(d,d')}\colon\cT_{N^+,N^-\ell}\to\cT_{N^+/d,N^-\ell}$
(see Notation \ref{no:order_rational} (5)).
\end{remark}

\subsubsection{Bad reduction at $\ell\mid N^-$}

Suppose that $\ell$ divides $N^-$. In particular, the scheme
$\cY_{N^+,N^-}=\cY^\natural_{N^+,N^-}$ is proper. For simplicity and also for
the later application, we further assume that $\ell$ is odd and inert in $F$.
Recall that we have always fixed an embedding $O_F\hookrightarrow\cO_{N^-}$.

For each geometric point $s\in\cY_{N^+,N^-;\dF_\ell}(k)$, denote the
associated abelian variety by $A_s$. Denote by $\tau(s)$ the multi-set of
characters of $O_F/\ell$ appearing in its action on the Dieudonn\'{e} module
of $A_s[\ell]_\alpha$. There are three possibilities for $\tau(s)$:
$\{\tau^\bullet_\ell\}$, $\{\tau^\circ_\ell\}$ and
$\{\tau^\bullet_\ell,\tau^\circ_\ell\}$.

Denote by $\cY^\bullet_{N^+,N^-;\dF_{\ell^2}}$ (resp.\
$\cY^\circ_{N^+,N^-;\dF_{\ell^2}}$) the union of irreducible components of
$\cY_{N^+,N^-;\dF_{\ell^2}}$ whose generic geometric point $s$ satisfies
$\tau(s)=\{\tau^\bullet_\ell\}$ (resp.\ $\tau(s)=\{\tau^\circ_\ell\}$).
Denote by $\cY^\ssp_{N^+,N^-;\dF_{\ell^2}}$ the scheme theoretical
intersection of $\cY^\bullet_{N^+,N^-;\dF_{\ell^2}}$ and
$\cY^\circ_{N^+,N^-;\dF_{\ell^2}}$. Note that
$\cY^\ssp_{N^+,N^-;\dF_{\ell^2}}$ classifies \emph{mixed and exceptional}
data in the terminology of \cite{Rib89}.

\begin{proposition}\label{pr:curve_superspecial}
We have that
\begin{enumerate}
  \item $\cY^\bullet_{N^+,N^-;\dF_{\ell^2}}$ and
      $\cY^\circ_{N^+,N^-;\dF_{\ell^2}}$ intersect transversally at
      $\cY^\ssp_{N^+,N^-;\dF_{\ell^2}}$;

  \item $\cY^\ssp_{N^+,N^-;\dF_{\ell^2}}$ is a disjoint union of
      $\Spec\dF_{\ell^2}$, \emph{canonically} indexed by
      $\cT_{N^+\ell,N^-/\ell}$, under which $\Fr_\ell$ acts by
      $\op_\ell$;

  \item both $\cY^\bullet_{N^+,N^-;\dF_{\ell^2}}$ and
      $\cY^\circ_{N^+,N^-;\dF_{\ell^2}}$ are disjoint union of
      $\bP^1_{\dF_{\ell^2}}$, each \emph{canonically} indexed by
      $\cT_{N^+,N^-/\ell}$, under which $\Fr_\ell$ acts by switching two
      factors.
\end{enumerate}
\end{proposition}

\begin{proof}
Part (1) is a consequence of \v{C}erednik--Drinfeld uniformization. The
remaining parts are proved in \cite{Rib89}*{Theorem 4.15 \& 4.16, \Sec 5}.
\end{proof}

\begin{remark}\label{re:degeneracy_superspecial}~
\begin{enumerate}
  \item In order to distinguish between different orientations, when we
      identify the set $\pi_0(\cY^\bullet_{N^+,N^-;\dF_{\ell^2}})$
      (resp.\ $\pi_0(\cY^\circ_{N^+,N^-;\dF_{\ell^2}})$) with
      $\cT_{N^+,N^-/\ell}$, we will use the notation
      $\cT_{N^+,N^-/\ell}^\bullet$ (resp.\ $\cT_{N^+,N^-/\ell}^\circ$).
      As a set, it is nothing but $\cT_{N^+,N^-/\ell}$.

 \item The natural map
     \begin{align*}
     &\cY^\ssp_{N^+,N^-;\dF_{\ell^2}}(\dF_{\ell^2})\to\pi_0(\cY^\bullet_{N^+,N^-;\dF_{\ell^2}})\\
     \text{resp. }&
     \cY^\ssp_{N^+,N^-;\dF_{\ell^2}}(\dF_{\ell^2})\to\pi_0(\cY^\circ_{N^+,N^-;\dF_{\ell^2}})
     \end{align*}
     is given by $\delta_{N^+\ell,N^-/\ell}^{(\ell,1)}$ (resp.\
     $\delta_{N^+\ell,N^-/\ell}^{(\ell,\ell)}$); see Notation
     \ref{no:order_rational} (5).

  \item For $?=\bullet,\circ$, the map
      $\pi_0(\cY^?_{N^+,N^-;\dF_{\ell^2}})\to\pi_0(\cY^?_{N^+/d,N^-;\dF_{\ell^2}})$
      induced by the degeneracy morphism $\delta_{N^+,N^-}^{(d,d')}\colon
      \cY_{N^+,N^-}\to\cY_{N^+/d,N^-}$ coincides with the map
      \[\delta_{N^+,N^-/\ell}^{(d,d')}\colon\cT_{N^+,N^-/\ell}\to\cT_{N^+/d,N^-\ell}.\]
      See Notation \ref{no:order_rational} (5).
\end{enumerate}
\end{remark}

\subsection{Hilbert modular surfaces}
\label{ss:hilbert_modular}

We first introduce the Hilbert modular scheme for $F$. The details of the
construction can be found in \cite{Rap78}*{\Sec 1} or \cite{Gor02}*{Chapter
3}. Let $\fM$ be an ideal of $O_F$.

For each representative $(\fa,\fa^+)\in\Cl(F)^+$, denote by
$\widetilde{\underline{\cX}}_\fM^{(\fa,\fa^+)}$ the functor from the category
of schemes over $\dZ[1/\Nm_{F/\dQ}\fM\cdot\disc{F}]$ to the category of
groupoids such that $\widetilde{\underline{\cX}}_\fM^{(\fa,\fa^+)}(S)$ is the
groupoid of quadruples $(A,\iota,\psi,C)$, where
\begin{itemize}
  \item $A$ is an $S$-abelian scheme of relative dimension $2$;

  \item $\iota$ is an \emph{$O_F$-multiplication} on $A$, that is, a
      homomorphism $\iota\colon O_F\to\End(A)$ of $\dZ$-algebras
      satisfying the \emph{Rapoport condition}: locally over $S$, the
      $O_F\otimes_\dZ\sO_S$-module $\Lie A$ is free of rank $1$;

  \item $\psi\colon(\fP_A,\fP_A^+)\xrightarrow\sim(\fa,\fa^+)$, where
      $\fP_A=\Hom_{O_F}(A,A^\vee)^\symm$ is the $O_F$-module of symmetric
      homomorphisms from $A$ to $A^\vee$, and $\fP_A^+\subset\fP_A$ is
      the subset of polarizations;

  \item $C$ is a level-$\fM$ structure, that is, a subgroup scheme of
      $A[\fM]$ of order $\Nm_{F/\dQ}\fM$ that is stable and cyclic under
      the action of $O_F$.
\end{itemize}
Suppose that $\fM$ is \emph{neat} in the sense that the \emph{triple}
$(A,\iota,C)$ has no non-trivial automorphism. Then the functor
$\widetilde{\underline{\cX}}_\fM^{(\fa,\fa^+)}$ is represented by a scheme
$\widetilde\cX_\fM^{(\fa,\fa^+)}$, smooth and quasi-projective over
$\dZ[1/\Nm_{F/\dQ}\fM\cdot\disc(F)]$. Put
\[\widetilde\cX_\fM^\natural=\bigsqcup_{\Cl(F)^+}\widetilde\cX_\fM^{(\fa,\fa^+)},\]
which is known as the \emph{Hilbert modular scheme} associated to $F$ of
level-$\fM$ structure. Put $U_F=O_F^{\times,+}/O_F^{\times,2}$, whose
cardinality is at most $2$. The finite group $U_F$ acts on
$\widetilde\cX_\fM^{(\fa,\fa^+)}$ by sending $(A,\iota,\psi,C)$ to
$(A,\iota,\psi\circ u,C)$ for $u\in U_F$. Denote by $\cX_\fM^{(\fa,\fa^+)}$
the quotient of $\widetilde\cX_\fM^{(\fa,\fa^+)}$ by $U_F$ under this (free,
since $\fM$ is neat) action, which is an algebraic space. Put
\[\cX_\fM^\natural=\bigsqcup_{\Cl(F)^+}\cX_\fM^{(\fa,\fa^+)}.\]

\subsubsection{Relation with Shimura varieties}

Put $\bG=\Res_{F/\dQ}\GL_{2,F}$ as a reductive group over $\dQ$. Define
$\widetilde{\bG}$ as the pullback in the following diagram
\[\xymatrix{
\widetilde{\bG} \ar[r]\ar[d]_-{\det} & \bG \ar[d]^-{\det} \\
\bG_{m,\dQ} \ar[r]^-{\subset} & \Res_{F/\dQ}\bG_{m,F}. }\] Denote by
$\bK_\fM$ the open compact subgroup of $\bG(\dA^\infty)$ corresponding to the
level-$\fM$ structure, and $\widetilde\bK_\fM$ its intersection with
$\widetilde\bG(\dA^\infty)$. We have the natural morphism of canonical models
(over $\dQ$) of Shimura varieties
\begin{align}\label{eq:shimura1}
\Sh(\widetilde\bG,\widetilde\bK_\fM)_\dQ\to\Sh(\bG,\bK_\fM)_\dQ,
\end{align}
whose complex analytification is identified with the following well-known
double-coset expression
\begin{align}\label{eq:shimura2}
\widetilde\bG(\dQ)\backslash(\dC-\dR)^2\times\widetilde\bG(\dA^\infty)/\widetilde\bK_\fM
\to\bG(\dQ)\backslash(\dC-\dR)^2\times\bG(\dA^\infty)/\bK_\fM.
\end{align}

The scheme $\widetilde\cX_\fM^{(\fD^{-1},(\fD^{-1})^+)}$ is an integral
canonical model of the Shimura variety
$\Sh(\widetilde\bG,\widetilde\bK_\fM)_\dQ$ over
$\dZ[1/\Nm_{F/\dQ}\fM\cdot\disc(F)]$ (see \cite{Rap78}*{Theorem 1.28}). By
abuse of notation, denote by $\Sh(\bG,\bK_\fM)$ \emph{the} integral canonical
model of $\Sh(\bG,\bK_\fM)_\dQ$ over $\dZ[1/\Nm_{F/\dQ}\fM\cdot\disc(F)]$.
Therefore, \eqref{eq:shimura1} extends uniquely to a morphism
\begin{align}\label{eq:shimura3}
\widetilde\cX_\fM^{(\fD^{-1},(\fD^{-1})^+)}\to\Sh(\bG,\bK_\fM).
\end{align}
In fact, we have a morphism
\begin{align}\label{eq:shimura4}
\widetilde\cX_\fM^\natural\to\Sh(\bG,\bK_\fM).
\end{align}

\begin{lem}\label{le:shimura_compare}
The morphism \eqref{eq:shimura4} factorizes through $\cX^\natural_\fM$, and
the induced morphism $\cX^\natural_\fM\to\Sh(\bG,\bK_\fM)$ is an isomorphism.
In particular, the scheme $\cX^\natural_\fM$ is smooth over
$\dZ[1/\Nm_{F/\dQ}\fM\cdot\disc(F)]$.
\end{lem}

\begin{proof}
It suffices to check that \eqref{eq:shimura3} factorizes through
$\cX_\fM^{(\fD^{-1},(\fD^{-1})^+)}$ and the induced morphism
$\cX_\fM^{(\fD^{-1},(\fD^{-1})^+)}\to\Sh(\bG,\bK_\fM)$ is an open and closed
immersion.

For the first part, by canonicality, we only need to check on the generic
fiber. For this, we only need to check for $\dC$-points. Then the claim
follows since \eqref{eq:shimura2} is simply the quotient by $U_F$ onto its
image (see \cite{Gor02}*{Theorem 2.17, Remark 2.18}).

For the second part, it suffices to show that for each prime $\ell\nmid
\Nm_{F/\dQ}\fM\cdot\disc(F)$, the morphism
$\cX_{\fM;\dZ_{(\ell)}}^{(\fD^{-1},(\fD^{-1})^+)}\to\Sh(\bG,\bK_\fM)_{\dZ_{(\ell)}}$
is an open and closed immersion. We note that $\widetilde\bG$ and $\bG$ have
the same connected Shimura data. In particular, the induced morphism of
neutral connected components of
$\widetilde\cX_{\fM;\rW(\dF_\ell^\ac)}^{(\fD^{-1},(\fD^{-1})^+)}$ and
$\Sh(\bG,\bK_\fM)_{\rW(\dF_\ell^\ac)}$ is an isomorphism. Then the base
change of \eqref{eq:shimura3} to $\rW(\dF_\ell^\ac)$ is finite \'{e}tale, of
degree $\# U_F$. Therefore, the induced morphism
$\cX_{\fM;\dZ_{(\ell)}}^{(\fD^{-1},(\fD^{-1})^+)}\to\Sh(\bG,\bK_\fM)_{\dZ_{(\ell)}}$
is an open and closed immersion.
\end{proof}

\begin{remark}
Let $k$ be an algebraically closed field. The set $\cX^\natural_\fM(k)$
corresponds canonically to the set of isomorphism classes of triples
$(A,\iota,C)$ where $A$ is a $k$-abelian surface, $\iota$ is an
$O_F$-multiplication, and $C$ is a level-$\fM$ structure.
\end{remark}

\begin{notation}\label{no:hilbert_surface}
Denote by $X^\natural_\fM=\cX^\natural_{\fM;\dQ}$ the generic fiber of
$\cX^\natural_\fM$. Let $X_\fM$ be the minimal resolution of the Baily--Borel
compactification of $X^\natural_\fM$. It is a smooth projective surface over
$\dQ$ and we have $\pi_0(X_\fM)\simeq\Cl(F)^+$.
\end{notation}

\subsubsection{Hecke algebra and cohomology}

\begin{lem}\label{le:vanishing}
The singular cohomology group $\rH^3(X_\fM(\dC),\dZ)$ vanishes.
\end{lem}

\begin{proof}
This follows from \cite{vdG}*{Theorem 6.1} and the Poincar\'{e} duality for
compact manifolds.
\end{proof}

For ideals $\fd'\mid \fd\mid\fM$ of $O_F$, we have degeneracy morphisms
\[\gamma_\fM^{(\fd,\fd')}\colon X_\fM\to X_{\fM/\fd}\]
such that $\gamma_\fM^{(\fd,1)}$ is the natural projection by forgetting the
level structure. For each finite place $\fv$ of $F$ that is coprime to
$\Nm_{F/\dQ}\fM\cdot\disc(F)$, we have the Hecke correspondence
$S_\fM^{(\fv)}\in\Corr(X_\fM)$ as
\[\xymatrix{
& X_{\fM\fv} \ar[ld]_-{\gamma_{\fM\fv}^{(\fv,1)}} \ar[rd]^-{\gamma_{\fM\fv}^{(\fv,\fv)}} \\
X_\fM && X_\fM.
}\]

Denote by $\dS_\fM^{\sph,\rat}\subset\Corr(X_\fM)$ the (commutative)
$\dQ$-subalgebra generated by $S_\fM^{(\fv)}$ for all places $\fv$ that are
coprime to $\Nm_{F/\dQ}\fM\cdot\disc(F)$. Denote by $\dS_\fM^{\sph,\dr}$ the
image of $\dS_\fM^{\sph,\rat}$ under the homomorphism $\cor_\dr$, which is a
finite-dimensional commutative $\dQ$-algebra. For each
$\Gal(\dC/\dQ)$-conjugacy class $\pi$ of irreducible cuspidal automorphic
representations of $\Res_{F/\dQ}\GL_2(\dA)$ that occurs in
$\rH^2_\dr(X_\fM)$, we have a unique idempotent
$\sP^\dr_\pi\in\dS_\fM^{\sph,\dr}$ such that it induces the natural
projection $\bigoplus\rH^*_\dr(X_\fM)\to\rH^2_\dr(X_\fM)[\pi]$ onto the
$\pi$-isotypic subspace $\rH^2_\dr(X_\fM)[\pi]$ of $\rH^2_\dr(X_\fM)$.

\begin{definition}\label{de:projector_surface}
A \emph{$\pi$-projector} (of level $\fM$) is an element $\sP_\pi$ in
$\dS_\fM^{\sph,\rat}$ such that $\cor_\dr(\sP_\pi)=\sP^\dr_\pi$.
\end{definition}

Consider the $p$-adic \'{e}tale cohomology
$\rH^2_{\et}(X_{\fM;\dQ^\ac},\dQ_p(1))$ as a $p$-adic representation of
$\Gamma_\dQ$. Denote by $\rH^2_\pi(Y_{\fM;\dQ^\ac})$ its $\pi$-isotypic
subspace which is preserved under the action of $\Gamma_\dQ$. Then
$\sP^\dr_\pi$ induces a natural projection
\[\sP^{\et}_\pi\colon\rH^2_{\et}(X_{\fM;\dQ^\ac},\dQ_p(1))\to\rH^2_\pi(X_{\fM;\dQ^\ac})\]
by the Comparison Theorem between de Rham and \'{e}tale cohomology. It is
$\Gamma_\dQ$-equivariant.

\subsubsection{Supersingular locus}

Let $\ell$ be a prime not dividing $2\fM$ and \emph{inert} in $F$. We recall
some constructions in \cite{GO00}.

For each geometry point $s\in\cX^\natural_{\fM;\dF_\ell}(k)$, denote the
associated abelian variety by $A_s$. As in the curve case, denote by
$\tau(s)$ the multi-set of characters of $O_F/\ell$ appearing in its action
on the Dieudonn\'{e} module of $A_s[\ell]_\alpha$. Now there are four
possibilities for $\tau(s)$: $\emptyset$, $\{\tau^\bullet_\ell\}$,
$\{\tau^\circ_\ell\}$ and $\{\tau^\bullet_\ell,\tau^\circ_\ell\}$.

Denote by $\cX^\ssp_{\fM;\dF_{\ell^2}}$ (resp.\
$\cX^\bullet_{\fM;\dF_{\ell^2}}$, $\cX^\circ_{\fM;\dF_{\ell^2}}$) the closed
subscheme of $\cX^\natural_{\fM;\dF_{\ell^2}}$, with the reduced induced
structure, having the property that for every geometric point $s$ of it, the
set $\tau(s)$ contains $\{\tau^\bullet_\ell,\tau^\circ_\ell\}$ (resp.\
$\{\tau^\bullet_\ell\}$, $\{\tau^\circ_\ell\}$).

\begin{proposition}\label{pr:surface_supersingular}
We have that
\begin{enumerate}
  \item $\cX^\bullet_{\fM;\dF_{\ell^2}}$ and
      $\cX^\circ_{\fM;\dF_{\ell^2}}$ intersect transversally at
      $\cX^\ssp_{\fM;\dF_{\ell^2}}$;

  \item $\cX^\ssp_{\fM;\dF_{\ell^2}}$ is a disjoint union of
      $\Spec\dF_{\ell^2}$, \emph{canonically} indexed by $\cS_{\fM\ell}$,
      under which $\Fr_\ell$ acts by $\op_\ell$;

  \item both $\cX^\bullet_{\fM;\dF_{\ell^2}}$ and
      $\cX^\circ_{\fM;\dF_{\ell^2}}$ are disjoint unions of
      $\bP^1_{\dF_{\ell^2}}$, each \emph{canonically} indexed by
      $\cS_\fM$, under which $\Fr_\ell$ acts by switching the two sets
      and preserving values in $\cS_\fM$;

  \item the self-intersection number (inside
      $\cX_{\fM;\dF_{\ell^2}}^\natural$) of each irreducible component of
      $\cX^\bullet_{\fM;\dF_{\ell^2}}$ or $\cX^\circ_{\fM;\dF_{\ell^2}}$
      is $-2\ell$.
\end{enumerate}
\end{proposition}

\begin{proof}
Part (1) is proved in \cite{GO00}*{Corollary 2.3.7}. Part (2) follows from
Proposition \ref{pr:surface_superspecial}, since
$\cX^\ssp_{\fM;\dF_{\ell^2}}$ is of dimension $0$ and parameterizes
superspecial mixed triples $(A,\iota,C)$.

For (3), take an algebraically closed field $k$ containing $\dF_{\ell^2}$. By
\cite{BG99}*{Theorem 5.1}, every irreducible component of
$\cX^\bullet_{\fM;k}$ (resp.\ $\cX^\circ_{\fM;k}$) is the base of a
Moret-Bailly family associated to a $\bullet$-pure (resp.\ $\circ$-pure)
triple $(A,\iota,C)$, unique up to isomorphism, over $k$. Thus by Proposition
\ref{pr:surface_supersingular}, both $\cX^\bullet_{\fM;k}$ and
$\cX^\circ_{\fM;k}$ are disjoint unions of $\bP^1_k$, each \emph{canonically}
indexed by $\cS_\fM$. It is clear from the construction that $\Fr_\ell$ acts
by switching the two factors. In particular, the Frobenius element
$\Fr_{\ell^2}$ fixes each connected component of $\cX^\bullet_{\fM;k}$ and
$\cX^\circ_{\fM;k}$. Since each component contains an $\dF_{\ell^2}$-point,
it is isomorphic to $\bP^1_{\dF_{\ell^2}}$.

Part (4) is proved in \cite{Nic00}*{\Sec 3.0.2}, which is also a very special
case of the main calculation in the work of Tian--Xiao (see \cite{TX14}*{\Sec
1.1}), in terms of Lemma \ref{le:shimura_compare}.
\end{proof}

\begin{remark}\label{re:degeneracy_surface}~
\begin{enumerate}
  \item In order to distinguish between different orientations, when we
      identify the set $\pi_0(\cX^\bullet_{\fM;\dF_{\ell^2}})$ (resp.\
      $\pi_0(\cX^\circ_{\fM;\dF_{\ell^2}})$) with $\cS_\fM$, we use the
      notation $\cS_\fM^\bullet$ (resp.\ $\cS_\fM^\circ$). As a set, it
      is nothing but $\cS_\fM$.

  \item By Remark \ref{re:surface_superspecial}, the natural map
      \begin{align*}
      &\cX^\ssp_{\fM;\dF_{\ell^2}}(\dF_{\ell^2})\to\pi_0(\cX^\bullet_{\fM;\dF_{\ell^2}})\\
      \text{resp. }&
      \cX^\ssp_{\fM;\dF_{\ell^2}}(\dF_{\ell^2})\to\pi_0(\cX^\circ_{\fM;\dF_{\ell^2}})
      \end{align*}
      is given by $\gamma_{\fM\ell}^{(\ell,1)}$ (resp.\
      $\gamma_{\fM\ell}^{(\ell,\ell)}$); see Notation
      \ref{no:order_quadratic} (2).

  \item Define $\cX^\ssl_{\fM;\dF_{\ell^2}}$ to be the union of
      $\cX^\bullet_{\fM;\dF_{\ell^2}}$ and
      $\cX^\circ_{\fM;\dF_{\ell^2}}$. It is a proper subscheme and
      coincides with the \emph{supersingular locus} of
      $\cX^\natural_{\fM;\dF_{\ell^2}}$.

  \item Similar parametrization in Proposition
      \ref{pr:surface_supersingular} (2,3) is also obtained by Tian and
      Xiao \cite{TX13}*{\Sec 1.5.1}, where they use the language of
      Shimura sets, that is, certain double quotients of
      $(Q\otimes_\dZ\widehat\dZ)^\times$. However, their parametrization
      is not canonical, and the canonicality in our parametrization is
      crucial in this article since we need to use it for different
      primes. See Propositions \ref{pr:special_supersingular} and
      \ref{pr:special_superspecial}.
\end{enumerate}

\end{remark}

\subsection{Hirzebruch--Zagier morphisms}
\label{ss:special_morphisms}

Now we assume that all prime factors of $N^-$ are inert in $F$. In what
follows, we will view $\cY^\natural_{N^+,N^-}$ and $\cY_{N^+,N^-}$ as schemes
over $\dZ[1/N^+\disc(F)]$. There is a natural composite morphism
\[\zeta^\natural_{N^+,N^-}\colon\cY^\natural_{N^+,N^-}\to\widetilde\cX^\natural_{N^+}\to\cX^\natural_{N^+},\]
in which the first morphism sends $(A,\iota,C)$ to
$(A,\iota\res_{O_F},\psi_D,C)$ where $\psi_D$ is the canonical principal
polarization associated to $(\cO_{N^-},\ast)$ (Notation
\ref{no:order_rational} (3)). The morphism $\zeta^\natural_{N^+,N^-}$ is
finite and its image is contained in $\cX_{N^+}^{(\fD^{-1},(\fD^{-1})^+)}$.
We call it the \emph{Hirzebruch--Zagier morphism}, whose generic fiber is an
example of \emph{special morphisms} of Shimura varieties in the sense that it
is induced by embedding of Shimura data.

Let $\cD_{N^+}>0$ be the least square-free integer such that
\begin{itemize}
  \item it is divisible by all prime factors of $N^+\disc(F)$;

  \item $X_{N^+}$ extends to a projective smooth scheme $\cX_{N^+}$ over
      $\dZ[1/\cD_{N^+}]$ such that the morphism $Y_{N^+,1}\to X_{N^+}$
      extends to a finite morphism
      $\cY_{N^+,1;\dZ[1/\cD_{N^+}]}\to\cX_{N^+}$ whose restriction to
      $\cY^\natural_{N^+,1;\dZ[1/\cD_{N^+}]}$ coincides with
      $\zeta^\natural_{N^+,1;\dZ[1/\cD_{N^+}]}$.
\end{itemize}

Therefore, we always have a (unique) morphism
\[\zeta_{N^+,N^-}\colon\cY_{N^+,N^-;\dZ[1/\cD_{N^+}]}\to\cX_{N^+}\]
over $\dZ[1/\cD_{N^+}]$ extending $\zeta^\natural_{N^+,N^-}$. In what
follows, $\cY_{N^+,N^-}$ will automatically be regarded as a scheme over
$\dZ[1/\cD_{N^+}]$ and by abuse of notation, all kinds of base change of
$\zeta_{N^+,N^-}$ will still be denoted by $\zeta_{N^+,N^-}$. Although here
we use the same notation as in Notation \ref{no:order_quadratic}, no
confusion will arise since $\zeta_{N^+,N^-}$ is a morphism of schemes (resp.\
map of sets) if and only if $\wp(N^-)$ is even (resp.\ odd).

Put $\cZ_{N^+,N^-}=\cY_{N^+,N^-}\times_{\Spec\dZ[1/\cD_{N^+}]}\cX_{N^+}$. We
have two projection morphisms
\[\cY_{N^+,N^-}\xleftarrow{\pi_\cY}\cZ_{N^+,N^-}\xrightarrow{\pi_\cX}\cX_{N^+}.\]
Denote by $\b\Delta_{N^+,N^-}$ the graph of $\zeta_{N^+,N^-}$, which is a
codimension $2$ cycle on $\cZ_{N^+,N^-}$. Put
$Z_{N^+,N^-}=\cZ_{N^+,N^-;\dQ}$.

\subsubsection{Case of proper intersection}

Let $\ell$ be a prime not dividing $2N^-\cD_{N^+}$ and inert in $F$.

\begin{lem}\label{le:proper_intersection}
Denote by $\cI$ the scheme-theoretical intersection of the closed subschemes
$\b\Delta_{N^+,N^-;\dF_{\ell^2}}$ and
$\pi_\cX^{-1}\cX_{N^+;\dF_{\ell^2}}^\ssl$ of $\cZ_{N^+,N^-;\dF_{\ell^2}}$.
Then $\pi_\cY\cI$ (resp.\ $\pi_\cX\cI$) is contained in the supersingular
locus $\cY^\ssl_{N^+,N^-;\dF_{\ell^2}}$ (resp.\ superspecial locus
$\cX^\ssp_{N^+;\dF_{\ell^2}}$), and $\cI$ is a disjoint union of
$\Spec\dF_{\ell^2}$.
\end{lem}

\begin{proof}
It is obvious that $\pi_\cY\cI$ is contained in
$\cY^\natural_{N^+,N^-;\dF_{\ell^2}}$. Suppose that $s\in\cI(k)$ is a
geometric point, which corresponds to the data
$(A_s,\iota_s,C_s)\in\cY_{N^+,N^-;\dF_{\ell^2}}(k)$. If $A_s$ is ordinary,
then $\pi_\cX(s)$ does not locate in $\cX_{N^+;\dF_{\ell^2}}^\ssl$, which is
a contradiction. Therefore, the abelian surface $A_s$ is supersingular and
hence superspecial. This means that
$\pi_\cX(s)\in\cX^\ssp_{N^+;\dF_{\ell^2}}(k)$. Thus, The first claim follows.
In particular, $\cI$ is of dimension $0$.

For the second claim, we only need to show that
$\b\Delta_{N^+,N^-;\dF_{\ell^2}}$ and
$\pi_\cX^{-1}\cX_{N^+;\dF_{\ell^2}}^\ssl$ intersect transversally. We use
displays of $p$-divisible groups (in the version developed in \cite{Zin02}).
Since $\ell$ is odd, we may choose an element $\delta\in \rW(\dF_\ell^\ac)$
such that $\delta^2\in\dZ_\ell^\times$ and we have an isomorphism
$O_{F,\ell}\xrightarrow\sim\dZ_\ell[\delta]$ whose composition with the
natural inclusion $\dZ_\ell[\delta]\subset \rW(\dF_\ell^\ac)$ gives
$\tau^\bullet_\ell$. Then $\tau^\circ_\ell(\delta)=-\delta$.  Choose an
isomorphism $\cO_{N^-}\otimes_\dZ\dZ_\ell\simeq\Mat_2(\dZ_\ell)$ such that
\begin{itemize}
  \item the inclusion $O_{F,\ell}\hookrightarrow
      \cO_{N^-}\otimes_\dZ\dZ_\ell$ sends $x+y\delta$ to $\left(
      \begin{array}{cc}
        x & y\delta^2 \\
        y & x \\
      \end{array}
    \right)$;

  \item the involution $\ast$ sends $\left(
      \begin{array}{cc}
        a & b \\
        c & d \\
      \end{array}
    \right)$ to $\left(
      \begin{array}{cc}
        a & c\delta^2 \\
        b\delta^{-2} & d \\
      \end{array}
    \right)$.
\end{itemize}
Since the problem is local, we pick up an arbitrary $\dF_\ell^\ac$-point
$s\in\cI(\dF_\ell^\ac)$ and consider only the abelian surface $A_s$. Suppose
that the $\ell$-divisible group $A_s[\ell^\infty]=G^{\oplus 2}$, and $G$ is
displayed by the matrix $\left(
    \begin{array}{cc}
      a & b \\
      c & d \\
    \end{array}
  \right)
$ under the basis $e_1,e_2$ such that the polarization is given by the
standard symplectic matrix $\left(
  \begin{array}{cc}
     & 1 \\
    -1 &  \\
  \end{array}
\right)$. By the theory of displays, the universal deformation ring of $G$ is
isomorphic to $R_0=\dF_\ell^\ac[[t_0]]$, and the universal deformation is
displayed by the matrix $\left(
    \begin{array}{cc}
      a+cT_0 & b+dT_0 \\
      c & d \\
    \end{array}
  \right)
$, where $T_0\in \rW(R_0)$ is the Teichm\"{u}ller lift of $t_0$. Now
$A_s[\ell^\infty]$ is displayed by the matrix
\[\left(
    \begin{array}{cccc}
      a &  & b &  \\
       & a &  & b \\
      c &  & d &  \\
       & c &  & d \\
    \end{array}
  \right),
\] with the (principal) polarization given by the symplectic matrix
\[\left(
    \begin{array}{cccc}
       &  & 1 &  \\
       &  &  & \delta \\
      -1 &  &  &  \\
       & -\delta &  &  \\
    \end{array}
  \right).
\] The universal deformation ring
of $A_s[\ell^\infty]$ is isomorphic to
$R=\dF_\ell^\ac[[t_{11},t_{12},t_{21},t_{22}]]$, and the universal
deformation is displayed by the matrix
\[\left(
    \begin{array}{cccc}
      a+cT_{11} & cT_{12}  & b+dT_{11} & dT_{12} \\
      cT_{21} & a+cT_{22} & dT_{21} & b+dT_{22} \\
      c &  & d &  \\
       & c &  & d \\
    \end{array}
  \right),
\] where $T_{ij}\in \rW(R)$ is the Teichm\"{u}ller lift of $t_{ij}$. We have that
\begin{itemize}
  \item the universal deformation ring of $A_s[\ell^\infty]$ with the
      above given polarization is cut out by the ideal generated by
      $t_{12}-\delta^2t_{21}$;

  \item the universal deformation ring of $A_s[\ell^\infty]$ with the
      $O_{F,\ell}$-action and the above given polarization is cut out by
      the ideal generated by $\{t_{11}-t_{22},t_{12}-\delta^2t_{21}\}$;

  \item the universal deformation ring of $A_s[\ell^\infty]$ with the
      $\Mat_2(\dZ_\ell)$-action and the above given polarization is cut
      out by the ideal generated by $\{t_{11}-t_{22},t_{12},t_{21}\}$.
\end{itemize}
Moreover, by the computation in \cite{GO00}*{Theorem 2.3.4}, the universal
deformation ring of $A_s[\ell^\infty]$ with the $O_{F,\ell}$-action and the
above given polarization along $\cX^\bullet_{N^+;\dF_\ell^\ac}$ is cut out by
the ideal generated by $\{t_{11}-t_{22},t_{12}-\delta^2t_{21},t_{11}-\delta
t_{21}\}$. Therefore, the local ring of
$\b\Delta_{N^+,N^-;\dF_{\ell^2}}\times_{\cZ_{N^+,N^-;\dF_{\ell^2}}}\pi_\cX^{-1}\cX_{N^+;\dF_{\ell^2}}^\bullet$
at $s$ is isomorphic to
\[\dF_\ell^\ac[[t_0,t_{11},t_{12},t_{21},t_{22}]]/\langle
t_0-t_{11},t_{11}-t_{22},t_{12},t_{21},t_{11}-\delta t_{21}\rangle,\] which
is isomorphic to $\dF_\ell^\ac$. This implies that
$\b\Delta_{N^+,N^-;\dF_{\ell^2}}$ and
$\pi_\cX^{-1}\cX_{N^+;\dF_{\ell^2}}^\bullet$ intersect transversally. We have
a similar argument for $\cX_{N^+;\dF_{\ell^2}}^\circ$. The lemma then
follows.
\end{proof}

\begin{proposition}\label{pr:special_supersingular}
The restriction of the map $\zeta_{N^+,N^-}(\dF_{\ell^2})$ to
$\cY^\ssl_{N^+,N^-;\dF_{\ell^2}}(\dF_{\ell^2})=\cT_{N^+,N^-\ell}$ takes
values in the subset $\cX^\ssp_{N^+;\dF_{\ell^2}}(\dF_{\ell^2})$. If we
compose it with the natural map from
$\cX^\ssp_{N^+;\dF_{\ell^2}}(\dF_{\ell^2})$ to
$\pi_0(\cX^\bullet_{N^+;\dF_{\ell^2}})=\cS_{N^+}^\bullet$ (resp.\
$\pi_0(\cX^\circ_{N^+;\dF_{\ell^2}})=\cS_{N^+}^\circ$), then the resulting
map is equal to
\begin{align*}
\cT_{N^+,N^-\ell}&\xrightarrow{\zeta_{N^+,N^-\ell}}\cS_{N^+}=\cS_{N^+}^\bullet, \\
\text{resp.\ }\cT_{N^+,N^-\ell}&\xrightarrow{\zeta_{N^+,N^-\ell}\circ\op_\ell}\cS_{N^+}=\cS_{N^+}^\circ.
\end{align*}
\end{proposition}

\begin{proof}
We prove the first case, by which the second one will follow. Suppose that
$k$ is an algebraically closed field containing $\dF_{\ell^2}$, and
$(A,\iota,C)$ is a triple in $\cY^\ssl_{N^+,N^-;\dF_{\ell^2}}(k)$, which
implies that $(A,\iota\res_{O_F},C)$ is superspecial mixed. We only need to
show that if $R=\End(A,\iota,C)$ with the natural orientation as in the proof
of Proposition \ref{pr:curve_supersingular}, then
$\End(A/H_\bullet,\iota\res_{O_F},C)$ is the order $R_\sharp$ as in
Definition \ref{de:special}. Note that $R\otimes_\dZ O_F$ is naturally
contained in $\End(A/H_\bullet,\iota\res_{O_F},C)$.

It amounts to show that for every prime $v\mid N^-\ell$, the
$(R/v,O_F/v)$-bimodule
$\End(A/H_\bullet,\iota\res_{O_F},C)\otimes_\dZ\dZ_v/R\otimes_\dZ
O_{F,v}\simeq\dF_{v^2}$ is isomorphic to $o_v\otimes\tau^\bullet_v$. When
$v=\ell$, it holds since $O_F/\ell$ acts on the Dieudonn\'{e} module of
$H_\bullet$ by $\tau_\ell^\bullet$. When $v\mid N^-$, the question is local.
We identify $O_{F,v}$ with $\dZ_{v^2}$ via $\tau_v^\bullet$, and
$\cO_{N^-,v}\colonequals\cO_{N^-}\otimes_\dZ\dZ_v$ with
$\dZ_{v^2}\oplus\dZ_{v^2}[\Pi]$ such that $\Pi^2=v$ and $a\Pi=\Pi a^\theta$
for $a\in\dZ_{v^2}$. We identify $\rT_v(A)$ with $\cO_{N^-,v}$ such that
$\iota$ acts by left multiplication. In particular, the algebra
$R_v\simeq\cO_{N^-,v}$ acts on $\rT_v(A)$ via right multiplication. The
algebra $\End(A/H_\bullet,\iota\res_{O_F},C)_v$ is identified with the
centralizer of $\dZ_{v^2}$ in $\End_{\dZ_v}(\rT_v(A))$, where the latter is
identified with $\Mat_2(\dZ_{v^2})$ under the basis $\{1,\Pi\}$. Under the
same basis, the suborder $R_v\otimes_{\dZ_v}\dZ_{v^2}$ is identified with
\[\left\{\left(
           \begin{array}{cc}
             a & b \\
             c & d \\
           \end{array}
         \right)\in\Mat_2(\dZ_{v^2})\res c\in v\dZ_{v^2}\right\}.\] Moreover,
the action of $a\otimes b\in\dZ_{v^2}\otimes_{\dZ_v}\dZ_{v^2}\subset
R_v\otimes_{\dZ_v}\dZ_{v^2}$ on $\End(A/H_\bullet,\iota\res_{O_F},C)_v$ is
given by
\[x\in\Mat_2(\dZ_{v^2})\mapsto \left(
           \begin{array}{cc}
             b &  \\
              & b \\
           \end{array}
         \right)x\left(
           \begin{array}{cc}
             a &  \\
              & a^\theta \\
           \end{array}
         \right).\] This implies that the bi-action of $(R/v,O_F/v)$ on
$\End(A/H_\bullet,\iota\res_{O_F},C)\otimes_\dZ\dZ_v/R\otimes_\dZ O_{F,v}$ is
$o_v\otimes\tau^\bullet_v$.
\end{proof}

\subsubsection{Case of improper intersection}

Let $\ell$ be a prime that divides $N^-$ but not $2\cD_{N^+}$, and is inert
in $F$.

\begin{proposition}\label{pr:special_superspecial}
We have that
\begin{enumerate}
  \item The restriction of the map $\zeta_{N^+,N^-}(\dF_{\ell^2})$ to
      $\cY^\ssp_{N^+,N^-;\dF_{\ell^2}}(\dF_{\ell^2})$ takes values in
      $\cX^\ssp_{N^+;\dF_{\ell^2}}(\dF_{\ell^2})$. Regarded as a map from
      $\cT_{N^+\ell,N^-/\ell}$ to $\cS_{N^+\ell}$, it is equal to
      \begin{align*}
      \cT_{N^+\ell,N^-/\ell}\xrightarrow{\zeta_{N^+\ell,N^-/\ell}}\cS_{N^+\ell}.
      \end{align*}

  \item For $?=\bullet,\circ$, the restriction of the morphism
      $\zeta_{N^+,N^-}$ to $\cY^?_{N^+,N^-;\dF_{\ell^2}}$ has image
      contained in $\cX^?_{N^+;\dF_{\ell^2}}$. The induced map
      $\pi_0(\cY^?_{N^+,N^-;\dF_{\ell^2}})\to\pi_0(\cX^?_{N^+;\dF_{\ell^2}})$
      is equal to
      \begin{align*}
      \cT_{N^+,N^-/\ell}\xrightarrow{\zeta_{N^+,N^-/\ell}}\cS_{N^+}.
      \end{align*}
\end{enumerate}
\end{proposition}

\begin{proof}
We first prove (2), for $?=\bullet$ and the other is similar. As before, we
fix an algebraically closed field $k$ containing $\dF_{\ell^2}$. By
\cite{Rib89}*{\Sec 5}, there is a bijection between
$\pi_0(\cY^?_{N^+,N^-;\dF_{\ell^2}})$ and the set of isomorphism classes of
triples $(A,\iota,C)$ over $k$ such that $(A,\iota\res_{O_F},C)$ is
$\bullet$-pure. By the same proof of Proposition
\ref{pr:special_supersingular}, we know that if $R=\End(A,\iota,C)$ with the
natural orientation as in the proof of Proposition
\ref{pr:curve_supersingular}, then $\End(A,\iota\res_{O_F},C)$, which
naturally contains $R\otimes_\dZ O_F$, is the order $R_\sharp$ as in
Definition \ref{de:special}. In particular, (2) follows.

Part (1) is proved in the same way as (2), with the only extra work of the
orientation at $\ell$. But this follows from (2), Remark
\ref{re:degeneracy_superspecial} (2) and Remark \ref{re:degeneracy_surface}
(2).
\end{proof}

\begin{remark}
The above proposition together with its proof amounts to say that the natural
commutative diagram
\[\xymatrix{
\pi_0(\cY^\bullet_{N^+,N^-;\dF_{\ell^2}}) \ar[d]& \cY^\ssp_{N^+,N^-;\dF_{\ell^2}}(\dF_{\ell^2})
\ar[r]\ar[l]\ar[d]& \pi_0(\cY^\circ_{N^+,N^-;\dF_{\ell^2}}) \ar[d] \\
\pi_0(\cX^\bullet_{N^+;\dF_{\ell^2}}) & \cX^\ssp_{N^+;\dF_{\ell^2}}(\dF_{\ell^2})
\ar[r]\ar[l]& \pi_0(\cX^\circ_{N^+;\dF_{\ell^2}})
}\] is canonically identified with
\[\xymatrix{
\cT^\bullet_{N^+,N^-/\ell} \ar[d]_-{\zeta_{N^+,N^-/\ell}}&&
\cT_{N^+\ell,N^-/\ell} \ar[rr]^-{\delta_{N^+\ell,N^-/\ell}^{(\ell,\ell)}}\ar[ll]_-{\delta_{N^+\ell,N^-/\ell}^{(\ell,1)}}
\ar[d]^-{\zeta_{N^+\ell,N^-/\ell}} &&
\cT^\circ_{N^+,N^-/\ell} \ar[d]^-{\zeta_{N^+,N^-/\ell}} \\
\cS^\bullet_{N^+} && \cS_{N^+\ell}\ar[rr]^-{\gamma_{N^+\ell}^{(\ell,\ell)}}\ar[ll]_-{\gamma_{N^+\ell}^{(\ell,1)}}
 &&  \cS^\circ_{N^+}.
}\]
\end{remark}

\section{Construction of cohomology classes}
\label{s3}

In this chapter, we construct various cohomology classes from
Hirzebruch--Zagier cycles. In \Sec \ref{ss:bloch_kato}, we recall the
definition of Bloch--Kato Selmer groups of general Galois modules with
various coefficients, and collect some facts for later use. In \Sec
\ref{ss:twisted_triple}, we set up our main theorem in the language of
automorphic representations, and list assumptions which will be needed at
various stages in the later discussion. In \Sec \ref{ss:gross_kudla}, we
introduce Hirzebruch--Zagier cycles in details, from which we construct the
very first class in the Selmer group with rational coefficient. The main
theorem in the language of automorphic representations will be stated at the
end of this section. In \Sec \ref{ss:testing_factors}, we fix suitable
projections, which are related to non-vanishing testing vectors of local
invariant forms. In \Sec \ref{ss:raising_levels}, we collect some important
facts from Bertolini--Darmon \cite{BD05}, with slight modification to fit our
need. In \Sec \ref{ss:construction_annihilators}, we construct variants of
Hirzebruch--Zagier classes with torsion coefficients, as candidates for
annihilators of Selmer groups. In particular, we introduce the notion of
strongly admissible primes.

\subsection{Bloch--Kato Selmer group}
\label{ss:bloch_kato}

We recall the definition of Bloch--Kato Selmer groups for Galois
representations with various coefficients. For simplicity, our number field
will be $\dQ$, and $p$ is odd according to our convention.

Let $(\rho,\rV)$ be a $p$-adic representation of $\Gamma_\dQ$, that is, as we
recall $\rV$ is a finite dimensional $\dQ_p$-vector space and
\[\rho\colon\Gamma_\dQ\to\GL(\rV)\]
is a homomorphism, continuous with respect to the profinite (resp.\ $p$-adic)
topology on the source (resp.\ target). Define
\[\rH^1_f(\dQ_p,\rV)=\Ker[\rH^1(\dQ_p,\rV)\to\rH^1(\dQ_p,\rV\otimes_{\dQ_p}\rB_{\r{cris}})].\]
It is known that if $\rV$ is crystalline, then elements in
$\rH^1_f(\dQ_p,\rV)$ correspond exactly to those extensions
\[0\to\rV\to\rV_1\to\dQ_p\to0\]
in which $\rV_1$ is also crystalline.

For each finite set $\Sigma$ of places of $\dQ$ containing $\{p,\infty\}$,
denote by $\dQ^\Sigma$ the maximal subfield of $\dQ^\ac$ that is unramified
outside of $\Sigma$, and $\Gamma_{\dQ,\Sigma}=\Gal(\dQ^\Sigma/\dQ)$ which is
a quotient group of $\Gamma_\dQ$.

\begin{definition}[Bloch--Kato Selmer group, \cite{BK90}]\label{de:selmer_group}
Suppose that there exists a finite set $\Sigma$ of places of $\dQ$ containing
$\{p,\infty\}$ such that $\rV$ is unramified outside $\Sigma$. We define the
\emph{Bloch--Kato Selmer group} $\rH^1_f(\dQ,\rV)$ of $\rV$ as the subspace
of $\rH^1(\dQ,\rV)$ consisting of $s$ such that
\begin{itemize}
  \item $s$ belongs to the subspace $\rH^1(\Gamma_{\dQ,\Sigma},\rV)$;

  \item $\loc_v(s)\in\rH^1_\unr(\dQ_v,\rV)$ for all primes $v\nmid p$;
      and

  \item $\loc_p(s)\in\rH^1_f(\dQ_p,\rV)$.
\end{itemize}
It is clear that $\rH^1_f(\dQ,\rV)$ does not depend on the choice of
$\Sigma$. If $\rT\subset\rV$ is a stable lattice, then we have a natural map
$\rH^1(\dQ,\rT)\to\rH^1(\dQ,\rV)$ and define the Bloch--Kato Selmer group
$\rH^1_f(\dQ,\rT)$ of $\rT$ to be the inverse image of $\rH^1_f(\dQ,\rV)$
under the above map. For each integer $n\geq 1$, define
$\rH^1_f(\dQ,\bar\rT^n)$ to be the image of $\rH^1_f(\dQ,\rT)$ under the
natural map $\rH^1(\dQ,\rT)\to\rH^1(\dQ,\bar\rT^n)$.
\end{definition}

\begin{remark}
By \cite{Mil06}*{Corollary 4.15} and \cite{Tat76}*{Corollary 2.1}, the groups
$\rH^1(\Gamma_{\dQ,\Sigma},\rT)$ and hence $\rH^1_f(\dQ,\rT)$,
$\rH^1_f(\dQ,\bar\rT^n)$ are all finitely generated $\dZ_p$-modules.
\end{remark}

\begin{definition}
We say that a $p$-adic representation $\rV$ of $\Gamma_\dQ$ is \emph{tamely
pure} if for all primes $v\nmid p$, $\rH^1(\dQ_v,\rV)=0$. If $\rV$ is tamely
pure, then $s\in\rH^1_f(\dQ,\rV)$ if and only if
$\loc_p(s)\in\rH^1_f(\dQ_p,\rV)$.
\end{definition}

\begin{lem}\label{le:unramified_torsion}
Let $\rT$ be a stable lattice in a tamely pure $p$-adic representation $\rV$
of $\Gamma_\dQ$. Then for each prime $v\nmid p$ such that $\rI_v$ acts
trivially on $\rV$, the image of the reduction map
$\rH^1(\dQ_v,\rT)\to\rH^1(\dQ_v,\bar\rT^n)$ is contained in
$\rH^1_\unr(\dQ_v,\bar\rT^n)$ for every integer $n\geq 1$.
\end{lem}

\begin{proof}
The short exact sequence $0\to\rT\to\rV\to\rV/\rT\to0$ induces the following
commutative diagram
\[\xymatrix{
\rV^{\Gamma_v} \ar[r]\ar[d]& (\rV/\rT)^{\Gamma_v} \ar[r]\ar[d]& \rH^1(\Gamma_v,\rT) \ar[r]\ar[d]& \rH^1(\Gamma_v,\rV) \ar[d] \\
\rV^{\rI_v} \ar[r]& (\rV/\rT)^{\rI_v} \ar[r]& \rH^1(\rI_v,\rT) \ar[r]&
\rH^1(\rI_v,\rV). }\] Since $\rI_v$ acts trivially on $\rT$, the arrow
$\rV^{\rI_v}\to(\rV/\rT)^{\rI_v}$ is surjective, which implies that
$\rH^1(\rI_v,\rT)\to\rH^1(\rI_v,\rV)$ is injective. Since
$\rH^1(\Gamma_v,\rV)=0$, the map $\rH^1(\Gamma_v,\rT)\to\rH^1(\rI_v,\rT)$ is
trivial. The lemma follows immediately.
\end{proof}

\begin{lem}\label{le:tamely_pure}
The $p$-adic representation $\rH^3_{\et}(Z_{N^+M,N^-;\dQ^\ac},\dQ_p(2))$ is
tamely pure. In other words, for each prime $v\nmid p$, we have
$\rH^1(\dQ_v,\rH^3_{\et}(Z_{N^+M,N^-;\dQ^\ac},\dQ_p(2)))=0$.
\end{lem}

This lemma actually can be avoided for the proof of main theorems (see Lemma
\ref{le:semisimple}). However, we include it here as a typical example for
tamely pure representations.

\begin{proof}
The lemma follows from the purity of monodromy filtration for
$Z_{N^+M,N^-;\dQ_v}$, which is known. In fact, since $Z_{N^+M,N^-;\dQ_v}$ is
the product of a smooth projective curve and a smooth projective surface over
$\dQ_v$, the purity follows from that of the latter two by K\"{u}nneth
formula and \cite{Del80}*{(1.6.9)}. The case for curves is proved in
\cite{SGA}*{Expos\'{e} IX}, while the case of surfaces is proved by
Rapoport--Zink \cite{RZ82} together with de Jong's alteration \cite{dJ96}.
\end{proof}

\subsection{Twisted triple product Galois representation}
\label{ss:twisted_triple}

We now switch to the language of automorphic representations. Let $\sigma$
(resp.\ $\pi$) be an irreducible cuspidal automorphic representation of
$\GL_2(\dA)$ (resp.\ $\Res_{F/\dQ}\GL_2(\dA)$) defined over $\dQ$, of weight
$2$ (resp.\ $(2,2)$) and trivial central character, with conductor $N$
(resp.\ $\fM$). We also assume that they are both non-dihedral.

Let $p$ be an (odd) prime, which will be fixed. We may associate to $\sigma$
(resp.\ $\pi$) a $p$-adic representation
$\rho_\sigma\colon\Gamma_\dQ\to\GL(\rV_\sigma)$ (resp.\
$\rho_\pi\colon\Gamma_F\to\GL(\rV_\pi)$) of $\Gamma_\dQ$ (resp.\ $\Gamma_F$)
of dimension $2$.

\begin{notation}
Put
\[\pres{\sharp}{\rho}_\pi=(\otimes\Ind^{\Gamma_\dQ}_{\Gamma_F}\rho_\pi)(-1)\]
as the tensor induction (with the Tate twist), with the underlying
$\dQ_p$-vector space $\pres{\sharp}{\rV}_\pi$ which is isomorphic to
$\rV_\pi^{\otimes 2}$. It is isomorphic to the Asai representation
$\As\rho_\pi$ associated to $\rho_\pi$ (with the Tate twist).
\end{notation}

There is a perfect $\Gamma_\dQ$-equivariant symmetric bilinear pairing
$\pres{\sharp}{\rV}_\pi\times\pres{\sharp}{\rV}_\pi\to\dQ_p$ coming from the
construction. Denote by $\rO(\pres{\sharp}{\rV}_\pi)$ the orthogonal group
defined by the above pairing. Then we have a homomorphism
\[\pres{\sharp}{\rho}_\pi\colon\Gamma_\dQ\to\rO(\pres{\sharp}{\rV}_\pi).\]
Finally, put
$(\rho_{\sigma,\pi},\rV_{\sigma,\pi})=(\rho_\sigma,\rV_\sigma)\otimes_{\dQ_p}(\pres{\sharp}{\rho}_\pi,\pres{\sharp}{\rV}_\pi)$.

\subsubsection{Asai-decomposable case}
\label{sss:asai}

Recall that $\pi$ has trivial central character and is non-dihedral.

\begin{definition}\label{de:asai_decomposable}
We say that $\pi$ is \emph{Asai-decomposable} if
$\pi^\theta\simeq\pi\otimes\omega$ for some (necessarily quadratic)
automorphic character $\omega$ of $\Res_{F/\dQ}\bG_{m,F}(\dA)$, where
$\pi^\theta=\pi\circ\theta$.
\end{definition}

If $\pi$ is Asai-decomposable, then we have a decomposition
\[\pres{\sharp}{\rV}_\pi=\rV_{\breve\pi}(-1)\oplus\dQ_p(\breve\eta)\]
of $p$-adic representations of $\Gamma_\dQ$, where the latter two direct
summands are both absolutely irreducible. Here, $\breve\eta$ is a Dirichlet
character of $\Gamma_\dQ$. If we denote by $\omega_{\breve\eta}$ the
associated automorphic character of $\dA^\times$ via the global class field
theory, then $\omega_{\breve\eta}\circ\Nm_{F/\dQ}=\omega$. By
\cite{LR98}*{Theorem 1}, the restricted character $\omega\res_{\dA^\times}$
is trivial. Therefore, the characters $\omega_{\breve\eta}$ and hence
$\breve\eta$ are either trivial or quadratic. Denote by
$\breve{F}\subset\dQ^\ac$ the splitting field of $\breve\eta$, which is
either $\dQ$ or a quadratic field unramified outside
$\Nm_{F/\dQ}\fM\cdot\disc(F)$.

Put
$(\rho_{\sigma\times\breve\pi},\rV_{\sigma\times\breve\pi})=\rV_\sigma\otimes_{\dQ_p}\rV_{\breve\pi}(-1)$
and
$(\rho_{\sigma\otimes\breve\eta},\rV_{\sigma\otimes\breve\eta})=\rV_\sigma\otimes_{\dQ_p}\dQ_p(\breve\eta)$.
We adopt the convention that $\breve\eta$ is the trivial character if $\pi$
is not Asai-decomposable.

\subsubsection{Assumptions}

We write $N=N^+N^-$ such that $N^+$ (resp.\ $N^-$) is the product of prime
factors that are split (resp.\ inert) in $F$. Fix an integer $M\geq 1$ such
that $\fM\mid M$.

\begin{assumption}[Group \textbf{R}]\label{as:group_r}
We introduce the following assumptions on $\sigma$, $\pi$ and $p$.
\begin{description}
  \item[(R1)] $N^-$ is square-free; $p$, $N$, $M\disc(F)$ are pairwise
      coprime; and $N^+M$ is neat for the Hilbert modular surface
      associated to $F$ and the Shimura curve associated to an
      \emph{arbitrary} indefinite quaternion algebra $B_{N^\sim}$ with
      $N^\sim$ coprime to $N^+M$;

  \item[(R2)] $\rho_\sigma$ is residually surjective. In other words, for
      a stable lattice $\rT_\sigma$ of $\rV_\sigma$, $\bar\rho_\sigma$ is
      surjective (which also implies that $\rT_\sigma$ is unique up to
      homothety);

  \item[(R3)] $\bar\rho_\sigma$ is ramified at all primes $\ell\mid N^-$
      such that $p\mid \ell^2-1$;

  \item[(R4)] for \emph{both} $\epsilon=+,-$, there exists one strongly
      $(1,\epsilon)$-admissible prime (see Definition \ref{de:group_s}
      and Remark \ref{re:strongly_admissible});

  \item[(R5)] $p\geq 11$;

  \item[(R6)] either (and at most one) of the following two holds
  \begin{description}
    \item[(a)] $\rho_{\sigma,\pi}$ is absolutely irreducible and its
        image contains a non-trivial scalar matrix of finite order
        coprime to $p$; or
    \item[(b)] $\pi$ is Asai-decomposable, such that
        $\rho_{\sigma\times\breve\pi}$ is absolutely irreducible and
        its image contains a non-trivial scalar element of finite
        order coprime to $p$.
  \end{description}
\end{description}
\end{assumption}

In what follows, we will regard $\sigma$ as its Jacquet--Langlands
correspondence to $B_{N^-}^\times$; in particular, it is an automorphic
representation of $B_{N^-}^\times(\dA)$ defined over $\dQ$.

\subsubsection{A lemma on semisimplicity}

\begin{lem}\label{le:semisimple}
The $p$-adic representation $\rH^2_\pi(X_{\fM;\dQ^\ac})$ of $\Gamma_\dQ$ is
isomorphic to $\pres{\sharp}\rV_\pi$. In particular, the representation
$\rV_{\sigma,\pi}$ is tamely pure by Lemma \ref{le:tamely_pure}.
\end{lem}

In fact, one can prove the last assertion without using Lemma
\ref{le:tamely_pure}.

\begin{proof}
By \cite{BL84}, the semisimplification of $\rH^2_\pi(X_{\fM;\dQ^\ac})$ is
isomorphic to $\pres{\sharp}\rV_\pi$. If $\pi$ is not Asai-decomposable, then
$\pres{\sharp}\rV_\pi$ is irreducible since $\pi$ is not dihedral and we are
done. If $\pi$ is Asai-decomposable, then $\rH^2_\pi(X_{\fM;\dQ^\ac})$ is
semisimple. In fact, the Poincar\'{e} duality provides
$\rH^2_\pi(X_{\fM;\dQ^\ac})$ with a non-degenerate $\Gamma_\dQ$-invariant
symmetric bilinear pairing. If $\rH^2_\pi(X_{\fM;\dQ^\ac})$ is not
semisimple, then there will be at least three non-zero direct summands of its
semisimplification, which contradicts the fact that $\pres{\sharp}\rV_\pi$
has exactly two irreducible direct summands since $\pi$ is non-dihedral.
\end{proof}

\begin{notation}\label{no:lattice}
By Lemma \ref{le:semisimple}, we will identify $\rH^2_\pi(X_{\fM;\dQ^\ac})$
with $\pres{\sharp}\rV_\pi$ once and for all. Then by abuse of notation,
\[\pres{\sharp}\rT_\pi\colonequals\rH^2_\pi(X_{\fM;\dQ^\ac})\cap\rH^2_{\et}(X_{\fM;\dQ^\ac},\dZ_p(1))\]
is a stable lattice of $\pres{\sharp}\rV_\pi$. The Poincar\'{e} duality
provides us with a $\Gamma_\dQ$-invariant perfect pairing
\[\pres{\sharp}{\bar\rT}_\pi^n\times\pres{\sharp}{\bar\rT}_\pi^n\to\dZ/p^n\]
for every integer $n\geq 1$. Put
$\rT_{\sigma,\pi}=\rT_\sigma\otimes_{\dZ_p}\pres{\sharp}\rT_\pi$, which is a
stable lattice of $\rV_{\sigma,\pi}$. Then we have a $\Gamma_\dQ$-invariant
perfect pairing
\begin{align}\label{eq:pairing}
\bar\rT_{\sigma,\pi}^n\times\bar\rT_{\sigma,\pi}^n\to\dZ/p^n(1)
\end{align}
for every integer $n\geq 1$.
\end{notation}

\subsection{Hirzebruch--Zagier classes}
\label{ss:gross_kudla}

\subsubsection{Cycle class map}

Let $X$ be a regular noetherian scheme on which $p$ is invertible. Let
$\Lambda$ be either $\dQ_p$ or $\dZ/p^n$ for $1\leq n\leq\infty$. We will
assume that suitable adic formalism has been adopted without specification
(for our purpose, \cite{Jan88} suffices). Let us first recall the following
important fact from algebraic geometry.

\begin{lem}
Let $X$ be a regular noetherian scheme and $j\colon Y\to X$ be a closed
immersion of codimension $\geq c$. We have that
\begin{enumerate}
  \item $\rR^q j^!\Lambda=0$ for $q<2c$;
  \item $j^!\Lambda\simeq\Lambda[-c](-c)$ if $Y$ is also regular and $j$
      is purely of codimension $c$.
\end{enumerate}
\end{lem}

Part (2) is known as the absolute cohomological purity; and (1) is known as
the semi-purity which is a consequence of (2). They are proved by Gabber; see
\cite{Fuj02}.

\begin{notation}
For $c\geq 0$, we denote by $\rZ^c(X)=\bigoplus_x\dZ\overline{\{x\}}$ for the
abelian group generated by the Zariski closure $\overline{\{x\}}\subset X$ of
all (scheme-theoretical) points $x$ of codimension $c$. We equip
$\overline{\{x\}}$ with the induced reduced scheme structure.
\end{notation}

The same process after \cite{Mil80}*{Lemma 9.1}, which will actually be
recalled in the proof of Lemma \ref{le:comparison_smooth}, assigns $Z$ a
class $\cl_X(Z)$ in the \'{e}tale cohomology $\rH^{2c}_{\et}(X,\Lambda(c))$.
By linearity, we extend this to a map
\begin{align}\label{eq:cycle_class}
\cl_X\colon\rZ^c(X)\to\rH^{2c}_{\et}(X,\Lambda(c)),
\end{align}
known as the \emph{(absolute \'{e}tale) cycle class map}.

If $X$ is smooth over a field, then $\cl_X$ factorizes through the group of
Chow cycles $\CH^c(X)$. Moreover, the map $\cl_X$ is \emph{multiplicative}
with respect to the intersection product on Chow groups and the cup product
on \'{e}tale cohomology groups.

\subsubsection{General \'{e}tale Abel--Jacobi map}

Let $k$ be a field of characteristic not $p$, and $Z$ a smooth proper scheme
over $k$. For each integer $e$, we have the Hochschild--Serre spectral
sequence $\rE_\bullet^{p,q}$ converging to
$\rH^{p+q}_{\et}(Z_{k^\ac},\Lambda(e))$ with
\[\rE_2^{p,q}=\rH^p(k,\rH^q_{\et}(Z_{k^\ac},\Lambda(e))).\]
By restriction, we have the map
\[\xi_Z^0\colon\rH^q_{\et}(Z,\Lambda(e))\to\rH^0(k,\rH^q_{\et}(Z_{k^\ac},\Lambda(e))).\]
Denote by $\rH^q_{\et}(Z,\Lambda(e))^0$ the kernel of $\xi_Z^0$. The spectral
sequence yields an edge map
\[\xi_Z^1\colon\rH^q_{\et}(Z,\Lambda(e))^0\to\rH^1(k,\rH^{q-1}_{\et}(Z_{k^\ac},\Lambda(e))).\]

For $c\geq0$, put $\CH^c(Z,\Lambda)=\CH^c(Z)\otimes_\dZ\Lambda$ and recall
that we have the cycle class map \eqref{eq:cycle_class}. Put
$\cl_Z^0=\xi_Z^0\circ\cl_Z$ and let $\CH^c(Z,\Lambda)^0$ be its kernel. The
composition
\[\cl_Z^1\colonequals\xi_Z^1\circ(\cl_Z\res_{\CH^c(Z,\Lambda)^0})\colon\CH^c(Z,\Lambda)^0\to\rH^1(k,\rH^{2c-1}_{\et}(Z_{k^\ac},\Lambda(c))) \]
is usually referred as the \emph{\'{e}tale Abel--Jacobi map}.

\begin{lem}\label{le:functoriality}
Let $f\colon Z'\to Z$ be a smooth proper morphism of equidimension $d$. Then
the following diagram commutes
\[\xymatrix{
\CH^{c+d}(Z',\Lambda)^0 \ar[r]^-{\cl_{Z'}} \ar[d]_-{f_*} & \rH^{2(c+d)}(Z',\Lambda(c+d))^0 \ar[r]^-{\xi_{Z'}^1}\ar[d]_-{f_*} &
\rH^1(k,\rH^{2(c+d)-1}_{\et}(Z'_{k^\ac},\Lambda(c+d))) \ar[d]^-{f_*} \\
\CH^c(Z,\Lambda)^0 \ar[r]^-{\cl_Z}  & \rH^{2c}(Z,\Lambda(c))^0 \ar[r]^-{\xi_Z^1} &
\rH^1(k,\rH^{2c-1}_{\et}(Z_{k^\ac},\Lambda(c))).
}\]
\end{lem}

\begin{proof}
It follows from the functoriality of cycle class maps (for smooth varieties)
and Hochschild--Serre spectral sequences.
\end{proof}

\begin{lem}\label{le:multiplicative_structure}
Let notation be as above and suppose that $Z$ is purely of dimension $d$.
Then the following diagram
\[\xymatrix{
\CH^d(Z\times_{\Spec k}Z,\Lambda)\times\CH^c(Z,\Lambda)^0 \ar[r]\ar[d]_-{(\cl_{Z\times_{\Spec k}Z}^0,\cl_Z^1)}
& \CH^c(Z,\Lambda)^0 \ar[d]^-{\cl_Z^1} \\
\rH^{2d}_{\et}((Z\times_{\Spec k}Z)_{k^\ac},\Lambda(d))\times\rH^1(k,\rH^{2c-1}(Z_{k^\ac},\Lambda(c)))
\ar[r] & \rH^1(k,\rH^{2c-1}(Z_{k^\ac},\Lambda(c))) }\] commutes, where the
upper (resp.\ lower) arrow is induced by the action of Chow (resp.\
cohomological) correspondences via pullbacks.
\end{lem}

\begin{proof}
Let $X$ be a smooth proper scheme over $k$. By the multiplicative structure
of Hochschild--Serre spectral sequence induced by cup products, we have the
following commutative diagram
\[\xymatrix{
\rH^q_{\et}(X,\Lambda(d))\times\rH^p_{\et}(X,\Lambda(c))^0 \ar[r]^-{\cup}\ar[d]_-{(\xi_X^0,\xi_X^1)}
& \rH^{p+q}(X,\Lambda(c+d))^0 \ar[d]^-{\xi_X^1} \\
\rH^q_{\et}(X_{k^\ac},\Lambda(d))^k\times\rH^1(k,\rH^{p-1}_{\et}(X_{k^\ac},\Lambda(c)))
\ar[r]^-{\cup}& \rH^1(k,\rH^{p+q-1}_{\et}(X_{k^\ac},\Lambda(c+d))), }\] for
$p\geq 1$, $q\geq 0$ and $c,d\in\dZ$. Then the lemma follows by the above
diagram with $X=Z\times_{\Spec k}Z$ and suitable $p,q,c,d$, Lemma
\ref{le:functoriality} and the fact that \eqref{eq:cycle_class} is
multiplicative.
\end{proof}

\subsubsection{Construction of Hirzebruch--Zagier classes}
\label{sss:construction_gross}

The following discussion is under Assumption \ref{as:group_r} (R1) and that
$\wp(N^-)$ is even. We take the coefficient ring $\Lambda$ to be $\dQ_p$.
Recall that we have the Hirzebruch--Zagier morphism
\[\zeta_{N^+M,N^-}\colon Y_{N^+M,N^-}\to X_{N^+M}\]
and the Hirzebruch--Zagier cycle
$\Delta_{N^+M,N^-}\colonequals\b\Delta_{N^+M,N^-;\dQ}$, a Chow cycle in
$\CH^2(Z_{N^+M,N^-})$ where
$Z_{N^+M,N^-}=Y_{N^+M,N^-}\times_{\Spec\dQ}X_{N^+M}$.

Choose a $\sigma$-projector $\sP_\sigma$ and a $\pi$-projector $\sP_\pi$
(Definitions \ref{de:projector_curve} and \ref{de:projector_surface}), both
of level $N^+M$. Put $\sP_{\sigma,\pi}=\sP_\sigma\times\sP_\pi$ as an element
in $\Corr(Z_{N^+M,N^-})$. Put
\[\Delta_{\sigma,\pi}=\sP_{\sigma,\pi}\Delta_{N^+M,N^-}\in\CH^2(Z_{N^+M,N^-})\otimes_\dZ\dQ,\]
which in fact belongs to $\CH^2(Z_{N^+M,N^-})^0$. Recall that we have the
corresponding \'{e}tale Abel--Jacobi map
\[\AJ_p\colonequals\cl_{Z_{N^+M,N^-}}^1\colon\CH^2(Z_{N^+M,N^-},\dQ_p)^0
\to\rH^1(\dQ,\rH^3_{\et}(Z_{N^+M,N^-;\dQ^\ac},\dQ_p(2))).\] Put
\[\pres{p}\Delta_{\sigma,\pi}=\AJ_p(\Delta_{\sigma,\pi}).\]
For simplicity, we also put $\rV_{\sigma,\pi;
N^+M,N^-}=\rH^1_\sigma(Y_{N^+M,N^-;\dQ^\ac})\otimes_{\dQ_p}\rH^2_\pi(X_{N^+M;\dQ^\ac})$
as a $p$-adic representation of $\Gamma_\dQ$. A priori, the Chow cycle
$\Delta_{\sigma,\pi}$ depends on the choices of both $\sP_\sigma$ and
$\sP_\pi$. Nevertheless, we have the following result.

\begin{proposition}\label{pr:belong_selmer}
The element $\pres{p}\Delta_{\sigma,\pi}$ belongs to the subspace
$\rH^1_f(\dQ,\rV_{\sigma,\pi; N^+M,N^-})$, and is independent of the choices
of $\sP_\sigma$ and $\sP_\pi$.
\end{proposition}

By Lemma \ref{le:semisimple}, the $p$-adic representation
$\rV_{\sigma,\pi;N^+M,N^-}$ of $\Gamma_\dQ$ is isomorphic to a non-zero
finite sum of copies of $\rV_{\sigma,\pi}$. Before the proof, we need a
little bit of preparation.

\subsubsection{An auxiliary Hecke operator}

Choose an auxiliary prime $q\nmid pNM$. Define an operator
\[\sT_\sigma^{(q)}=\frac{1}{q+1-\tr(\Fr_q;\rV_\sigma)}(q+1-T_{N^+M,N^-}^{(q)})\]
as an element in $\Corr(Z_{N^+M,N^-})$. This operator was also considered in
\cite{BD07}*{\Sec 4.3} and \cite{Zha14}*{\Sec 2.2}.

\begin{lem}\label{le:triviality}
The Chow cycle $\sT_\sigma^{(q)}\Delta_{N^+M,N^-}$ belongs to
$\CH^2(Z_{N^+M,N^-})^0$.
\end{lem}

\begin{proof}
By Lemma \ref{le:vanishing} and the K\"{u}nneth decomposition of algebraic de
Rham cohomology, we have
\begin{align*}
\rH^4_{\dr}(Z_{N^+M,N^-})=\rH^0_{\dr}(Y_{N^+M,N^-})\otimes_\dQ\rH^4_{\dr}(X_{N^+M})
\oplus\rH^2_{\dr}(Y_{N^+M,N^-})\otimes_\dQ\rH^2_{\dr}(X_{N^+M}).
\end{align*}
The lemma follows since the correspondence $q+1-T_{N^+M,N^-}^{(q)}$ acts by
$0$ on both $\rH^0_{\dr}(Y_{N^+M,N^-})$ and $\rH^2_{\dr}(Y_{N^+M,N^-})$.
\end{proof}

\begin{proof}[Proof of Proposition \ref{pr:belong_selmer}]
Choose an auxiliary prime $q$ as above. By the definition of
$\sP_{\sigma,\pi}$ and Lemmas \ref{le:triviality} and
\ref{le:multiplicative_structure}, we have
\[\AJ_p(\sP_{\sigma,\pi}\sT_\sigma^{(q)}\Delta_{N^+M,N^-})=\rH^1(\dQ,\sP^{\et}_{\sigma,\pi})\AJ_p(\sT_\sigma^{(q)}\Delta_{N^+M,N^-})\]
where
\begin{align*}
\sP_{\sigma,\pi}^{\et}&\colon\rH^3_{\et}(Z_{N^+M,N^-;\dQ^\ac},\dQ_p(2))\\
&\xrightarrow\sim\rH^1_{\et}(Y_{N^+M,N^-;\dQ^\ac},\dQ_p(1))\otimes_{\dQ_p}\rH^2_{\et}(X_{N^+M;\dQ^\ac},\dQ_p(1))\\
&\xrightarrow{\sP^{\et}_\sigma\otimes\sP^{\et}_\pi}
\rH^1_\sigma(Y_{N^+M,N^-;\dQ^\ac})\otimes_{\dQ_p}\rH^2_\pi(X_{N^+M;\dQ^\ac})
=\rV_{\sigma,\pi; N^+M,N^-}
\end{align*}
is the natural projection. On the other hand, we have
\[\AJ_p(\sP_{\sigma,\pi}\sT_\sigma^{(q)}\Delta_{N^+M,N^-})=\AJ_p(\sT_\sigma^{(q)}\sP_{\sigma,\pi}\Delta_{N^+M,N^-}).\]
By Lemmas \ref{le:multiplicative_structure} and \ref{le:vanishing}, and the
fact that the induced action of $\sT_\sigma^{(q)}$ on
$\rH^1_{\et}(Y_{N^+M,N^-;\dQ^\ac},\dQ_p(1))$ is invertible, the element
$\AJ_p(\sP_{\sigma,\pi}\Delta_{N^+M,N^-})$ itself belongs to
$\rH^1(\dQ,\rV_{\sigma,\pi; N^+M,N^-})$. Therefore, we have
\[\AJ_p(\sT_\sigma^{(q)}\sP_{\sigma,\pi}\Delta_{N^+M,N^-})=\AJ_p(\sP_{\sigma,\pi}\Delta_{N^+M,N^-}).\]
By Lemmas \ref{le:triviality} and \ref{le:multiplicative_structure}, the
elements $\AJ_p(\sP_{\sigma,\pi}\sT_\sigma^{(q)}\Delta_{N^+M,N^-})$ and hence
$\AJ_p(\sP_{\sigma,\pi}\Delta_{N^+M,N^-})$ are independent of the choices of
$\sP_\sigma$ and $\sP_\pi$.

Finally, we show that $s\colonequals\pres{p}\Delta_{\sigma,\pi}$ belongs to
$\rH^1_f(\dQ,\rV_{\sigma,\pi; N^+M,N^-})$. By Lemma \ref{le:semisimple}, we
only need to check that $\loc_p(s)\in\rH^1_f(\dQ_p,\rV_{\sigma,\pi;
N^+M,N^-})$. By \cite{Nek00}*{Theorem 3.1 (ii)}, $\loc_p(s)$ belongs to
$\rH^1_{\r{st}}(\dQ_p,\rV_{\sigma,\pi; N^+M,N^-})$. Since the $p$-adic
representation $\rV_{\sigma,\pi; N^+M,N^-}$ of $\Gamma_p$ is a direct summand
of
$\rH^1_{\et}(Y_{N^+M,N^-;\dQ_p^\ac},\dQ_p(1))\otimes\rH^2_\cusp(X_{N^+M;\dQ_p^\ac},\dQ_p(1))$,
and for the latter we have $\rH^1_{\r{st}}(\dQ_p,-)=\rH^1_f(\dQ_p,-)$, the
proposition follows.
\end{proof}

\subsubsection{Statement of main theorem}

The following theorem is the main result of the article, in terms of
automorphic representations.

\begin{theorem}\label{th:selmer}
Suppose that all of Assumption \ref{as:group_r} are satisfied.
\begin{enumerate}
  \item If $L(1/2,\sigma\times\pi)\neq 0$, which in particular implies
      that $\wp(N^-)$ is odd, then
      \[\dim_{\dQ_p}\rH^1_f(\dQ,\rV_{\sigma,\pi})=0.\]

  \item If $\wp(N^-)$ is even and $\pres{p}{\Delta}_{\sigma,\pi}\neq 0$,
      then
      \[\dim_{\dQ_p}\rH^1_f(\dQ,\rV_{\sigma,\pi})=1.\]
\end{enumerate}
\end{theorem}

\subsection{Testing factors}
\label{ss:testing_factors}

From now on until the end of \Sec \ref{ss:proof_theorem}, all discussions
will be under Assumption \ref{as:group_r} (R1 -- R3), and we \emph{fix} a
prime $q$ not dividing $pNM\disc(F)$, no matter $\wp(N^-)$ is even or odd,
such that $p\nmid q+1-\tr(\Fr_q;\rV_\sigma)$. This is possible due to (R2).

\subsubsection{The case where $\wp(N^-)$ is even}

Suppose that $\wp(N^-)$ is even. From now on until the end of \Sec
\ref{ss:proof_theorem}, we will assume that $\pres{p}\Delta_{\sigma,\pi}$ is
non-zero in this case.

Recall from Proposition \ref{pr:belong_selmer} that
$\pres{p}\Delta_{\sigma,\pi}$ belongs to $\rH^1_f(\dQ,\rV_{\sigma,\pi;
N^+M,N^-})$. For each factor $d\mid M$ and $\fd\mid N^+M/\fM$, the class
\[\pres{p}\Delta_{\sigma,\pi}^{(d,\fd)}\colonequals
(\delta^{(M,d)}_{N^+M,N^-}\otimes\gamma^{N^+M/\fM,\fd}_{N^+M})_*\pres{p}\Delta_{\sigma,\pi}\]
belongs to $\rH^1_f(\dQ,\rV_{\sigma,\pi})$ by Lemma \ref{le:semisimple}.

\begin{definition}\label{de:testing}
A pair of \emph{testing factors} is a pair $(d,\fd)$ as above such that
$\pres{p}\Delta_{\sigma,\pi}^{(d,\fd)}\neq 0$.
\end{definition}

By the theory of old forms, there exists at least one pair of testing factors
and we will fix such a pair $(d,\fd)$ once and for all.

\subsubsection{The case where $\wp(N^-)$ is odd}

Suppose that $\wp(N^-)$ is odd. From now on until the end of \Sec
\ref{ss:proof_theorem}, we will assume that $L(1/2,\sigma\times\pi)$ is
non-zero in this case.

Since $\cT_{N^+,N^-}$ is (non-canonically) isomorphic to the double quotient
$B_{N^-}^\times\backslash(B_{N^-}\otimes_\dZ\widehat\dZ)^\times/(R\otimes_\dZ\widehat\dZ)^\times$
for some $R\in\cT_{N^+,N^-}$, the newform (fixed by
$(R\otimes_\dZ\widehat\dZ)^\times$) of the automorphic representation
$\sigma$ gives rise to a map
\begin{align}\label{eq:new_form}
g_\sigma\colon\cT_{N^+,N^-}\to\dZ
\end{align}
whose image is not contained in a proper ideal of $\dZ$. Up to $\pm1$, the
map $g_\sigma$ is independent of the choice of $R$. Similarly, since
$\cS_\fM$ is (non-canonically) isomorphic to the double quotient
$Q^\times\backslash(Q\otimes_\dZ\widehat\dZ)^\times/(S\otimes_\dZ\widehat\dZ)^\times$
for some $S\in\cS_\fM$, the newform (fixed by
$(S\otimes_\dZ\widehat\dZ)^\times$) of the automorphic representation
$\pi^Q$, the Jacquet--Langlands correspondence of $\pi$ to $Q$, gives rise to
a surjective map
\begin{align*}
f_\pi\colon\cS_\fM\to\dZ
\end{align*}
whose image is not contained in a proper ideal of $\dZ$. Up to $\pm1$, the
map $f_\pi$ is independent of the choice of $S$.

\begin{definition}\label{de:testing_bis}
A pair of \emph{testing factors} is a pair $(d,\fd)$ with $d\mid M$ and
$\fd\mid N^+M/\fM$ such that
\begin{align}\label{eq:ichino}
\sum_{t\in\cT_{N^+M,N^-}}((\zeta_{N^+M,N^-})^*(\gamma^{(N^+M/\fM,\fd)}_{N^+M})^*f_\pi)(t)\cdot
((\delta^{(M,d)}_{N^+M,N^-})^*g_\sigma)(t)\neq 0.
\end{align}
\end{definition}

\begin{proposition}
There exists a pair of testing factors.
\end{proposition}

In what follows, we will fix such a pair $(d,\fd)$ once and for all.
Moreover, put
\[f_\pi^{(\fd)}=(\gamma^{(N^+M/\fM,\fd)}_{N^+M})^*f_\pi\]
for simplicity.

\begin{proof}
By \cite{Ich08}*{Theorem 1.1}, we have
\[\eqref{eq:ichino}=C\cdot L^S(1/2,\sigma\times\pi)\cdot \prod_{v\in S}\alpha_v(d,\fd)\]
where $S$ is the set of places of $\dQ$ consisting of $\infty$ and primes
dividing $NM\disc(F)$, $C$ is a non-zero number, and $\alpha_v(d,\fd)$ is a
certain integral of local matrix coefficients depending on $(d,\fd)$. We
claim that there is a pair $(d,\fd)$ such that $\prod_{v\in
S}\alpha_v(d,\fd)\neq 0$. For $v\nmid N^+$, it follows from the lemma below.
For $v\mid N^+$, the situation is in fact the same as in the lemma below for
$v\nmid N$ (one needs to change the role of $\sigma_v$ by a component of
$\pi_v$ as $v$ splits in $F$ and $\pi_v$ is unramified).

On the other hand, for each $v\in S$, the local $L$-function
$L_v(s,\sigma\times\pi)$ does not have a pole at $s=1/2$ since
$\rV_{\sigma,\pi}$ is tamely pure by Lemma \ref{le:semisimple}. Therefore,
$L^S(1/2,\sigma\times\pi)\neq 0$, and hence the lemma follows.
\end{proof}

\begin{lem}\label{le:test_vector}
Let $v$ be a prime in $S$ that does not divide $N^+$. Suppose that $l$ spans
the $1$-dimensional space $\Hom_{B_{N^-}(\dQ_v)}(\pi_v,\sigma_v)$. Then the
image of a new vector of $\pi_v$ under $l$ is non-zero.
\end{lem}

The dimension of $\Hom_{B_{N^-}(\dQ_v)}(\pi_v,\sigma_v)$ is $1$, and
$l\otimes l$ is equal to the matrix coefficient integral $\alpha_v$ up to a
non-zero constant.

\begin{proof}
Denote by $B\subset\GL_2$ the standard Borel subgroup of upper-triangular
matrices and $T\subset B$ the diagonal subgroup of $\GL_2$.

\textbf{Case 1: $v\nmid N$.} Then $B_{N^-}$ is unramified at $v$. Suppose
that $\sigma_v$ is the non-normalized induction from $B(\dQ_v)$ to
$\GL_2(\dQ_v)$ of an unramified character $\chi$ of $T(\dQ_v)$. By Frobenius
reciprocity, we have
$\Hom_{\GL_2(\dQ_v)}(\pi_v,\sigma_v)=\Hom_{B(\dQ_v)}(\pi_v\res_{B(\dQ_v)},\chi)$.

Choose a non-zero element $j\in F_v$ such that $j^\theta=-j$. Choose an
additive character $\psi$ of $\dQ_v$ of conductor $0$. We realize $\pi_v$ in
the Whittaker model with respect to the unipotent radical of $B(F_v)$ and the
character $\psi\circ(\frac{1}{2}\Tr_{F_v/\dQ_v})$. The following integral
\begin{align}\label{eq:integral}
\int_{\dQ_v^\times}W\(\left(
                    \begin{array}{cc}
                      ja &  \\
                       & 1 \\
                    \end{array}
                  \right)
\)\chi\(\left(
          \begin{array}{cc}
            a &  \\
             & 1 \\
          \end{array}
        \right)
\)\rd a
\end{align}
is absolutely convergent for all functions $W$ in (the above Whittaker model
of) $\pi_v$, where $\rd a$ is a non-zero Haar measure on $\dQ_v^\times$.
Moveover, it defines an element in the space
$\Hom_{B(\dQ_v)}(\pi_v\res_{B(\dQ_v)},\chi)$. It is easy to calculate that
the integral \eqref{eq:integral} is non-zero for a new vector $f$ as $\chi$
is unramified. The lemma follows.

\textbf{Case 2: $v\mid N^-$.} Then $B_{N^-}$ is ramified at $v$ and $v$ is
inert in $F$. Thus $\sigma_v$ is an unramified character of $D_v^\times$ of
order at most $2$, where $D_v$ is the division quaternion algebra over
$\dQ_v$. Since $\pi_v$ is now unramified, we assume
$\pi_v=\Ind_{B(F_v)}^{\GL_2(F_v)}\chi$ where the induction is normalized.
Then the character $\chi'$ of $F_v^\times$ defined in \cite{Pra92}*{\Sec 4.1}
is trivial. Thus, we have
$\pi_v\res_{D_v^\times}\simeq\ind_{F_v^\times}^{D_v^\times}\b{1}$. Here, we
have fixed a decomposition $D_v=F_v\oplus F_v\Pi$ where $\Pi$ is a
uniformizer of $D_v$. As
$\Hom_{D_v^\times}(\ind_{F_v^\times}^{D_v^\times}\b{1},\sigma_v)\simeq\Hom_{F_v^\times}(\b{1},\sigma_v)\neq\{0\}$,
the lemma is equivalent to that
\[\int_{F_v^\times\backslash D_v^\times}\sigma_v(g^{-1})f(g)\rd g\neq 0\]
for a new vector $f\in\pi_v$, where $\rd g$ is a Haar measure on
$F_v^\times\backslash D_v^\times$. The above nonvanishing holds since
$f(\Pi)/f(1)\not\in\{\pm1\}$.
\end{proof}

\begin{remark}\label{re:PSR}
In fact, it is not clear whether for all places $v$, the two local
$L$-functions $L_v(s,\sigma\times\pi)$ defined from the Galois side and
$L_v^{\r{PSR}}(s,\sigma\times\pi)$ defined by Piatetski-Shapiro--Rallis
\cite{PSR87} (and used in \cite{Ich08}) coincide or not. Nevertheless, we
have
\[\ord_{s=1/2}L(s,\sigma\times\pi)=\ord_{s=1/2}L^{\r{PSR}}(s,\sigma\times\pi).\]
\end{remark}

\subsection{Raising levels}
\label{ss:raising_levels}

In this section, we study only the representation $\sigma$. We collect some
results from \cite{BD05} which generalizes the previous work of Ribet
\cite{Rib90}.

\begin{definition}[Group \textbf{A}, compare with \cite{BD05}*{\Sec
2.2}]\label{de:group_a} Let $n\geq 1$ be an integer. We say that a prime
$\ell$ is \emph{$n$-admissible} (with respect to $\sigma$) if
\begin{description}
  \item[(A1)] $\ell\nmid 2pqN^-\cD_{N^+M}$;

  \item[(A2)] $p\nmid\ell^2-1$;

  \item[(A3)] $p^n\mid\ell+1-\epsilon_\sigma(\ell)
      \tr(\Fr_\ell;\rV_\sigma)$ for some
      $\epsilon_\sigma(\ell)\in\{\pm1\}$;

  \item[(A4)] $p^n\mid\ell^{p-1}-1$;

  \item[(A5)] $\ell$ is inert in $F$.
\end{description}
\end{definition}

\begin{remark}
For results in this section, (A4) and (A5) are not used and we actually only
need $\ell\nmid pN$ in (A1).
\end{remark}

\begin{lem}\label{le:cohomology_sigma}
Let $\ell$ be an $n$-admissible prime. Then
\begin{enumerate}
  \item the submodule $\bar\rT_\sigma^n[\ell^2|1]$ is free of rank $1$;

  \item the natural map
      $\rH^1(\dQ_{\ell^2},\bar\rT_\sigma^n[\ell^2|1])\to\rH^1(\dQ_{\ell^2},\bar\rT_\sigma^n)$
      is injective and its image coincides with
      $\rH^1_\unr(\dQ_{\ell^2},\bar\rT_\sigma^n)$;

  \item the induced map
      $\bar\rT_\sigma^n[\ell^2|\ell^2]\to\bar\rT_\sigma^n/\bar\rT_\sigma^n[\ell^2|1]$
      is an isomorphism;

  \item both $\rH^1(\dQ_{\ell^2},\bar\rT_\sigma^n[\ell^2|1])$ and
      $\rH^1(\dQ_{\ell^2},\bar\rT_\sigma^n[\ell^2|\ell^2])$ are
      isomorphic to $\dZ/p^n$.
\end{enumerate}
Here we recall from \Sec \ref{ss:galois_modules} that
$\bar\rT_\sigma^n[\ell^2|r]$ is the module of $\bar\rT_\sigma^n$ on which
$\rI_{\ell^2}$ acts trivially and $\Fr_{\ell^2}$ acts by multiplication by
$r$.
\end{lem}

\begin{proof}
These are proved in \cite{GP12}*{Lemma 8}.
\end{proof}

By definition, we have an exact sequence
\[0 \to \rH^1_\unr(\dQ_{\ell^2},\bar\rT_\sigma^n)\to\rH^1(\dQ_{\ell^2},\bar\rT_\sigma^n)
\xrightarrow{\partial_{\ell^2}}\rH^1_\sing(\dQ_{\ell^2},\bar\rT_\sigma^n)\to 0,\] in which we
may fix isomorphisms
\begin{align}\label{eq:us}
\tu_{\ell}\colon\rH^1_\unr(\dQ_{\ell^2},\bar\rT_\sigma^n)\xrightarrow\sim\dZ/p^n,\quad
\ts_{\ell}\colon\rH^1_\sing(\dQ_{\ell^2},\bar\rT_\sigma^n)\xrightarrow\sim\dZ/p^n,
\end{align}
by the previous lemma. Only in this section, we use full Hecke algebras as
defined in \cite{BD05}. There are two cases.

\subsubsection{The case $\wp(N^-)$ is odd}

For $1\leq n\leq \infty$, denote by
\begin{align}\label{eq:level_raising3_bis}
g_\sigma^n\colon\cT_{N^+,N^-}\to\dZ/p^n
\end{align}
the composition of $g_\sigma$ \eqref{eq:new_form} with the natural map
$\dZ\to\dZ/p^n$. Now we fix an integer $n\geq 1$ and let $\ell$ be an
$n$-admissible prime. Applying the argument of \cite{BD05}*{\Sec 5.6} to
$g_\sigma^n$, we obtain a surjective homomorphism
\[\varphi_{\sigma\res\ell}^n\colon\dT_{N^+,N^-\ell}\to\dZ/p^n\]
by \cite{BD05}*{Theorem 5.15}, whose kernel is denoted by
$\cI_{\sigma\res\ell}^n$. Note that condition (2) of \cite{BD05}*{Theorem
5.15} is satisfied by Assumption \ref{as:group_r} (R2, R3) and
\cite{PW11}*{Theorem 6.2}. In fact, the condition that $N$ is square-free in
\cite{PW11} is unnecessary.

\begin{proposition}\label{pr:raising_1_bis}
There is a unique isomorphism
\[\beta\colon\Phi^{(\ell)}_{N^+,N^-\ell}/\cI_{\sigma\res\ell}^n\xrightarrow\sim\dZ/p^n\]
rendering the following diagram commutative
\[\xymatrix{
\dZ[\cT_{N^+,N^-}^\bullet\sqcup\cT_{N^+,N^-}^\circ]^0 \ar[rr]
\ar[d]_-{\phi^{(\ell)}_{N^+,N^-\ell}}
&& \dZ/p^n \\
\Phi^{(\ell)}_{N^+,N^-\ell} \ar[rr]&&
\Phi^{(\ell)}_{N^+,N^-\ell}/\cI_{\sigma\res\ell}^n \ar[u]_-{\beta}, }\]
where the upper arrow is induced by the function
\[g_\sigma^n\sqcup\epsilon_\sigma(\ell)g_\sigma^n\colon\cT_{N^+,N^-}^\bullet\sqcup\cT_{N^+,N^-}^\circ
\simeq\cT_{N^+,N^-}\sqcup\cT_{N^+,N^-}\to\dZ/p^n;\] the group
$\Phi^{(\ell)}_{N^+,N^-\ell}$ is defined in Notation \ref{no:neron}; and the
map $\phi^{(\ell)}_{N^+,N^-\ell}$ is the one in \cite{BD05}*{Corollary 5.12}.
\end{proposition}

\begin{proof}
It follows from the proof of \cite{BD05}*{Theorem 5.15}.
\end{proof}

\begin{proposition}\label{pr:raising_2_bis}
There is unique isomorphism
\[\hat\beta\colon\rT_p(J_{N^+,N^-\ell})/\cI_{\sigma\res\ell}^n\to\bar\rT^n_\sigma\]
of $\dZ/p^n[\Gamma_\dQ]$-modules rendering the following diagram commutative
\[\xymatrix{
J_{N^+,N^-\ell}(\dQ_{\ell^2}) \ar[rr]\ar[d]&& \rH^1(\dQ_{\ell^2},J_{N^+,N^-\ell}[p^n](\dQ_{\ell}^\ac)) \ar[d]\\
\Phi^{(\ell)}_{N^+,N^-\ell} \ar[d]^-{\beta} &&
\rH^1(\dQ_{\ell^2},\rT_p(J_{N^+,N^-\ell})/\cI_{\sigma\res\ell}^n) \ar[d]^-{\rH^1(\dQ_{\ell^2},\hat\beta)} \\
\dZ/p^n & \rH^1_\sing(\dQ_{\ell^2},\bar\rT^n_\sigma) \ar[l]_-{\ts_{\ell}}&
\rH^1(\dQ_{\ell^2},\bar\rT^n_\sigma) \ar[l]_-{\partial_{\ell^2}}, }\] where
the upper horizontal arrow is the Kummer map; and the first left vertical
arrow is the reduction map.
\end{proposition}

\begin{proof}
It follows from the proof of \cite{BD05}*{Theorem 5.17, Corollary 5.18}.
\end{proof}

Note that there is a canonical $\Gamma_\dQ$-equivariant $\dZ_p$-linear
isomorphism between the Tate module $\rT_p(J_{N^+,N^-\ell})$ and
$\rH^1_{\et}(Y_{N^+,N^-\ell},\dZ_p(1))$. Therefore, Proposition
\ref{pr:raising_2} yields a projection
\begin{align}\label{eq:level_raising4_bis}
\psi_{\ell}^n\colon\rH^1_{\et}(Y_{N^+,N^-\ell;\dQ^\ac},\dZ/p^n(1))\to\bar\rT^n_\sigma.
\end{align}

\subsubsection{The case $\wp(N^-)$ is even}

The automorphic representation $\sigma$ gives rise to a surjective
homomorphism $\varphi_\sigma\colon \dT_{N^+,N^-}\to\dZ_p$. For $1\leq n\leq
\infty$, denote by $\varphi_\sigma^n$ the composition of $\varphi_\sigma$
with the natural map $\dZ_p\to\dZ/p^n$, and $\cI_\sigma^n$ the kernel of
$\varphi_\sigma^n$.

Now we fix an integer $n\geq 1$. Assumption \ref{as:group_r} (R2, R3) imply
that $\rT_p(J_{N^+,N^-})/\cI_\sigma^n$ is isomorphic to $\bar\rT_\sigma^n$ as
a $\dZ/p^n[\Gamma_\dQ]$-module, by \cite{Wil95}*{Theorem 2.1} and
\cite{Hel07}*{Corollary 8.11, Remark 8.12}. Take an $n$-admissible prime
$\ell_1$. The Kummer map
\[J_{N^+,N^-}(\dQ_{\ell_1^2})/\cI_\sigma^n\to\rH^1(\dQ_{\ell_1^2},\bar\rT^\infty_p(J_{N^+,N^-})/\cI_\sigma^n)\]
induces a map
\begin{align*}
J_{N^+,N^-}(\dQ_{\ell_1^2})/\cI_\sigma^n\to\rH^1(\dQ_{\ell_1^2},\bar\rT_\sigma^n),
\end{align*}
whose image is contained in $\rH^1_\unr(\dQ_{\ell_1^2},\bar\rT_\sigma^n)$.
Running the same argument on \cite{BD05}*{Page 57}, we obtain a natural
surjective map
\begin{align}\label{eq:level_raising3_pre}
\dZ[\cT_{N^+,N^-\ell_1}]
\to\rH^1_\unr(\dQ_{\ell_1^2},\bar\rT_\sigma^n)\xrightarrow{\tu_{\ell_1}}\dZ/p^n.
\end{align}
It induces a function
\begin{align}\label{eq:level_raising3}
g_{\sigma\res\ell_1}^n\colon\cT_{N^+,N^-\ell_1}\to\dZ/p^n,
\end{align}
which satisfies
$g_{\sigma\res\ell_1}^n\circ\op_{\ell_1}=\epsilon_\sigma(\ell_1)g_{\sigma\res\ell_1}^n$.

Let $\ell_2$ be an $n$-admissible prime other than $\ell_1$. Applying the
argument of \cite{BD05}*{\Sec 5.6} to $g_{\sigma\res\ell_1}^n$ with
$N^+=N^+$, $N^-=N^-\ell_1$, $m=\ell_1$ and $\ell=\ell_2$, we obtain a
surjective homomorphism
\[\varphi_{\sigma\res\ell_1,\ell_2}^n\colon\dT_{N^+,N^-\ell_1\ell_2}\to\dZ/p^n\]
by \cite{BD05}*{Theorem 5.15}. We denote its kernel by
$\cI_{\sigma\res\ell_1,\ell_2}^n$. We have the following results similar to
Propositions \ref{pr:raising_1_bis} and \ref{pr:raising_2_bis}.

\begin{proposition}\label{pr:raising_1}
There is a unique isomorphism
\[\beta\colon\Phi^{(\ell_2)}_{N^+,N^-\ell_1\ell_2}/\cI_{\sigma\res\ell_1,\ell_2}^n\xrightarrow\sim\dZ/p^n\]
rendering the following diagram commutative
\[\xymatrix{
\dZ[\cT_{N^+,N^-\ell_1}^\bullet\sqcup\cT_{N^+,N^-\ell_1}^\circ]^0 \ar[rr]
\ar[d]_-{\phi^{(\ell_2)}_{N^+,N^-\ell_1\ell_2}}
&& \dZ/p^n  \\
\Phi^{(\ell_2)}_{N^+,N^-\ell_1\ell_2} \ar[rr]&&
\Phi^{(\ell_2)}_{N^+,N^-\ell_1\ell_2}/\cI_{\sigma\res\ell_1,\ell_2}^n
\ar[u]_-{\beta}, }\] where the upper arrow is induced by the function
$g_{\sigma\res\ell_1}^n\sqcup\epsilon_\sigma(\ell_2)g_{\sigma\res\ell_1}^n$.
\end{proposition}

\begin{proposition}\label{pr:raising_2}
There is unique isomorphism
\[\hat\beta\colon\rT_p(J_{N^+,N^-\ell_1\ell_2})/\cI_{\sigma\res\ell_1,\ell_2}^n\to\bar\rT^n_\sigma\]
of $\dZ/p^n[\Gamma_\dQ]$-modules rendering the following diagram commutative
\[\xymatrix{
J_{N^+,N^-\ell_1\ell_2}(\dQ_{\ell_2^2}) \ar[rr]\ar[d]&& \rH^1(\dQ_{\ell_2^2},J_{N^+,N^-\ell_1\ell_2}[p^n](\dQ_{\ell_2}^\ac)) \ar[d]\\
\Phi^{(\ell_2)}_{N^+,N^-\ell_1\ell_2} \ar[d]^-{\beta} &&
\rH^1(\dQ_{\ell_2^2},\rT_p(J_{N^+,N^-\ell_1\ell_2})/\cI_{\sigma\res\ell_1,\ell_2}^n) \ar[d]^-{\rH^1(\dQ_{\ell_2^2},\hat\beta)} \\
\dZ/p^n & \rH^1_\sing(\dQ_{\ell_2^2},\bar\rT^n_\sigma) \ar[l]_-{\ts_{\ell_2}}&
\rH^1(\dQ_{\ell_2^2},\bar\rT^n_\sigma) \ar[l]_-{\partial_{\ell_2^2}}.
}\]
\end{proposition}

Note that there is a canonical $\Gamma_\dQ$-equivariant $\dZ_p$-linear
isomorphism between the Tate module $\rT_p(J_{N^+,N^-\ell_1\ell_2})$ and
$\rH^1_{\et}(Y_{N^+,N^-\ell_1\ell_2},\dZ_p(1))$. Therefore, Proposition
\ref{pr:raising_2} yields a projection
\begin{align}\label{eq:level_raising4}
\psi_{\ell_1,\ell_2}^n\colon\rH^1_{\et}(Y_{N^+,N^-\ell_1\ell_2;\dQ^\ac},\dZ/p^n(1))\to\bar\rT^n_\sigma.
\end{align}

\subsection{Construction of annihilators}
\label{ss:construction_annihilators}

In this section, the coefficient ring $\Lambda$ will be $\dZ/p^n$ for $1\leq
n\leq\infty$. In order to emphasize the exponent $n$, we will use the
notation $[\quad]_Z^n$ for the \'{e}tale Abel--Jacobi map $\cl_Z^1$ with the
coefficient ring $\Lambda=\dZ/p^n$.

Let $N^\sim$ be a product of even number of distinct primes that are inert in
$F$, such that $N^-$ divides $N^\sim$ and $N^\sim/N^-$ is coprime to
$2pq\cD_{N^+M}$. In what follows, we will regard the schemes
$\cY_{N^+M,N^\sim}$, $\cX_{N^+M}$, $\cZ_{N^+M,N^\sim}$ and the
Hirzebruch--Zagier morphism $\zeta_{N^+M,N^\sim}$ as over
$\dZ[1/q\cD_{N^+M}]$. Recall that we have the cycle $\b\Delta_{N^+M,N^\sim}$
which is the graph of $\zeta_{N^+M,N^\sim}$, as an element in
$\rZ^2(\cZ_{N^+M,N^\sim})$. We also have two degeneracy morphisms
\begin{align*}
\cY_{N^+M,N^\sim} \xleftarrow{\delta_{N^+Mq,N^\sim}^{(q,1)}}\cY_{N^+Mq,N^\sim}
\xrightarrow{\delta_{N^+Mq,N^\sim}^{(q,q)}}\cY_{N^+M,N^\sim},
\end{align*}
which induce two morphisms
\begin{align*}
\zeta_{N^+M,N^\sim}^\dag\colonequals\zeta_{N^+Mq,N^\sim}\circ\delta_{N^+Mq,N^\sim}^{(q,1)}&\colon\cY_{N^+Mq,N^\sim}\to\cX_{N^+M},\\
\zeta_{N^+M,N^\sim}^\ddag\colonequals\zeta_{N^+Mq,N^\sim}\circ\delta_{N^+Mq,N^\sim}^{(q,q)}&\colon\cY_{N^+Mq,N^\sim}\to\cX_{N^+M}.
\end{align*}
Put
$\cZ'_{N^+M,N^\sim}=\cY_{N^+Mq,N^\sim}\times_{\Spec\dZ[1/q\cD_{N^+M}]}\cX_{N^+M}$.
Denote the graphs of $\zeta_{N^+Mq,N^\sim}^\dag$ and
$\zeta_{N^+Mq,N^\sim}^\ddag$ by $\b\Delta_{N^+M,N^\sim}^\dag$ and
$\b\Delta_{N^+M,N^\sim}^\ddag$, respectively, which are both smooth over the
base $\dZ[1/q\cD_{N^+M}]$. We also have the \'{e}tale projection
\begin{align*}
\delta'\colonequals\delta_{N^+Mq,N^\sim}^{(q,1)}\times\r{id}\colon\cZ'_{N^+M,N^\sim}\to\cZ_{N^+M,N^\sim}
\end{align*}
under which the image of $\b\Delta_{N^+M,N^\sim}^\dag$ is simply
$\b\Delta_{N^+M,N^\sim}$. We remark that the image of
$\b\Delta_{N^+M,N^\sim}^\ddag$ under $\delta'$ is not necessarily smooth over
the base. Put
\begin{align}\label{eq:cycle_modified}
\b\Delta'_{N^+M,N^\sim}=\frac{1}{q+1-\tr(\Fr_q;\rV_\sigma)}
(\b\Delta_{N^+M,N^\sim}^\dag-\b\Delta_{N^+M,N^\sim}^\ddag)
\end{align}
as a cycle of $\cZ'_{N^+M,N^\sim}$. Put
$\Delta'_{N^+M,N^\sim}=\b\Delta'_{N^+M,N^\sim;\dQ}$ as usual.

\begin{lem}\label{le:change_cycle}
We have that
\begin{enumerate}
  \item for every point $\Spec k\to\Spec\dZ[1/pqN^\sim\cD_{N^+M}]$ where
      $k$ is a field and every coefficient ring $\Lambda=\dZ/p^n$ with
      arbitrary $1\leq n\leq \infty$, the (induced) Chow cycle (of)
      $\b\Delta'_{N^+M,N^\sim;k}$ belongs to
      $\CH^2(\cZ'_{N^+M,N^\sim;k},\Lambda)^0$;

  \item
      $\delta'_{\dQ*}\Delta'_{N^+M,N^\sim}=\sT_\sigma^{(q)}\Delta_{N^+M,N^\sim}$
      as an equality in $\CH^2(Z_{N^+M,N^\sim})\otimes_\dZ\dQ$;

  \item for arbitrary $1\leq n\leq\infty$,
      \[\delta'_{\dQ*}[\Delta'_{N^+M,N^\sim}]_{\cZ'_{N^+M,N^\sim;\dQ}}^n
      =[\sT_\sigma^{(q)}\Delta_{N^+M,N^\sim}]_{Z_{N^+M,N^\sim}}^n.\]
\end{enumerate}
\end{lem}

\begin{proof}
Part (1) follows from the same proof of Lemma \ref{le:triviality}, and the
Comparison Theorem for \'{e}tale cohomology with $\dZ_p$-coefficient. Part
(2) follows directly from the construction. Part (3) follows from (2) and
Lemma \ref{le:functoriality}.
\end{proof}

\subsubsection{The case $N^\sim=N^-$}

Assume that $\wp(N^-)$ is even and $N^\sim=N^-$. Recall that we have fixed a
pair of testing factors $(d,\fd)$. Consider the following composite map
\begin{align*}
\sP_{\sigma,\pi}^{(d,\fd)}&\colon\rH^3_{\et}(Z_{N^+Mq,N^-;\dQ^\ac},\dZ_p(2))\\
&\to\rH^1_{\et}(Y_{N^+Mq,N^-;\dQ^\ac},\dZ_p(1))\otimes_{\dZ_p}\rH^2_{\et}(X_{N^+M;\dQ^\ac},\dZ_p(1))\\
&\to\rH^1_{\et}(Y_{N^+,N^-;\dQ^\ac},\dZ_p(1))\otimes_{\dZ_p}\rH^2_{\et}(X_{\fM;\dQ^\ac},\dZ_p(1))\\
&\to\rT_\sigma\otimes_{\dZ_p}\pres{\sharp}{\rT}_\pi=\rT_{\sigma,\pi},
\end{align*}
where the second last arrow is induced by
$\delta^{(Mq,d)}_{N^+Mq,N^-}\otimes\gamma^{(N^+M/\fM,\fd)}_\fM$; and the last
one is the projection to the $(\sigma,\pi)$-component.

Introduce the class
\[\bar\Delta_{\sigma,\pi}^\infty=\rH^1(\dQ,\sP_{\sigma,\pi}^{(d,\fd)})
([\Delta'_{N^+M,N^-}]_{Z_{N^+Mq,N^-}}^\infty)\in\rH^1(\dQ,\rT_{\sigma,\pi}).\]
Then $\bar\Delta_{\sigma,\pi}^\infty$ is non-torsion by Definition
\ref{de:testing}, the proof of Proposition \ref{pr:belong_selmer}, and Lemma
\ref{le:change_cycle} (3). It induces classes
$\bar\Delta_{\sigma,\pi}^n\in\rH^1(\dQ,\bar\rT_{\sigma,\pi}^n)$ for all
$1\leq n\leq \infty$, where we recall that
$\bar\rT_{\sigma,\pi}^n=\rT_{\sigma,\pi}\otimes_{\dZ_p}\dZ/p^n$.

\subsubsection{The case $N^\sim\neq N^-$}

We first define strongly admissible primes.

\begin{definition}[Group \textbf{S}]\label{de:group_s}
Let $n\geq 1$ be an integer. We say that a prime $\ell$ is \emph{strongly
$(n,\epsilon)$-admissible} (with respect to $\sigma$ and $\pi$) if
\begin{description}
  \item[(S1)] $\ell$ is $n$-admissible (Definition \ref{de:group_a});

  \item[(S2)] $\epsilon=-\epsilon_\sigma(\ell)\breve\eta(\Fr_\ell)$,
      where $\breve\eta$ is defined in \Sec \ref{sss:asai};

  \item[(S3)] $\tr(\Fr_{\ell^2};\rV_\pi)\mod
      p\not\in\{2\ell,-2\ell,\ell^2+1,-\ell^2-1\}$, and
      \begin{itemize}
        \item $\tr(\Fr_{\ell^2};\rV_\pi)\mod p\neq 0$ if
            $\dF_p^\times$ contains an element of order $4$;
        \item $\tr(\Fr_{\ell^2};\rV_\pi)\mod p\neq -\ell$ if
            $\dF_p^\times$ contains an element of order $3$;
        \item $\tr(\Fr_{\ell^2};\rV_\pi)\mod p\neq \ell$ if
            $\dF_p^\times$ contains an element of order $6$.
      \end{itemize}
\end{description}
A prime $\ell$ is \emph{strongly $n$-admissible} if it is either strongly
$(n,+)$-admissible or strongly $(n,-)$-admissible.
\end{definition}

\begin{remark}\label{re:strongly_admissible}~
\begin{enumerate}
  \item For $\epsilon=\pm$, if there exists one strongly
      $(1,\epsilon)$-admissible prime, then the density of strongly
      $(n,\epsilon)$-admissible primes among all $n$-admissible primes is
      strictly positive, and is independent of $n$.

  \item Assumption (S3) is equivalent to that the conjugacy class
      $(\As\rho_\pi)(-1)(\Fr_\ell)$ is semisimple, whose mod $p$
      reduction contains none of
      \[\left(
          \begin{array}{cccc}
            \mu &  &  &  \\
             & 1 &  &  \\
             &  & -1 & \\
             &  &  & \mu^{-1} \\
          \end{array}
        \right),\quad
        \left(
          \begin{array}{cccc}
            \pm\ell &  &  &  \\
             & 1 &  &  \\
             &  & -1 & \\
             &  &  & \pm\ell^{-1} \\
          \end{array}
        \right)
      \] for all $\mu\in\dF_p^\times$ of order in $\{1,2,3,4,6\}$.
\end{enumerate}
\end{remark}

When $\wp(N^-)$ is even, let $\ell_1$ and $\ell_2$ be two distinct strongly
$n$-admissible primes, and put $N^\sim=N^-\ell_1\ell_2$. When $\wp(N^-)$ is
odd, let $\ell$ be a strongly $n$-admissible primes, and put
$N^\sim=N^-\ell$. We have the following composite map
\begin{align*}
\sQ_{\sigma,\pi}^{(d,\fd)}&\colon\rH^3_{\et}(Z_{N^+Mq,N^\sim;\dQ^\ac},\dZ_p(2))\\
&\to\rH^1_{\et}(Y_{N^+Mq,N^\sim;\dQ^\ac},\dZ/p^n(1))\otimes_{\dZ_p}\rH^2_{\et}(X_{N^+M;\dQ^\ac},\dZ_p(1))\\
&\to\rH^1_{\et}(Y_{N^+,N^\sim;\dQ^\ac},\dZ/p^n(1))\otimes_{\dZ_p}\rH^2_{\et}(X_{\fM;\dQ^\ac},\dZ_p(1))\\
&\to\bar\rT_\sigma^n\otimes_{\dZ_p}\pres{\sharp}{\rT}_\pi=\bar\rT_{\sigma,\pi}^n,
\end{align*}
where the second last arrow is induced by
$\delta^{(Mq,d)}_{N^+Mq,N^\sim}\otimes\gamma^{(N^+M/\fM,\fd)}_\fM$; and the
last one is induced by $\psi_{\ell_1,\ell_2}^n$ \eqref{eq:level_raising4}
(resp.\ $\psi_{\ell}^n$ \eqref{eq:level_raising4_bis}) when $\wp(N^-)$ is
even (resp.\ odd). Finally, we denote by
$\bar\Delta^n_{\sigma,\pi\res\ell_1,\ell_2}$ (resp.\
$\bar\Delta^n_{\sigma,\pi\res\ell}$) the class
\[\rH^1(\dQ,\sQ_{\sigma,\pi}^{(d,\fd)})
([\Delta'_{N^+M,N^\sim}]_{Z_{N^+Mq,N^\sim}}^n)\in\rH^1(\dQ,\bar\rT_{\sigma,\pi}^n)\]
when $\wp(N^-)$ is even (resp.\ odd).

\section{Congruences of Hirzebruch--Zagier classes}
\label{s4}

This chapter is the technical heart of the article, in which we propose and
prove some explicit congruence formulae for Hirzebruch--Zagier classes. In
\Sec \ref{ss:local_cohomology}, we study the local cohomology of the Galois
representation in question at various primes, using local Tate pairings. In
\Sec \ref{ss:integral_tate}, we state the result about the integral Tate
conjecture for special fibers of Hilbert modular surfaces at good inert
primes, in the version convenient for our use. In \Sec
\ref{ss:congruence_gross}, we state our theorems on explicit congruence
formulae for Hirzebruch--Zagier classes. The last two sections are devoted to
the proof of previous theorems, in which \Sec \ref{ss:computation_i} is
responsible for the case of smooth reduction and \Sec \ref{ss:computation_ii}
is responsible for the case of semistable reduction.

\subsection{Local cohomology and Tate duality}
\label{ss:local_cohomology}

We will proceed under Assumption \ref{as:group_r} (R1 -- R4) till the end of
this \emph{chapter}. In this section, we study various localizations of the
Galois cohomology $\rH^1(\dQ,\bar\rT_{\sigma,\pi}^n)$. Let $n\geq 1$ be an
integer.

The $\Gamma_\dQ$-invariant pairing
$\bar\rT_{\sigma,\pi}^n\times\bar\rT_{\sigma,\pi}^n\to\dZ/p^n(1)$
\eqref{eq:pairing} induces, for each prime power $v$, a local Tate pairing
\begin{align}\label{eq:tate_pairing}
\langle\;,\;\rangle_v\colon\rH^1(\dQ_v,\bar\rT_{\sigma,\pi}^n)\times\rH^1(\dQ_v,\bar\rT_{\sigma,\pi}^n)
\to\rH^2(\dQ_v,\dZ/p^n(1))\simeq \dZ/p^n.
\end{align}
In what follows, we will write $\langle s_1,s_2\rangle_v$ rather than
$\langle\loc_v(s_1),\loc_v(s_2)\rangle_v$ for
$s_1,s_2\in\rH^1(\dQ,\bar\rT_{\sigma,\pi}^n)$.

\begin{lem}\label{le:trivial_global}
The restriction of $\sum_v\langle\;,\;\rangle_v$ to
$\rH^1(\dQ,\bar\rT_{\sigma,\pi}^n)$ under the map $\prod_v\loc_v$ is a finite
sum and moreover equal to zero, where both the sum and the product are taken
over all primes $v$.
\end{lem}

\begin{proof}
It follows from the global class field theory, and the convention that $p$ is
odd.
\end{proof}

\subsubsection{The case $v\nmid p$}

\begin{proposition}\label{pr:cohomology_mixed}
Suppose that $v=\ell$ is a strongly $n$-admissible prime. Let
$\alpha\in\{1,2\}$ be an exponent.
\begin{enumerate}
  \item The natural map
      $\rH^1(\dQ_{\ell^\alpha},\bar\rT_\sigma^n\otimes_{\dZ/p^n}\pres{\sharp}{\bar\rT}_\pi^n[\ell^2|1])
      \to\rH^1(\dQ_{\ell^\alpha},\bar\rT_{\sigma,\pi}^n)$ is an
      isomorphism.

  \item The $\dZ/p^n$-module
      $\rH^1(\dQ_{\ell^\alpha},\bar\rT_{\sigma,\pi}^n)$ is free of rank
      $2\alpha$.

  \item The natural map
      $\rH^1(\dQ_{\ell^\alpha},\bar\rT_\sigma^n[\ell^2|1]\otimes_{\dZ/p^n}\pres{\sharp}{\bar\rT}_\pi^n[\ell^2|1])
      \to\rH^1(\dQ_{\ell^\alpha},\bar\rT_{\sigma,\pi}^n)$ is injective
      with the image
      $\rH^1_\unr(\dQ_{\ell^\alpha},\bar\rT_{\sigma,\pi}^n)$.

  \item The natural map
      $\rH^1(\dQ_{\ell^\alpha},\bar\rT_\sigma^n[\ell^2|\ell^2]\otimes_{\dZ/p^n}\pres{\sharp}{\bar\rT}_\pi^n[\ell^2|1])
      \to\rH^1_\sing(\dQ_{\ell^\alpha},\bar\rT_{\sigma,\pi}^n)$ is an
      isomorphism.

  \item The local Tate pairing $\langle\;,\;\rangle_{\ell^\alpha}$
      induces a perfect pairing between
      $\rH^1_\unr(\dQ_{\ell^\alpha},\bar\rT_{\sigma,\pi}^n)$ and
      $\rH^1_\sing(\dQ_{\ell^\alpha},\bar\rT_{\sigma,\pi}^n)$.

  \item The restriction map
      $\rH^1(\dQ_{\ell},\bar\rT_{\sigma,\pi}^n)\to\rH^1(\dQ_{\ell^2},\bar\rT_{\sigma,\pi}^n)$
      is injective.
\end{enumerate}
\end{proposition}

\begin{proof}
By Definition \ref{de:group_s}, in particular that
$\tr(\Fr_{\ell^2};\rV_\pi)\mod p\not\in\{2\ell,-2\ell\}$, the
$\dZ/p^n$-module $\pres{\sharp}{\bar\rT}_\pi^n$ has a unique free
$\Gamma_\ell$-stable submodule ${\hat\rT}_\pi^n$ (of rank $2$) such that
$\pres{\sharp}{\bar\rT}_\pi^n=\pres{\sharp}{\bar\rT}_\pi^n[\ell^2|1]\oplus{\hat\rT}_\pi^n$.
Moreover, since $\tr(\Fr_{\ell^2};\rV_\pi)\mod
p\not\in\{\ell^2+1,-\ell^2-1\}$, we have
\[(\bar\rT_\sigma^n\otimes_{\dZ/p^n}{\hat\rT}_\pi^n)[\ell^2|1]=
\Hom_{\dZ/p^n}(\bar\rT_\sigma^n\otimes_{\dZ/p^n}{\hat\rT}_\pi^n,\dZ/p^n(1))[\ell^2|1]=0.\]
In particular, we have
\[\rH^1(\dQ_\ell,\bar\rT_\sigma^n\otimes_{\dZ/p^n}{\hat\rT}_\pi^n)
=\rH^1(\dQ_{\ell^2},\bar\rT_\sigma^n\otimes_{\dZ/p^n}{\hat\rT}_\pi^n)=0.\]
Thus (1) follows. Note that the action of $\Fr_\ell$ on
$\bar\rT_\sigma^n\otimes_{\dZ/p^n}\pres{\sharp}{\bar\rT}_\pi^n[\ell^2|1]$ has
eigenvalues $\{1,-1,\ell,-\ell\}$ which are distinct even modulo $p$.
Therefore, (2--5) follow by the same reason as Lemma
\ref{le:cohomology_sigma}. Since $p$ is odd, (6) is obvious.
\end{proof}

\begin{lem}\label{le:trivial_unramified}
Let $v$ be a prime other than $p$.
\begin{enumerate}
  \item There is an integer $n_v\geq 0$ such that the image of the
      pairing \eqref{eq:tate_pairing} is annihilated by $p^{n_v}$ for
      every $n\geq 1$.

  \item The submodule $\rH^1_\unr(\dQ_v,\bar\rT_{\sigma,\pi}^n)$ is its
      own annihilator under the pairing $\langle\;,\;\rangle_v$
      \eqref{eq:tate_pairing}.
\end{enumerate}
\end{lem}

\begin{proof}
By Lemma \ref{le:semisimple}, we have $\rH^1(\dQ_v,\rV_{\sigma,\pi})=0$,
which implies that $\rH^1(\dQ_v,\rT_{\sigma,\pi})$ is torsion (and finitely
generated). Since we have the short exact sequence
\[\rH^1(\dQ_v,\rT_{\sigma,\pi})\to\rH^1(\dQ_v,\bar\rT_{\sigma,\pi}^n)
\to\rH^2(\dQ_v,\rT_{\sigma,\pi})_\tor,\] (1) follows. Part (2) is well-known.
\end{proof}

\begin{notation}\label{no:bad}
Put $n_{\r{bad}}=\max\{n_v\res v\text{ is a prime dividing
}qN^-\cD_{N^+M}\}$.
\end{notation}

\subsubsection{The case $v=p$}

We start with some general discussion.

\begin{definition}[see, for example, \cite{BM02}*{\Sec 3.1}]
Let $a\leq 0\leq b$ be two integers such that $b-a\leq p-2$. Let $\rT$ be a
finite $\dZ_p[\Gamma_p]$-module. We say that $\rT$ is a \emph{torsion
crystalline module} with Hodge--Tate weights in $[a,b]$ if $\rT=\rL/\rL'$
where $\rL'\subset\rL$ are stable lattices in a crystalline $p$-adic
representation of $\Gamma_p$ with Hodge--Tate weights in $[a,b]$.

Let $\rT$ be a finitely generated $\dZ_p$-module with a continuous action of
$\Gamma_p$. We say that $\rT$ is a \emph{crystalline module} with Hodge--Tate
weights in $[a,b]$ if for all integers $n\geq1$, the quotient $\rT/p^n\rT$ is
a torsion crystalline module with Hodge--Tate weights in $[a,b]$.
\end{definition}

\begin{definition}[see \cite{Niz93}*{\Sec 4}]\label{de:finite_integral}
Let $\rT$ be a crystalline module of $\Gamma_p$ with Hodge--Tate weights in
$[a,b]$. An element $s\in\rH^1(\dQ_p,\rT)$, represented by an extension
\[0\to\rT\to\rT_s\to\dZ_p\to0,\]
is \emph{finite} if $\rT_s$ is a crystalline module. Denote by
$\rH^1_f(\dQ_p,\rT)\subset\rH^1(\dQ_p,\rT)$ the subset of finite elements
which form a $\dZ_p$-submodule. Moreover, the assignment $\rH^1_f(\dQ_p,-)$
is functorial.
\end{definition}

\begin{remark}\label{re:finite}
By the above definition, if $\rT$ is a stable lattice in a crystalline
$p$-adic representation $\rV$ of $\Gamma_p$ with Hodge--Tate weights in
$[a,b]$, then $\rH^1_f(\dQ_p,\rT)$ is the inverse image of
$\rH^1_f(\dQ_p,\rV)$ under the natural map
$\rH^1(\dQ_p,\rT)\to\rH^1(\dQ_p,\rV)$.
\end{remark}

\begin{lem}\label{le:trivial_p}
Suppose that $p\geq 11$. Then $\rH^1_f(\dQ_p,\bar\rT_{\sigma,\pi}^n)$ is its
own annihilator under the local Tate pairing $\langle\;,\;\rangle_p$
\eqref{eq:tate_pairing}.
\end{lem}

\begin{proof}
We know that $\bar\rT_{\sigma,\pi}^n$ is a torsion crystalline module of
Hodge--Tate weights in $[-2,1]$. The lemma follows from
\cite{Niz93}*{Proposition 6.2} since $3\leq (11-2)/2$.
\end{proof}

\subsection{Integral Tate cycles}
\label{ss:integral_tate}

We study $\pi$-isotypic Tate cycles. Let $\ell$ be a strongly $1$-admissible
prime. Recall that we have a canonical isomorphism
\[\rZ^0(\cX^\ssl_{N^+M;\dF_{\ell^2}})\simeq\dZ[\cS_{N^+M}^\bullet\sqcup\cS_{N^+M}^\circ]
\simeq\dZ[\cS_{N^+M}]^{\oplus 2}.\] Put
\[\rZ_{\pi\res N^+M}=\dZ_p[\cS_{N^+M}^\bullet\sqcup\cS_{N^+M}^\circ]\cap(\dQ_p[\cS_{N^+M}][\pi^Q])^{\oplus 2},\]
which is a finite generated free $\dZ_p$-submodule of
$\rZ^0(\cX^\ssl_{N^+M;\dF_{\ell^2}})\otimes_\dZ\dZ_p$. The intersection
pairing induces a map
\begin{align}\label{eq:intersection}
\rH^2_{\et}(\cX_{N^+M;\dF_\ell^\ac},\dZ_p(1))^{\dF_{\ell^2}}\to
\Hom(\rZ^0(\cX^\ssl_{N^+M;\dF_{\ell^2}})\otimes_\dZ\dZ_p,\dZ_p).
\end{align}

\begin{proposition}
The restriction of \eqref{eq:intersection} induces an isomorphism
\begin{align}\label{eq:tate_class_integral}
\rH^2_{\et}(\cX_{N^+M;\dF_\ell^\ac},\dZ_p(1))^{\dF_{\ell^2}}\cap\rH_\pi(\cX_{N^+M;\dF_\ell^\ac})
\xrightarrow\sim\Hom(\rZ_{\pi\res N^+M},\dZ_p),
\end{align}
under which the action of $\Fr_\ell$ on the source coincides with the
involution induced by switching two factors on the target.
\end{proposition}

\begin{proof}
This is a very special case of the main result of \cite{TX14}. See
\cite{TX14}*{\Sec 1.1} for this particular case. In fact, the proposition is
a direct consequence of Proposition \ref{pr:surface_supersingular} (4)
together with some Hecke theoretical argument. Note that the reason we may
replace the coefficient $\dQ_p$ by $\dZ_p$ is that
$\tr(\Fr_{\ell^2};\rV_\pi)\mod p\not\in\{2\ell,-2\ell\}$ in Definition
\ref{de:group_s} (S2).
\end{proof}

In what follows, we will present an element of $\rZ_\pi$ as a pair of
functions $f=(f^\bullet,f^\circ)$ on
$\cS_{N^+M}^\bullet\sqcup\cS_{N^+M}^\circ$. In particular, we have the
element
\[f_{\nu^\bullet,\nu^\circ}\colonequals (\nu^\bullet f_\pi^{(\fd)},\nu^\circ f_\pi^{(\fd)})
=(\nu^\bullet(\gamma^{(N^+M/\fM,\fd)}_{N^+M})^*f_\pi,\nu^\circ(\gamma^{(N^+M/\fM,\fd)}_{N^+M})^*f_\pi)\]
for $\nu^\bullet,\nu^\circ\in\dZ_p$. We define
\[\rZ_\pi=\{f_{\nu^\bullet,\nu^\circ}\res\nu^\bullet,\nu^\circ\in\dZ_p\}\subset\rZ_{\pi\res N^+M}\]
as a free submodule of rank $2$. Then the projection $\Hom(\rZ_{\pi\res
N^+M},\dZ_p)\to\Hom(\rZ_\pi,\dZ_p)$ coincides with the map
\[\rH^2_{\et}(\cX_{N^+M;\dF_\ell^\ac},\dZ_p(1))^{\dF_{\ell^2}}\cap\rH_\pi(\cX_{N^+M;\dF_\ell^\ac})
\xrightarrow{(\gamma^{(N^+M,\fd)}_{N^+M})_*}\pres{\sharp}{\rT}_\pi[\ell^2\res
1]\] under \eqref{eq:tate_class_integral}. In particular, we have an isomorphism
\begin{align}\label{eq:tate_class}
\tc_{\ell}\colon\pres{\sharp}{\bar\rT}_\pi^n[\ell^2|1]\xrightarrow\sim\Hom(\bar\rZ_\pi^n,\dZ/p^n)
\end{align}
for $1\leq n\leq\infty$.

\subsection{Statement of congruence formulae}
\label{ss:congruence_gross}

The following two theorems are the main technical results of the article. We
recall the two maps $\tu_\ell$ and $\ts_\ell$ in \eqref{eq:us}.

\begin{theorem}\label{th:congruence}
Suppose that Assumption \ref{as:group_r} (R1 -- R4) are satisfied and
$\wp(N^-)$ is even. Let $n\geq1$ be an integer and $\ell_1,\ell_2$ be two
distinct strongly $n$-admissible primes. Then
\begin{enumerate}
  \item $\loc_{\ell_1^2}(\bar\Delta^n_{\sigma,\pi})$ belongs to
      $\rH^1_\unr(\dQ_{\ell_1^2},\bar\rT_{\sigma,\pi}^n)\simeq
      \rH^1_\unr(\dQ_{\ell_1^2},\bar\rT_\sigma^n)\otimes_{\dZ/p^n}\pres{\sharp}{\bar\rT}_\pi^n[\ell_1^2|1]$;

  \item as an element in $\Hom(\bar\rZ_\pi^n,\dZ/p^n)$, the map
      $(\tu_{\ell_1}\otimes\tc_{\ell_1})\loc_{\ell_1^2}(\bar\Delta^n_{\sigma,\pi})$
      sends $f_{\nu^\bullet,\nu^\circ}$ to
      \[(\nu^\bullet+\epsilon_\sigma(\ell_1)\nu^\circ)\sum_{t\in\cT_{N^+M,N^-\ell_1}}
      (\zeta_{N^+M,N^-\ell_1}^*f_\pi^{(\fd)})(t)\cdot((\delta^{N^+M,d}_{N^+M,N^-\ell_1})^*g_{\sigma\res\ell_1}^n)(t)\]
      viewed as in $\dZ/p^n$, where $\delta^{N^+M,d}_{N^+M,N^-\ell_1}$ is
      defined in Notation \ref{no:order_rational} (5); and
      $g_{\sigma\res\ell_1}^n$ is defined in \eqref{eq:level_raising3};

  \item as an element in $\Hom(\bar\rZ_\pi^n,\dZ/p^n)$, the map
      $(\ts_{\ell_2}\otimes\tc_{\ell_2})
      \partial_{\ell_2^2}\loc_{\ell_2^2}(\bar\Delta^n_{\sigma,\pi\res\ell_1,\ell_2})$
      sends $f_{\nu^\bullet,\nu^\circ}$ to
      \begin{align*}
      &(\tr(\Fr_{\ell_2^2};\rV_\pi)-2\epsilon_\sigma(\ell_2)\ell_2)
      (\nu^\circ+\epsilon_\sigma(\ell_2)\nu^\bullet) \\
      &\times\sum_{t\in\cT_{N^+M,N^-\ell_1}}
      (\zeta_{N^+M,N^-\ell_1}^*f_\pi^{(\fd)})(t)\cdot((\delta^{N^+M,d}_{N^+M,N^-\ell_1})^*g_{\sigma\res\ell_1}^n)(t)
      \end{align*}
      viewed as in $\dZ/p^n$.
\end{enumerate}
\end{theorem}

\begin{theorem}\label{th:congruence_bis}
Suppose that Assumption \ref{as:group_r} (R1 -- R4) are satisfied and
$\wp(N^-)$ is odd. Let $n\geq1$ be an integer and $\ell$ be a strongly
$n$-admissible prime. Then as an element in $\Hom(\bar\rZ_\pi^n,\dZ/p^n)$,
the map $(\ts_{\ell}\otimes\tc_{\ell})
\partial_{\ell^2}\loc_{\ell^2}(\bar\Delta^n_{\sigma,\pi\res\ell})$
sends $f_{\nu^\bullet,\nu^\circ}$ to
\begin{align*}
&(\tr(\Fr_{\ell^2};\rV_\pi)-2\epsilon_\sigma(\ell)\ell)
(\nu^\circ+\epsilon_\sigma(\ell)\nu^\bullet) \\
&\times\sum_{t\in\cT_{N^+M,N^-}}
(\zeta_{N^+M,N^-}^*f_\pi^{(\fd)})(t)\cdot((\delta^{N^+M,d}_{N^+M,N^-})^*g_\sigma^n)(t)
\end{align*}
viewed as in $\dZ/p^n$, where $g_\sigma^n$ is defined in
\eqref{eq:level_raising3_bis}.
\end{theorem}

Note that if $\ell$ is strongly $1$-admissible, then
$\tr(\Fr_{\ell^2};\rV_\pi)\pm2\ell\in\dZ_p^\times$ by Definition
\ref{de:group_s} (S3). We will prove Theorem  \ref{th:congruence} (1, 2) at
the end of \Sec \ref{ss:computation_i} and Theorem \ref{th:congruence} (3),
Theorem \ref{th:congruence_bis} at the end of \Sec \ref{ss:computation_ii}.
In fact, among the three congruence formulae: Theorem \ref{th:congruence}
(2), (3), and Theorem \ref{th:congruence_bis}, the latter two are proved by
the same way. However, we would like to pack them according to the parity of
$\wp(N^-)$ as appeared above.

\subsection{Computation of localization, I}
\label{ss:computation_i}

Let $R$ be a discrete valuation ring in which $p$ is invertible. Let $k$
(resp.\ $\tilde{k}$) be its fraction (resp.\ residue) field. Let $X$ (resp.\
$Y$) be a smooth proper (resp.\ flat proper) scheme over $R$. Assume that the
generic fiber $X_k$ (resp.\ $Y_k$) is smooth purely of dimension $2d$ (resp.\
$2d-1$) for some $d\geq 1$. Put $Z=Y\times_{\Spec R}X$ with $\pi_X\colon Z\to
X$ and $\pi_Y\colon Z\to Y$ the two projections.

For $1\leq n\leq \infty$, denote by $\Theta^n$ the composition of the
following maps:
\begin{align}\label{eq:theta}
\Theta^n&\colon\rH^1(k,\rH^{4d-1}_{\et}(Z_{k^\ac},\dZ/p^n(2d))) \\
&\to\rH^1(k,\Hom(\rH^{2d}_{\et}(X_{k^\ac},\dZ/p^n(d)),\rH^{2d-1}_{\et}(Y_{k^\ac},\dZ/p^n(d)))) \notag
\\
&\to\rH^1(k,\Hom(\rH^{2d}_{\et}(X_{k^\ac},\dZ/p^n(d))^k,\rH^{2d-1}_{\et}(Y_{k^\ac},\dZ/p^n(d)))) \notag
\\
&\xrightarrow\sim\Hom(\rH^{2d}_{\et}(X_{k^\ac},\dZ/p^n(d))^k,\rH^1(k,\rH^{2d-1}_{\et}(Y_{k^\ac},\dZ/p^n(d)))) \notag
\\
&\xrightarrow\sim\Hom(\rH^{2d}_{\et}(X_{\tilde{k}^\ac},\dZ/p^n(d))^{\tilde{k}},\rH^1(k,\rH^{2d-1}_{\et}(Y_{k^\ac},\dZ/p^n(d))))\notag \\
&\to\Hom(\CH^d(X_{\tilde{k}},\dZ/p^n),\rH^1(k,\rH^{2d-1}_{\et}(Y_{k^\ac},\dZ/p^n(d)))), \notag
\end{align}
where the first arrow is induced by K\"{u}nneth decomposition and
Poincar\'{e} duality; and the last arrow is induced by
$\cl_{X_{\tilde{k}}}^0$. Here we recall the convention from \Sec
\ref{ss:galois_modules} that $-^k$ is nothing but $-^{\Gamma_k}$.

\subsubsection{Smooth reduction in general case}

Now assume that $Y$ is also smooth over $R$, and hence $Z$ is smooth proper
over $R$.

\begin{lem}\label{le:comparison_smooth}
Let $W\in\rZ^c(Z)$ be a codimension $c$ cycle on $Z$ such that its support
$W'$, equipped with the induced reduced structure, is normal and flat over
$R$, and the class $\cl_{Z_k}(W_k)$ belongs to
$\rH^{2c}_{\et}(Z_k,\dZ/p^n(c))^0$. Then
\begin{enumerate}
  \item the class $\cl_{Z_{\tilde{k}}}(W_{\tilde{k}})$ belongs to
      $\rH^{2c}_{\et}(Z_{\tilde{k}},\dZ/p^n(c))^0$;

  \item the class $[W_k]_{Z_k}^n$ belongs to
      $\rH^1_\unr(k,\rH^{2c-1}_{\et}(Z_{k^\ac},\dZ/p^n(c)))$;

  \item $[W_k]_{Z_k}^n=[W_{\tilde{k}}]_{Z_{\tilde{k}}}^n$ under the
      canonical isomorphism
      \[\rH^1_\unr(k,\rH^{2c-1}_{\et}(Z_{k^\ac},\dZ/p^n(c)))\simeq\rH^1(\tilde{k},\rH^{2c-1}_{\et}(Z_{\tilde{k}^\ac},\dZ/p^n(c))).\]
\end{enumerate}
\end{lem}

This lemma holds for an arbitrary scheme $Z$ that is smooth proper over $R$,
not necessarily of the form $X\times_{\Spec R}Y$.

\begin{proof}
For a noetherian scheme $S$, we denote by $S_\rs$ the singular locus of
$S_{\r{red}}$, which is a Zariski closed subset containing no generic point,
equipped with the induced reduced structure.

We have $W'_{\rs;k}=(W'_k)_\rs$. Define $\pres{\rs}{W}'_{\tilde{k}}$ to be
the union of $W'_{\rs;\tilde{k}}$ and $(W'_{\tilde{k}})_\rs$. Since $W'$ is
normal and flat proper over $R$, the subset $\pres{\rs}{W}'_{\tilde{k}}$ is
closed in $W'_{\tilde{k}}$ of codimension everywhere at least $1$. Let $I$
(resp.\ $\tilde{I}$) be the (finite) set of irreducible components of $W'$
(resp.\ $W'_{\tilde{k}}$). We have a natural specialization map
$\r{sp}\colon\bigoplus_I\dZ\to\bigoplus_{\tilde{I}}\dZ$. If we regard the
source (resp.\ target) as the subgroup of $\rZ^c(Z)$ (resp.\
$\rZ^c(Z_{\tilde{k}})$) consisting of cycles supported on $W'$ (resp.\
$W'_{\tilde{k}}$), then $\r{sp}$ sends $W$ to $W_{\tilde{k}}$. Thus, we have
the following commutative diagram
\[\xymatrix{
\bigoplus_I\dZ/p^n &  \bigoplus_I\dZ/p^n \ar[r]^-{\r{sp}\mod p^n}\ar@{=}[l] & \bigoplus_{\tilde{I}}\dZ/p^n \\
\rH^{2c}_{W'_k\setminus(W'_k)_\rs}(Z_k\setminus(W'_k)_\rs) \ar[u]^-{\simeq} &&
\rH^{2c}_{W'_{\tilde{k}}\setminus\pres{\rs}{W}'_{\tilde{k}}}(Z_{\tilde{k}}\setminus\pres{\rs}{W}'_{\tilde{k}}) \ar[u]_-{\simeq} \\
\rH^{2c}_{(W'\setminus W'_\rs)_k}((Z\setminus W'_\rs)_k) \ar@{=}[u] &
\rH^{2c}_{W'\setminus W'_\rs}(Z\setminus W'_\rs) \ar[uu]_-{\simeq}\ar[l]\ar[r] &
\rH^{2c}_{(W'\setminus W'_\rs)_{\tilde{k}}}((Z\setminus W'_\rs)_{\tilde{k}}) \ar[u]_-{\simeq} \\
\rH^{2c}_{W'_k}(Z_k) \ar[u]^-{\simeq}\ar[d] & \rH^{2c}_{W'}(Z) \ar[u]_-{\simeq}\ar[l]\ar[r]\ar[d] &
\rH^{2c}_{W'_{\tilde{k}}}(Z_{\tilde{k}}) \ar[u]_-{\simeq}\ar[d] \\
\rH^{2c}_{\et}(Z_k) \ar[d] & \rH^{2c}_{\et}(Z) \ar[l]\ar[r]\ar[d] &
\rH^{2c}_{\et}(Z_{\tilde{k}}) \ar[d] \\
\rH^{2c}_{\et}(Z_{k^\ac}) & \rH^{2c}_{\et}(Z_{R^\r{sh}})
\ar[l]_-{\simeq}\ar[r]^-{\simeq} & \rH^{2c}_{\et}(Z_{\tilde{k}^\ac}), }\] in
which the coefficient sheaf for all \'{e}tale cohomology groups is
$\dZ/p^n(c)$. We have that
\begin{itemize}
  \item the last three vertical arrows are all restriction maps;
  \item the three arrows toward the first row are all isomorphisms by
      purity;
  \item the remaining four upward arrows are all isomorphisms by
      semi-purity; and
  \item the two bottom arrows are isomorphisms since $Z$ is smooth proper
      over $R$, where $R^{\r{sh}}$ is the strict henselization of $R$.
\end{itemize}
Recall how we define the cycle class map. If we start from $W$, viewed as an
element in the middle term $\bigoplus_I\dZ/p^n$, and move along three routes
(left-down, down, right-down) all the way down to the second last row, then
what we obtain are $\cl_{Z_k}(W_k)$, $\cl_Z(W)$ and
$\cl_{Z_{\tilde{k}}}(W_{\tilde{k}})$, respectively. All these classes map to
$0$ under the last vertical arrows which are restriction maps. The lemma then
follows by the functoriality of Hochschild--Serre spectral sequences.
\end{proof}

\begin{proposition}\label{pr:theta}
Let $W\in\rZ^{2d}(Z)$ be a codimension $2d$ cycle on $Z$ such that its
support is normal and flat proper over $R$, and
$\cl_{Z_k}(W_k)\in\rH^{4d}_{\et}(Z_k,\dZ/p^n(2d))^0$. For
$c\in\CH^d(X_{\tilde{k}},\dZ/p^n)$, the map $\Theta^n[W_k]_{Z_k}^n$ sends $c$
to
\begin{multline*}
[(\pi_{Y;\tilde{k}})_*(\pi_{X;\tilde{k}}^*c.W_{\tilde{k}})]_{Y_{\tilde{k}}}^n
\in\rH^1(\tilde{k},\rH^{2d-1}_{\et}(Y_{\tilde{k}^\ac},\dZ/p^n(d)))\\
\simeq\rH^1_\unr(k,\rH^{2d-1}_{\et}(Y_{k^\ac},\dZ/p^n(d)))
\subset\rH^1(k,\rH^{2d-1}_{\et}(Y_{k^\ac},\dZ/p^n(d))).
\end{multline*}
\end{proposition}

\begin{proof}
By Lemma \ref{le:comparison_smooth}, we may replace $\Theta^n[W_k]_{Z_k}^n$
by $\Theta^n[W_{\tilde{k}}]_{Z_{\tilde{k}}}^n$. Moreover, it is clear that
the Chow cycle $\pi_{Y*}(\pi_{X;\tilde{k}}^*c.W_{\tilde{k}})$ belongs to
$\CH^{2d}(Y_{\tilde{k}},\dZ/p^n)^0$. Therefore, the expression
$[(\pi_{Y*}(\pi_{X;\tilde{k}}^*c.W_{\tilde{k}})]_{Y_{\tilde{k}}}^n$ is
well-defined.

By definition and Lemma \ref{le:multiplicative_structure}, the composition of
the following maps
\begin{align*}
&\rH^{4d}_{\et}(Z_{\tilde{k}},\dZ/p^n(2d))^0\otimes\rH^{2d}_{\et}(X_{\tilde{k}},\dZ/p^n(d))\\
&\to\rH^1(\tilde{k},\rH^{4d-1}_{\et}(Z_{\tilde{k}^\ac},\dZ/p^n(2d)))
\otimes\rH^{2d}_{\et}(X_{\tilde{k}^\ac},\dZ/p^n(d))^{\tilde{k}}\\
&\to\rH^1(\tilde{k},\Hom(\rH^{2d}_{\et}(X_{\tilde{k}^\ac},\dZ/p^n(d)),\rH^{2d-1}_{\et}(Y_{\tilde{k}^\ac},\dZ/p^n(d))))
\otimes\rH^{2d}_{\et}(X_{\tilde{k}^\ac},\dZ/p^n(d))^{\tilde{k}} \\
&\to\Hom(\rH^{2d}_{\et}(X_{\tilde{k}^\ac},\dZ/p^n(d))^{\tilde{k}},\rH^1(\tilde{k},\rH^{2d-1}_{\et}(Y_{\tilde{k}^\ac},\dZ/p^n(d))))
\otimes\rH^{2d}_{\et}(X_{\tilde{k}^\ac},\dZ/p^n(d))^{\tilde{k}}\\
&\to\rH^1(\tilde{k},\rH^{2d-1}_{\et}(Y_{\tilde{k}^\ac},\dZ/p^n(d))),
\end{align*}
in which the first arrow is $\xi_{Z_{\tilde{k}}}^1\otimes\r{id}$ and the last
one is the evaluation map, coincides with the following composition of maps
\begin{multline*}
\rH^{4d}_{\et}(Z_{\tilde{k}},\dZ/p^n(2d))^0\otimes\rH^{2d}_{\et}(X_{\tilde{k}},\dZ/p^n(d))
\xrightarrow{\r{id}\otimes\pi_Z^*}\rH^{4d}_{\et}(Z_{\tilde{k}},\dZ/p^n(2d))^0\otimes\rH^{2d}_{\et}(Z_{\tilde{k}},\dZ/p^n(d)) \\
\xrightarrow{\cup}\rH^{6d}_{\et}(Z_{\tilde{k}},\dZ/p^n(3d))^0
\xrightarrow{\pi_{Y*}}\rH^{2d}_{\et}(Y_{\tilde{k}},\dZ/p^n(d))^0
\xrightarrow{\xi_{Y_{\tilde{k}}}^1}\rH^1(\tilde{k},\rH^{2d-1}_{\et}(Y_{\tilde{k}^\ac},\dZ/p^n(d))).
\end{multline*}
The proposition follows since cycle class maps are multiplicative.
\end{proof}

\subsubsection{Proof of Theorem \ref{th:congruence} (1, 2)}

Let $n\geq 1$ be an integer and $\ell=\ell_1$ be a strongly $n$-admissible
prime. We apply the previous discussion to the case: $R=\dZ_{\ell^2}$ (and
hence $k=\dQ_{\ell^2}$, $\tilde{k}=\dF_{\ell^2}$),
$X=\cX_{N^+M;\dZ_{\ell^2}}$, $Y=\cY_{N^+Mq,N^-;\dZ_{\ell^2}}$ (and hence
$d=1$), and $W=\b\Delta'_{N^+M,N^-;\dZ_{\ell^2}}$. Consider the following
diagram
\begin{align}\label{eq:diagram_smooth}
\xymatrix{
\rH^1(\dQ_{\ell^2},\rH^3_{\et}(Z_{\dQ_{\ell}^\ac},\dZ/p^n(2))) \ar[d] \\
\rH^1(\dQ_{\ell^2},\rH^2_{\et}(X_{\dQ_{\ell}^\ac})\otimes\rH^1_{\et}(Y_{\dQ_{\ell}^\ac}))
\ar[r]\ar[d] & \rH^1(\dQ_{\ell^2},\pres{\sharp}{\bar\rT_\pi}^n\otimes\bar\rT_\sigma^n) \ar[d]^-{\simeq} \\
\rH^1(\dQ_{\ell^2},\Hom(\CH^1(X_{\dF_{\ell^2}},\dZ/p^n),\rH^1_{\et}(Y_{\dF_{\ell}^\ac})))
\ar[r]\ar[d]_-{\simeq}& \rH^1(\dQ_{\ell^2},\Hom(\bar\rZ_\pi^n,\bar\rT_\sigma^n))\ar[d]^-{\simeq}  \\
\Hom(\CH^1(X_{\dF_{\ell^2}},\dZ/p^n),\rH^1(\dQ_{\ell^2},\rH^1_{\et}(Y_{\dF_{\ell}^\ac})))
\ar[r]& \Hom(\bar\rZ_\pi^n,\rH^1(\dQ_{\ell^2},\bar\rT_\sigma^n)),}
\end{align}
in which
\begin{itemize}
  \item the coefficient sheaf of an \'{e}tale cohomology group is
      $\dZ/p^n(1)$ if not specified;

  \item the first horizontal arrow is induced by
      $\gamma^{N^+M/\fM,\fd}_{N^+M}\otimes\delta^{(Mq,d)}_{N^+Mq,N^-}$
      and the projection to the $(\sigma,\pi)$-component;

  \item the first right vertical arrow is induced by
      \eqref{eq:tate_class}, which is an isomorphism by Proposition
      \ref{pr:cohomology_mixed} (1);

  \item the composition of all left vertical arrows is simply $\Theta^n$
      \eqref{eq:theta}.
\end{itemize}
Denote the composite map in the above diagram by
\[\Theta_{\sigma,\pi}^n\colon\rH^1(\dQ_{\ell^2},\rH^3_{\et}(Z_{\dQ_{\ell}^\ac},\dZ/p^n(2)))
\to\Hom(\bar\rZ_\pi^n,\rH^1(\dQ_{\ell^2},\bar\rT_\sigma^n)).\] By construction, the element
$\Theta_{\sigma,\pi}^n[W_{\dQ_{\ell^2}}]^n_{Z_{\dQ_{\ell^2}}}$ is nothing but
$\tc_\ell\loc_{\ell^2}(\bar\Delta^n_{\sigma,\pi})$. The following lemma
confirms Theorem \ref{th:congruence} (1).

\begin{lem}\label{le:congruence_smooth}
The map $\Theta_{\sigma,\pi}^n[W_{\dQ_{\ell^2}}]^n_{Z_{\dQ_{\ell^2}}}$
factorizes through the natural map
\[\dZ/p^n[\cT_{N^+,N^-\ell}]^0\to\rH^1_\unr(\dQ_{\ell^2},\bar\rT_\sigma^n)
\to\rH^1(\dQ_{\ell^2},\bar\rT_\sigma^n),\] where the first arrow appears in
\eqref{eq:level_raising3_pre}.
\end{lem}

\begin{proof}
This is a corollary of Proposition \ref{pr:theta}. Indeed, combining with
Lemma \ref{le:proper_intersection}, Proposition
\ref{pr:special_supersingular} and the definition of $\b\Delta'_{N^+M,N^-}$,
we have for $f_{\nu^\bullet,\nu^\circ}\in\bar\rZ_\pi^n$,
\begin{multline*}
\pi_{Y*}(\pi_{X;\tilde{k}}^*f_{\nu^\bullet,\nu^\circ}.W_{\tilde{k}})\\
=\frac{1}{q+1-\tr(\Fr_q;\rV_\sigma)}
(\zeta_{N^+M,N^-\ell}\circ(\nu^\bullet\r{id}+\nu^\circ\op_{\ell})
\circ(\delta^{(q,1)}_{N^+Mq,N^-\ell}-\delta^{(q,q)}_{N^+Mq,N^-\ell}))^*f_\pi^{(\fd)}
\end{multline*}
viewed as an element in
$\dZ/p^n[\cT_{N^+Mq,N^-\ell}]^0\subset\CH^1(Y_{\dF_{\ell^2}},\dZ/p^n)^0$. The
lemma follows by taking the projection $\delta^{(M,d)}_{N^+Mq,N^-\ell}$.
\end{proof}

From the above proof, we have
\begin{align*}
&\tu_{\ell}\Theta_{\sigma,\pi}^n[W_{\dQ_{\ell^2}}]^n_{Z_{\dQ_{\ell^2}}}(f_{\nu^\bullet,\nu^\circ})\\
&=\frac{1}{q+1-\tr(\Fr_q;\rV_\sigma)}\sum_{t\in\cT_{N^+Mq,N^-\ell}}
((\delta^{(Mq,d)}_{N^+Mq,N^-\ell})^*g_{\sigma\res\ell}^n)(t) \\
&\times((\zeta_{N^+M,N^-\ell}\circ(\nu^\bullet\r{id}+\nu^\circ\op_{\ell})
\circ(\delta^{(q,1)}_{N^+Mq,N^-\ell}-\delta^{(q,q)}_{N^+Mq,N^-\ell}))^*f_\pi^{(\fd)})(t) \\
&=\sum_{t\in\cT_{N^+M,N^-\ell}}
((\delta^{(M,d)}_{N^+M,N^-\ell})^*g_{\sigma\res\ell}^n)(t)\cdot
((\zeta_{N^+M,N^-\ell}\circ(\nu^\bullet\r{id}+\nu^\circ\op_{\ell}))^*f_\pi^{(\fd)})(t).
\end{align*}
Theorem \ref{th:congruence} (2) follows since
$g_{\sigma\res\ell}^n\circ\op_{\ell}=\epsilon_\sigma(\ell)g_{\sigma\res\ell}^n$.

\subsection{Computation of localization, II}
\label{ss:computation_ii}

We keep the notation from the previous section. In addition, we assume that
$R$ is a \emph{strictly henselian} discrete valuation ring, and $d=1$. In
particular, the morphism $f\colon Y\to\Spec R$ is a flat proper curve with
smooth generic fiber. Let $n\geq 1$ be an integer.

\subsubsection{SNC reduction in general case}

\begin{definition}
We say that a flat proper curve $f\colon Y\to\Spec R$ is an \emph{SNC curve}
if
\begin{enumerate}
  \item $f_*\sO_Y=\sO_{\Spec R}$;

  \item $Y$ is regular and $Y_k$ is smooth;

  \item all irreducible components of $Y_{\tilde{k}}$ are smooth and
      their intersection points are nodal.
\end{enumerate}
\end{definition}

Let $Y/R$ be an SNC curve. In this paragraph, we recall some facts from
\cite{BLR90}*{\Sec 9.5 \& 9.6}. Denote by $\Pic_{Y/R}$ the relative Picard
functor of $Y/R$, $\rP_{Y/R}$ the open subfunctor of $\Pic_{Y/R}$ defined by
line bundles of total degree $0$. By \cite{BLR90}*{\Sec 9.5, Theorem 4}, the
largest separated quotient $\rJ_{Y/R}$ of $\rP_{Y/R}$ is representable, which
is in fact a N\'{e}ron model of the Jacobian of $Y_k$ over $\Spec R$. Let
$\Phi_{Y/R}$ be the group of connected components of $\rJ_{Y/R;\tilde{k}}$.
Denote by $\{Y_i\}_{i\in I}$ the set of irreducible components of
$Y_{\tilde{k}}$. Consider maps
\[\dZ[I]\xrightarrow{\alpha}\dZ[I]\xrightarrow{\deg}\dZ,\]
where $\alpha$ is induced by the intersection matrix $(Y_i.Y_j)_{i,j\in I}$
and $\deg$ is the degree map. Then $\Phi_{Y/R}$ is canonically isomorphic to
$\Ker\deg/\IM\alpha$. The projection map
$\phi\colon\rP_{Y/R}(\tilde{k})\to\Phi_{Y/R}$ is computed by the formula
$\phi(L)=(L.Y_i)_{i\in I}$ for every line bundle $L$ on $Y_{\tilde{k}}$ of
total degree $0$. We also have a reduction map $\rJ_{Y/R}(k)\to\Phi_{Y/R}$.
Denote by $\Pic_{Y/R}(\tilde{k})_n$ the subset of $L\in\Pic_{Y/R}(\tilde{k})$
such that $p^n\mid\sum_{i\in I}L.Y_i$. Then the natural map
\[\rP_{Y/R}(\tilde{k})/p^n\rP_{Y/R}(\tilde{k})\to\Pic_{Y/R}(\tilde{k})_n/p^n\Pic_{Y/R}(\tilde{k})\]
is a bijection. Therefore, the map $\phi$ induces the following map
\[\phi_n\colon\Pic_{Y/R}(\tilde{k})_n/p^n\Pic_{Y/R}(\tilde{k})\to\Phi_{Y/R}/p^n\Phi_{Y/R}.\]

On the other hand, we have the Kummer map
\[\rP_{Y/R}(k)/p^n\rP_{Y/R}(k)=\rJ_{Y/R}(k)/p^n\rJ_{Y/R}(k)\to\rH^1(k,\rJ_{Y/R}[p^n](k^\ac)).\]
If we denote by $\Pic_{Y/R}(k)_n$ the subset of $L\in\Pic_{Y/R}(k)$ with
degree divisible by $p^n$, then the natural map
\[\rP_{Y/R}(k)/p^n\rP_{Y/R}(k)\to\Pic_{Y/R}(k)_n/p^n\Pic_{Y/R}(k)\]
is a bijection. Therefore, the above Kummer map induces the following map
\begin{align}\label{eq:kummer_map}
\Pic_{Y/R}(k)_n/p^n\Pic_{Y/R}(k)\to\rH^1(k,\rJ_{Y/R}[p^n](k^\ac))
\end{align}

\begin{lem}\label{le:kummer_singular}
The map \eqref{eq:kummer_map} factorizes through the (surjective) reduction
map
\[\Pic_{Y/R}(k)_n/p^n\Pic_{Y/R}(k)\simeq\rJ_{Y/R}(k)/p^n\rJ_{Y/R}(k)\to\Phi_{Y/R}/p^n\Phi_{Y/R}.\]
In particular, we have a natural map
$\kappa_n\colon\Phi_{Y/R}/p^n\Phi_{Y/R}\to\rH^1(k,\rJ_{Y/R}[p^n](k^\ac))$.
\end{lem}

\begin{proof}
Denote by $\rJ^0_{Y/R}$ the identity component of $\rJ_{Y/R}$. Then we have
an exact sequence
\[0 \to \rJ^0_{Y/R}(R) \to \rJ_{Y/R}(k) \to \Phi_{Y/R} \to 0.\]
Since $p$ is invertible on $R$, the multiplication morphism
$p^n\colon\rJ^0_{Y/R}\to\rJ^0_{Y/R}$ is \'{e}tale and surjective. Thus one
has the Kummer map for $\rJ^0_{Y/R}$ as well, which yields the following
commutative diagram
\[\xymatrix{
\rJ^0_{Y/R}(R)/p^n\rJ^0_{Y/R}(R) \ar[r]\ar[d] & \rH^1(\Spec R,\rJ^0_{Y/R}[p^n]) \ar[d] \\
\rJ_{Y/R}(k)/p^n\rJ_{Y/R}(k) \ar[r]& \rH^1(k,\rJ_{Y/R}[p^n](k^\ac)). }\]
Since $R$ is strictly henselian, the cohomology $\rH^1(\Spec
R,\rJ^0_{Y/R}[p^n])$ vanishes. The lemma follows since the image of the map
$\rJ^0_{Y/R}(R)/p^n\rJ^0_{Y/R}(R)\to\rJ_{Y/R}(k)/p^n\rJ_{Y/R}(k)$ coincides
with the kernel of $\rJ_{Y/R}(k)/p^n\rJ_{Y/R}(k)\to\Phi_{Y/R}/p^n\Phi_{Y/R}$.
\end{proof}

Denote by $\widehat\Theta^n$ the composition of the following maps:
\begin{align}\label{eq:vartheta}
&\widehat\Theta^n\colon\rH^1(k,\rH^3_{\et}(Z_{k^\ac},\dZ/p^n(2))) \\
&\to\rH^1(k,\Hom(\rH^2_{\et}(X_{k^\ac},\dZ/p^n(1)),\rH^1_{\et}(Y_{k^\ac},\dZ/p^n(1)))) \notag \\
&\xrightarrow\sim\rH^1(k,\Hom(\rH^2_{\et}(X_{k^\ac},\dZ/p^n(1)),\rJ_{Y,R}[p^n](k^\ac))) \notag \\
&\xrightarrow\sim\Hom(\rH^2_{\et}(X_{\tilde{k}},\dZ/p^n(1)),\rH^1(k,\rJ_{Y,R}[p^n](k^\ac)))\notag \\
&\to\Hom(\CH^1(X_{\tilde{k}},\dZ/p^n),\rH^1(k,\rJ_{Y,R}[p^n](k^\ac))), \notag
\end{align}
where we recall that $X/R$ is smooth proper.

Let $W=\sum a_iW_i$ be a cycle of $Z$ such that each $W_i$ is the graph of a
finite flat morphism $w_i\colon Y\to X$, such that
$W_k\in\CH^2(Z_k,\dZ/p^n)^0$. In particular, we have the element
$[W_k]_{Z_k}^n\in\rH^1(k,\rH^{3}_{\et}(Z_{k^\ac},\dZ/p^n(2)))$.

\begin{proposition}\label{pr:vartheta}
For $L\in\CH^1(X_{\tilde{k}},\dZ/p^n)$, the element $\sum a_i\cdot w_i^*L$ is
naturally in the quotient group
$\Pic_{Y/R}(\tilde{k})_n/p^n\Pic_{Y/R}(\tilde{k})$, and moreover
\[\widehat\Theta^n[W_k]_{Z_k}^n(L)=\kappa_n\phi_n\(\sum a_i\cdot w_i^*L\).\]
\end{proposition}

\begin{proof}
The first statement is immediate since $W_k\in\CH^2(Z_k,\dZ/p^n)^0$ and by
the compatibility of cycle class maps under specialization.

Since $\rH^2_{\et}(X_{\tilde{k}},\dZ/p^n(1))\simeq\rH^2_{\et}(X,\dZ/p^n(1))$,
we may regard $\cl_{X_{\tilde{k}}}(L)$ as an element in
$\rH^2_{\et}(X,\dZ/p^n(1))$, denoted by a new notation $[L]$ to distinguish.
Denote by $[L]_k$ the restriction of $[L]$ in $\rH^2_{\et}(X_k,\dZ/p^n(1))$.
Then $\sum a_i\cdot w_{i;k}^*[L]_k$ belongs to
$\rH^2_{\et}(Y_k,\dZ/p^n(1))^0$.

By a similar argument in the proof of Proposition \ref{pr:theta}, we find
that
\begin{align}\label{eq:vartheta1}
\widehat\Theta^n[W_k]_{Z_k}^n(L)=\xi_{Z_k}^1\(\sum a_i\cdot w_{i;k}^*[L]_k\).
\end{align}
Note that the following diagram
\[\xymatrix{
\rH^2_{\et}(X_k,\dZ/p^n(1)) \ar[rr]^-{\sum a_i\cdot w_{i;k}^*} &&
\rH^2_{\et}(Y_k,\dZ/p^n(1)) \\
\rH^2_{\et}(X,\dZ/p^n(1)) \ar[rr]^-{\sum a_i\cdot w_i^*} \ar[u] &&
\rH^2_{\et}(Y,\dZ/p^n(1)) \ar[u] }\] is commutative. Therefore,
\eqref{eq:vartheta1} is the restriction of $\sum a_i\cdot
w_{i}^*[L]\in\rH^2_{\et}(Y,\dZ/p^n(1))$ to $Y_k$, which belongs to
$\rH^2_{\et}(Y_k,\dZ/p^n(1))^0$, followed by $\xi_{Y_k}^1$. The proposition
follows from the lemma below.
\end{proof}

\begin{lem}
Let $Y$ be an SNC curve over $R$. Denote by $\rH^2_{\et}(Y,\dZ/p^n(1))^0$ the
inverse image of $\rH^2_{\et}(Y_k,\dZ/p^n(1))^0$ under the restriction map.
We have the following commutative diagram
\[\xymatrix{
\rH^2_{\et}(Y_k,\dZ/p^n(1))^0 \ar[r]^-{\xi_{Y_k}^1} & \rH^1(k,\rJ_{Y/R}[p^n](k^\ac)) \\
\rH^2_{\et}(Y,\dZ/p^n(1))^0 \ar[r]\ar[u] & \Phi_{Y/R}/p^n\Phi_{Y/R}
\ar[u]_-{\kappa_n}, }\] in which the bottom arrow sends
$c\in\rH^2_{\et}(Y,\dZ/p^n(1))^0$ to $(\deg c\res_{Y_i}\mod p^n)_{i\in I}$
which is naturally an element in $\Phi_{Y/R}/p^n\Phi_{Y/R}$.
\end{lem}

\begin{proof}
By \cite{SS10}*{Theorem 5.1}, the cycle class map
\[\cl_Y\colon\rZ^1(Y)\otimes_{\dZ}\dZ/p^n\to\rH^2_{\et}(Y,\dZ/p^n(1))\]
is surjective. Denote by $\rZ^1(Y,\dZ/p^n)^0$ the inverse image of
$\rH^2_{\et}(Y,\dZ/p^n(1))^0$ and $j\colon Y_k\to Y$ the open immersion. For
every $C\in\rZ^1(Y,\dZ/p^n)^0$, we may write $C=C_h+C_v$ where $C_v$ is
vertical and $C_h$ is horizontal. We have that $j^*\cl_Y(C_v)=0$ and
$\cl_Y(C_v)$ is mapped to $0$ (even) in $\Phi_{Y/R}$. For $C_h$, the
restriction $j^*C_h$ is an element in $\Pic_{Y/R}(k)$. Similarly to the proof
of Lemma \ref{le:comparison_smooth}, we have
$\cl_{Y_k}(j^*C_h)=j^*\cl_Y(C_h)$. In particular, the element $j^*C_h$ has
degree divisible by $p^n$. The rest follows from Lemma
\ref{le:kummer_singular}.
\end{proof}

\subsubsection{Proof of Theorem \ref{th:congruence} (3)}

Let $n\geq 1$ be an integer. We consider first the case where $\wp(N^-)$ is
even. Let $\ell_1,\ell_2$ be two distinct strongly $n$-admissible primes. We
apply the previous discussion to the case: $R=\rW(\dF_{\ell_2}^\ac)$
($k=R\otimes_\dZ\dQ$ and $\tilde{k}=\dF_{\ell_2}^\ac$), $X=\cX_{N^+M;R}$,
$Y=\cY_{N^+Mq,N^-\ell_1\ell_2;R}$, and
$W=\b\Delta'_{N^+M,N^-\ell_1\ell_2;R}$. Note that from
\eqref{eq:cycle_modified}, we have
\[W=\frac{1}{q+1-\tr(\Fr_q;\rV_\sigma)}W^\dag-\frac{1}{q+1-\tr(\Fr_q;\rV_\sigma)}W^\ddag,\]
where $W^\dag$ (resp.\ $W^\ddag$) is the graph of
$\zeta_{N^+Mq,N^-\ell_1\ell_2;R}^\dag$ (resp.\
$\zeta_{N^+Mq,N^-\ell_1\ell_2;R}^\ddag$).

Put $Y_1=\cY_{N^+,N^-\ell_1\ell_2;k}$ for short, and $\rJ_1$ the Jacobian of
$Y_1$. Similarly to \eqref{eq:diagram_smooth}, we have the following
commutative diagram.
\[\xymatrix{
\rH^1(k,\rH^3_{\et}(Z_{k^\ac},\dZ/p^n(2))) \ar[d]\ar[rd]^-{\widehat\Theta^n\eqref{eq:vartheta}} \\
\rH^1(k,\rH^2_{\et}(X_{k^\ac})\otimes\rH^1_{\et}(Y_{k^\ac})) \ar[r]\ar[d]&
\Hom(\CH^1(X_{\tilde{k}},\dZ/p^n),\rH^1(k,\rJ_{Y,R}[p^n](k^\ac))) \ar[d] \\
\rH^1(k,\rH^2_{\et}(X_{k^\ac})\otimes\rH^1_{\et}(Y_{1;k^\ac}))
\ar[r]\ar[d]& \Hom(\CH^1(X_{\tilde{k}},\dZ/p^n),\rH^1(k,\rJ_1[p^n](k^\ac))) \ar[d]\\
\rH^1(k,\pres{\sharp}{\bar\rT}_\pi^n\otimes\bar\rT_\sigma^n) \ar[r]^-\sim&
\Hom(\bar\rZ_\pi^n,\rH^1(k,\bar\rT_\sigma^n)) \ar[d]^-{\simeq} \\
& \Hom(\bar\rZ_\pi^n,\rH^1_\sing(\dQ_{\ell_2^2},\bar\rT_\sigma^n)),
}\] in which
\begin{itemize}
  \item the coefficient sheaf of an \'{e}tale cohomology group is
      $\dZ/p^n(1)$ if not specified;

  \item the two vertical arrows from the second row are both induced by
      $\delta^{(Mq,d)}_{N^+Mq,N^-\ell_1\ell_2}$;

  \item the two vertical arrows from the third row are both induced by
      $\gamma^{(N^+M/\fM,\fd)}_{N^+M}$, the projection to the
      $\pi$-component, and $\psi_{\ell_1,\ell_2}^n$
      \eqref{eq:level_raising4}.
\end{itemize}
Denote the composite map in the above diagram by
\[\widehat\Theta_{\sigma,\pi}^n\colon\rH^1(k,\rH^3_{\et}(Z_{k^\ac},\dZ/p^n(2)))
\to\Hom(\bar\rZ_\pi^n,\rH^1_\sing(\dQ_{\ell_2^2},\bar\rT_\sigma^n)).\] By
Proposition \ref{pr:vartheta}, for
$f\in\dZ/p^n[\cS_{N^+M}^\bullet\sqcup\cS_{N^+M}^\circ]$, we have
\begin{align*}
\widehat\Theta^n[W_k]^n_{Z_k}(f)=
\frac{1}{q+1-\tr(\Fr_q;\rV_\sigma)}\kappa_n\phi_n\(\zeta_{N^+Mq,N^-\ell_1\ell_2;\tilde{k}}^{\dag*}f
-\zeta_{N^+Mq,N^-\ell_1\ell_2;\tilde{k}}^{\ddag*}f\).
\end{align*}
Note that we have a natural multi-degree map
\[\Pic_{Y/R}(\tilde{k})\to\dZ[\cT_{N^+Mq,N^-\ell_1}^\bullet]\oplus\dZ[\cT_{N^+Mq,N^-\ell_1}^\circ],\]
which induces a map
\[\hat\phi_n\colon\Pic_{Y/R}(\tilde{k})/p^n\Pic_{Y/R}(\tilde{k})
\to\dZ/p^n[\cT_{N^+Mq,N^-\ell_1}^\bullet]\oplus\dZ/p^n[\cT_{N^+Mq,N^-\ell_1}^\circ]\]
rendering the following diagram
\[\xymatrix{
\Pic_{Y/R}(\tilde{k})_n/p^n\Pic_{Y/R}(\tilde{k}) \ar[r]^-{\hat\phi_n}\ar[d]_-{\phi_n}&
\dZ/p^n[\cT_{N^+Mq,N^-\ell_1}^\bullet\sqcup\cT_{N^+Mq,N^-\ell_1}^\circ]^0 \ar[d]^-{\phi^{(\ell_2)}_{N^+,N^-\ell_1\ell_2}\mod p^n} \\
\Phi_{Y/R}/p^n\Phi_{Y/R} \ar@{=}[r]&\Phi^{(\ell_2)}_{N^+Mq,N^-\ell_1\ell_2}/p^n\Phi^{(\ell_2)}_{N^+Mq,N^-\ell_1\ell_2}
}\] commute, where $\phi^{(\ell_2)}_{N^+,N^-\ell_1\ell_2}$ is the map in Proposition \ref{pr:raising_1}.

By the definition of $\phi_n$, Remark \ref{re:degeneracy_surface} (2),
Proposition \ref{pr:surface_supersingular} (4), and Proposition
\ref{pr:special_superspecial}, we have
\begin{multline*}
\hat\phi_n\(\zeta_{N^+Mq,N^-\ell_1\ell_2;\tilde{k}}^{\dag*}f\)=
(\delta^{(q,1)}_{N^+Mq,N^-\ell_1})^*\zeta_{N^+M,N^-\ell_1}^*\\
\((\gamma^{(\ell_2,1)}_{N^+M\ell_2})_*(\gamma^{(\ell_2,\ell_2)}_{N^+M\ell_2})^*f^\circ-2\ell_2
f^\bullet,(\gamma^{(\ell_2,\ell_2)}_{N^+Mq\ell_2})_*(\gamma^{(\ell_2,1)}_{N^+Mq\ell_2})^*f^\bullet-2\ell_2
f^\circ\);
\end{multline*}
and
\begin{multline*}
\hat\phi_n\(\zeta_{N^+Mq,N^-\ell_1\ell_2;\tilde{k}}^{\ddag*}f\)=
(\delta^{(q,q)}_{N^+Mq,N^-\ell_1})^*\zeta_{N^+M,N^-\ell_1}^*\\
\((\gamma^{(\ell_2,1)}_{N^+M\ell_2})_*(\gamma^{(\ell_2,\ell_2)}_{N^+M\ell_2})^*f^\circ-2\ell_2
f^\bullet,(\gamma^{(\ell_2,\ell_2)}_{N^+Mq\ell_2})_*(\gamma^{(\ell_2,1)}_{N^+Mq\ell_2})^*f^\bullet-2\ell_2
f^\circ\).
\end{multline*}
Now suppose that $f=f_{\nu^\bullet,\nu^\circ}\in\bar\rZ_\pi^n$. Then
\begin{align*}
&\hat\phi_n\(\zeta_{N^+Mq,N^-\ell_1\ell_2;\tilde{k}}^{\dag*}f\)=
(\delta^{(q,1)}_{N^+Mq,N^-\ell_1})^*\zeta_{N^+M,N^-\ell_1}^*
f_{\nu(\ell_2)\nu^\circ-2\ell_2\nu^\bullet,\nu(\ell_2)\nu^\bullet-2\ell_2\nu^\circ};\\
&\hat\phi_n\(\zeta_{N^+Mq,N^-\ell_1\ell_2;\tilde{k}}^{\ddag*}f\)=
(\delta^{(q,q)}_{N^+Mq,N^-\ell_1})^*\zeta_{N^+M,N^-\ell_1}^*
f_{\nu(\ell_2)\nu^\circ-2\ell_2\nu^\bullet,\nu(\ell_2)\nu^\bullet-2\ell_2\nu^\circ},
\end{align*}
where $\nu(\ell_2)=\tr(\Fr_{\ell_2^2};\rV_\pi)$.

Theorem \ref{th:congruence} (3) follows from Propositions \ref{pr:raising_1}
and \ref{pr:raising_2}, Lemma \ref{le:kummer_singular}, and the following
elementary identity
\[(\nu(\ell_2)\nu^\circ-2\ell_2\nu^\bullet)+\epsilon_\sigma(\ell_2)
(\nu(\ell_2)\nu^\bullet-2\ell_2\nu^\circ)
=(\nu(\ell_2)-2\epsilon_\sigma(\ell_2)\ell_2)
(\nu^\circ+\epsilon_\sigma(\ell_2)\nu^\bullet).\]

\subsubsection{Proof of Theorem \ref{th:congruence_bis}}

We apply the previous discussion to the case: $R=\rW(\dF_{\ell}^\ac)$
($k=R\otimes_\dZ\dQ$ and $\tilde{k}=\dF_{\ell}^\ac$), $X=\cX_{N^+M;R}$,
$Y=\cY_{N^+Mq,N^-\ell;R}$, and $W=\b\Delta'_{N^+M,N^-\ell;R}$. Then Theorem
\ref{th:congruence_bis} follows exactly as for Theorem \ref{th:congruence}
(3).

\section{Bounding Selmer groups}
\label{s5}

We prove our main theorems in both versions of automorphic representations
and of elliptic curves. In \Sec \ref{ss:some_galois}, we develop a
Galois-theoretical gadget which is applicable to all finite Galois modules we
will consider later. In \Sec \ref{ss:proof_theorem}, we prove Theorem
\ref{th:selmer}. In \Sec \ref{ss:application_elliptic}, we prove Theorems
\ref{th:main_even} and \ref{th:main_odd} by applying the previous theorem.

\subsection{Some Galois-theoretical arguments}
\label{ss:some_galois}

We first introduce some notation.

\begin{definition}
Let $\rT$ be a $\dZ_p$-module.
\begin{enumerate}
  \item For an element $t\in\rT$, we define the \emph{$p$-divisibility}
      to be
      \[\ord_p(t)=\sup\{n\in\dZ_{\geq 0}\res t\in p^n\rT\}.\]

  \item Suppose that $\rT$ has finite exponent and is equipped with a
      linear action of a group $\rG$. Define the \emph{reducibility
      depth} of $\rT$ to be the least integer $\f\fn_{\rT}\geq 0$ such
      that
      \begin{itemize}
        \item if $\rT'$ is a $\rG$-stable submodule that is not
            contained in $p\rT$, then $\rT'$ contains
            $p^{\fn_{\rT}}\rT$;

        \item for every integer $m\geq 1$, the group
            $\Hom_\rG(\rT/p^m\rT,\rT/p^m\rT)/\dZ_p\cdot\r{id}$ is
            annihilated by $p^{\fn_{\rT}}$.
      \end{itemize}
\end{enumerate}
\end{definition}

\begin{lem}\label{le:reducibility_depth}
Let $\rT$ be a stable lattice in a $p$-adic representation $\rV$ of
$\Gamma_\dQ$ that is absolutely irreducible and almost everywhere unramified.
Then there exists an integer $\fn(\rT)$ depending only on $\rT$, such that
$\bar\rT^n$ has reducibility depth at most $\fn(\rT)$ for all $n\geq 1$.
\end{lem}

\begin{proof}
Since $\rV$ is irreducible, there are only finitely many stable lattices
$\rT_1,\dots,\rT_l$ contained in $\rT$ but not containing $p\rT$. Let
$\fn'\geq 0$ be the smallest integer such that $\rT_i$ contains $p^{\fn'}\rT$
for all $i=1,\dots,l$. For an integer $n\geq1$, suppose that $\rT'$ is a
$\rG$-stable submodule of $\bar\rT^n$ that is not contained in $p\bar\rT^n$.
Take $t\in\rT'$ and lift it to $\tilde{t}\in\rT$, the $\Gamma_\dQ$-stable
lattice generated by $\tilde{t}$ contains $p^{\fn'}\rT$. In particular,
$\rT'$ contains $p^{\fn'}\bar\rT^n$.

Choose a finite set $\Sigma$ of places of $\dQ$ containing $\{\infty,p\}$
such that $\rV$ is unramified outside $\Sigma$. For $n\geq 1$, we have a
sequence
\[\Hom_{\Gamma_\dQ}(\rT,\rT)/p^n\Hom_{\Gamma_\dQ}(\rT,\rT)\hookrightarrow
\Hom_{\Gamma_\dQ}(\bar\rT^n,\bar\rT^n)\to\Ext^1_{\Gamma_{\dQ,\Sigma}}(\rT,\rT)_\tor.\]
Since $\rV$ is absolutely irreducible, the group
$\Hom_{\Gamma_\dQ}(\rT,\rT)/\dZ_p\cdot\r{id}$ is torsion (and finitely
generated). Moreover, the group $\Ext^1_{\Gamma_{\dQ,\Sigma}}(\rT,\rT)_\tor$
is also finitely generated. Therefore, there exists some integer $\fn''\geq
0$ such that $\Hom_{\Gamma_\dQ}(\bar\rT^n,\bar\rT^n)/\dZ_p\cdot\r{id}$ is
annihilated by $p^{\fn''}$ for all $n\geq 1$. To conclude, take
$\fn(\rT)=\max\{\fn',\fn''\}$.
\end{proof}

Let $n\geq 1$ be an integer. Consider a Galois representation
$\rho\colon\Gamma_\dQ\to\GL(\rT)$ where $\rT$ is a free $\dZ/p^n$-module of
finite rank. Denote by $\rG$ the image of $\rho$, and $L/\dQ$ the Galois
extension determined by $\rG$. We suppose that $\rG$ contains a non-trivial
scalar element of order coprime to $p$.

\begin{lem}
We have $\rH^i(\rG,\rT)=0$ for all $i\geq 0$. The restriction of classes
gives an isomorphism:
\[\Res^L_\dQ\colon\rH^1(\dQ,\rT)\xrightarrow\sim\rH^1(L,\rT)^{\rG}
=\Hom_{\rG}(\Gamma_L^\ab,\rT).\]
\end{lem}

\begin{proof}
The proof is the same as \cite{Gro91}*{Proposition 9.1}.
\end{proof}

The above lemma yields a $\dZ/p^n[\rG]$-linear pairing
\[[\;,\;]\colon \rH^1(\dQ,\rT)\times\Gamma_L^\ab\to\rT.\]
For each finitely generated $\dZ/p^n$-submodule $S$ of $\rH^1(\dQ,\rT)$,
denote by $\rG_S$ the subgroup of $\varrho\in\Gamma_L^\ab$ such that
$[s,\varrho]=0$ for all $s\in S$. Let $L_S\subset\dQ^\ac$ be the subfield
fixed by $\rG_S$. Define a sequence $\ff$ by $\ff(1)=1$, $\ff(2)=4$ and
$\ff(r+1)=2(\ff(r)+1)$. The following lemma generalizes
\cite{Gro91}*{Proposition 9.3}.

\begin{lem}\label{le:image}
The induced pairing
\[[\;,\;]\colon S\times\Gal(L_S/L)\to\rT\]
yields an injective map $\Gal(L_S/L)\to\Hom(S,\rT)$ of
$\dZ/p^n[\rG]$-modules. If $S$ is free of rank $r_S$ and $\fn_{\rT}\leq n$ is
the reducibility depth of $\rT$, then the image of $\Gal(L_S/L)$ contains
$p^{\ff(r_S)\fn_{\rT}}\Hom(S,\rT)=\Hom(S,p^{\ff(r_S)\fn_{\rT}}\rT)$.
\end{lem}

\begin{proof}
It is clear that the map $\Gal(L_S/L)\to\Hom(S,\rT)$ is injective and hence
we may identify the Galois group $G\colonequals\Gal(L_S/L)$ as a subgroup of
$\Hom(S,\rT)$. Fix a basis $\{s_1,\dots,s_r\}$ of $S$. We prove by induction
on $r$. The case where $r=1$ is immediate, and hence we assume that $r\geq2$.
Put $S_1=\bigoplus_{i=1}^{r-1}\dZ/p^n\cdot s_i$ and identify the exact
sequence
\[0\to\Hom(\dZ/p^n\cdot s_r,\rT)\to\Hom(S,\rT)\to\Hom(S_1,\rT)\to 0\]
with $0\to\rT\to\rT^{\oplus r}\to\rT^{\oplus r-1}\to 0$. Denote by $G_1$ the
image of $G\subset\rT^{\oplus r}$ under the map $\rT^{\oplus r}\to\rT^{\oplus
r-1}$, and $G_0=G\cap\rT$ its kernel. By induction, the group $G_1$ contains
$p^{\ff(r-1)\fn_{\rT}}\rT^{\oplus r-1}$.

We claim that $G_0$ contains $p^{(\ff(r-1)+2)\fn_{\rT}}\rT$. For this, we
need to show that $G_0$ is not contained in $p^{(\ff(r-1)+1)\fn_{\rT}+1}\rT$.
Otherwise, the group $G_0$ will be annihilated by
$p^{n-1-(\ff(r-1)+1)\fn_{\rT}}$. Consider the exact sequence
\[0\to\Hom_\rG(G_1,\rT)\to\Hom_\rG(G,\rT)\to\Hom_\rG(G_0,\rT),\]
in which $S$ is naturally contained in $\Hom_\rG(G,\rT)$. Then the element
$p^{n-1-(\ff(r-1)+1)\fn_{\rT}}s_r$ maps to $0$ in $\Hom_\rG(G_0,\rT)$ and
hence belongs to $\Hom_\rG(G_1,\rT)$. For $G_1$, we have another exact
sequence
\[0\to\Hom_\rG(G_1/p^{\ff(r-1)\fn_{\rT}}\rT^{\oplus r-1},\rT)\to\Hom_\rG(G_1,\rT)\to\Hom_\rG(p^{\ff(r-1)\fn_{\rT}}\rT^{\oplus r-1},\rT).\]
The last term $\Hom_\rG(p^{\ff(r-1)\fn_{\rT}}\rT^{\oplus r-1},\rT)$, which
equals $\Hom_\rG(p^{\ff(r-1)\fn_{\rT}}\rT^{\oplus
r-1},p^{\ff(r-1)\fn_{\rT}}\rT)$, is isomorphic to
$\Hom_\rG(\rT/p^{\ff(r-1)\fn_{\rT}}\rT,\rT/p^{\ff(r-1)\fn_{\rT}}\rT)^{\oplus
r-1}$, and the natural quotient
\[\Hom_\rG(\rT/p^{\ff(r-1)\fn_{\rT}}\rT,\rT/p^{\ff(r-1)\fn_{\rT}}\rT)/\dZ/p^{n-\ff(r-1)\fn_{\rT}}\cdot\r{id}\]
is annihilated by $p^{\fn_{\rT}}$. Therefore, there is an element $s\in S_1$
such that $p^{n-1-\ff(r-1)\fn_{\rT}}s_r-s$ maps to $0$ in
$\Hom_\rG(p^{\ff(r-1)\fn_{\rT}}\rT^{\oplus r-1},\rT)$. In other words, the
element $p^{n-1-\ff(r-1)\fn_{\rT}}s_r-s$ belongs to
$\Hom_\rG(G_1/p^{\ff(r-1)\fn_{\rT}}\rT^{\oplus r-1},\rT)$, a group
annihilated by $p^{\ff(r-1)\fn_{\rT}}$. This implies that
$p^{n-1}s_r-p^{\ff(r-1)\fn_{\rT}}s=0$ in $\Hom_\rG(G,\rT)$, which contradicts
the fact that $S$ is free.

Now consider the exact sequence $0\to G_0\to G\to G_1\to 0$. For every
$t\in\rT^{\oplus r}$, the image of $p^{\ff(r-1)\fn_{\rT}}t$ in $\rT^{\oplus
r-1}$ is contained in $G_1$ by the induction assumption. Pick up an element
$g\in G$ such that $g-p^{\ff(r-1)\fn_{\rT}}t\in\rT=\ker[\rT^{\oplus
r}\to\rT^{\oplus r-1}]$. By the previous claim,
$p^{(\ff(r-1)+2)\fn_{\rT}}(g-p^{\ff(r-1)\fn_{\rT}}t)$ belongs to $G_0$, which
implies that
$p^{(\ff(r-1)+2)\fn_{\rT}}p^{\ff(r-1)\fn_{\rT}}t=p^{\ff(r)\fn_{\rT}}t$
belongs to $G$.
\end{proof}

\begin{notation}
For each place $w$ of $L$, we denote by $\ell(w)$ the prime underlying $w$.
If $\ell(w)$ is unramified in $L_S$, then the Frobenius substitution in
$\rH\colonequals\Gal(L_S/L)$ of an arbitrary place of $L_S$ above $w$ is
well-defined since $\rH$ is abelian. It will be denoted by $\Psi_w$.
\end{notation}

Take a non-trivial finite free $\dZ/p^n$-submodule $S$ of $\rH^1(\dQ,\rT)$.
Let $\varrho$ be an element of $\rG$ of order $l$ coprime to $p$. Consider
the following diagram of field extension
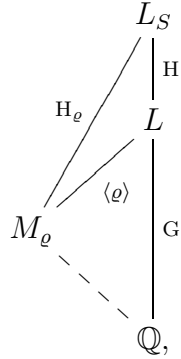
\begin{align*}
\xymatrix{
& L_S  \ar@{-}[d]^-{\rH}\ar@{-}[ddl]_-{\rH_\varrho}  \\
& L  \ar@{-}[dd]^-{\rG}\ar@{-}[dl]^-{\langle\varrho\rangle}  \\
M_\varrho \ar@{--}[dr]  \\
& \dQ,
}
\end{align*}
in which $M_\varrho$ is the subfield fixed by the cyclic subgroup
$\langle\varrho\rangle\subset\rG$ generated by $\varrho$, and
$\rH_\varrho=\Gal(L_S/M)$. In particular, we have an isomorphism
$\rH_\varrho\simeq\rH\rtimes\langle\varrho\rangle\subset\Hom(S,\rT)\rtimes\langle\varrho\rangle$
since $\langle\varrho\rangle$ is of order coprime to $p$.

\begin{definition}
Let $\Sigma$ be a finite set of primes containing those ramified in $L_S$. We
say that a finite place $w$ of $L$ is \emph{$(\Sigma,\varrho)$-admissible} if
its underlying prime $\ell(w)$ does not belong to $\Sigma$, and its Frobenius
substitution in $\Gal(L/\dQ)=\rG$, which is then well-defined, is $\varrho$.
\end{definition}

Let $w$ be a $(\Sigma,\varrho)$-admissible place of $L$. Then
$\loc_{\ell(w)}(s)$ belongs to $\rH^1_\unr(\dQ_{\ell(w)},\rT)$, where the
latter equals $\rH^1(\dF_{\ell(w)},\rT)$ which is contained in
$\Hom(\Gamma_{\ell(w)}/\rI_{\ell(w)},\rT)\simeq\rT$, where the last
isomorphism is induced by taking the image of $\Fr_{\ell(w)}$.

\begin{lem}\label{le:galois_pairing}
Let $w$ be a $(\Sigma,\varrho)$-admissible place of $L$. Then for $s\in S$,
we have
\[[s,\Psi_w]=\loc_{\ell(w)}(s)(\Psi_w)\]
as an equality in $\rT$.
\end{lem}

\begin{proof}
Consider the following composite map
\[S\to\rH^1(\dQ,\rT)\xrightarrow{\Res^L_\dQ}\rH^1(L,\rT)
\xrightarrow{\loc_w}\rH^1(L_w,\rT),\] whose image is contained in
$\rH^1_\unr(L_w,\rT)$. By definition, $[\;,\Psi_w]$ is simply the composition
of the above map with the following one $\rH^1_\unr(L_w,\rT)\to\rT$ obtained
by evaluating at $\Psi_w$.

The lemma follows since we have the following commutative diagram
\[\xymatrix{
\rH^1(\dQ,\rT) \ar[rr]^-{\Res^L_\dQ}\ar[d]_-{\loc_{\ell(w)}}
&& \rH^1(L,\rT) \ar[d]^-{\loc_w} \\
\rH^1(\dQ_{\ell(w)},\rT) \ar[rr]^-{\Res^{L_w}_{\dQ_{\ell(w)}}} &&
\rH^1(L_w,\rT), }\] in which the bottom arrow is injective since
$\Gal(L_w/\dQ_{\ell(w)})\simeq\langle\varrho\rangle$ has order coprime to
$p$.
\end{proof}

\begin{lem}\label{le:admissible}
The subset of $\Psi_w\in\rH\subset\Hom(S,\rT)$ for
$(\Sigma,\varrho)$-admissible places $w$ is exactly
$\rH\cap\Hom(S,\rT^{\langle\varrho\rangle})$.
\end{lem}

\begin{proof}
By Galois theory, an element $\rho$ in $\rH$ is of the form $\Psi_w$ if and
only if $\rho=\widetilde{h}^l$ with $\widetilde{h}\in\rH_\varrho$ which maps
to $\varrho$ under the quotient homomorphism
$\rH_\varrho\to\langle\varrho\rangle$. Suppose that
$\widetilde{h}=(h,\varrho)$ with $h\colon S\to\rT$ in $\rH$. Then
$\widetilde{h}^l=(h',1)$ where
$h'(s)=(\varrho^{l-1}+\cdots+\varrho+1)h(s)\in\rT^{\langle\varrho\rangle}$.
Therefore, the element $\rho$ belongs to
$\Hom(S,\rT^{\langle\varrho\rangle})$. The other inclusion is trivial since
$l$ is invertible in $\dZ/p^n$.
\end{proof}

\subsection{Proof of Theorem \ref{th:selmer}}
\label{ss:proof_theorem}

Put
$\rT_{\sigma\times\breve\pi}=\rT_{\sigma,\pi}\cap\rV_{\sigma\times\breve\pi}$
and
$\rT_{\sigma\otimes\breve\eta}=\rT_{\sigma,\pi}\cap\rV_{\sigma\otimes\breve\eta}$.

\begin{lem}\label{le:selmer_reduction}
For subscripts $?=\sigma,\pi;\sigma\times\breve\pi;\sigma\otimes\breve\eta$,
if $\dim_{\dQ_p}\rH^1_f(\dQ,\rV_?)=r$, then there is a free
$\dZ/p^n$-submodule of $\rH^1_f(\dQ,\bar\rT_?^n)$ (see Definition
\ref{de:selmer_group}) of rank $r$.
\end{lem}

\begin{proof}
Recall that we define $\rH^1_f(\dQ,\rT_?)\subset\rH^1(\dQ,\rT_?)$ to be the
inverse image of $\rH^1_f(\dQ,\rV_?)$ under the natural map
$\rH^1(\dQ,\rT_?)\to\rH^1(\dQ,\rV_?)$, and $\rH^1_f(\dQ,\bar\rT_?^n)$ to be
the image of $\rH^1_f(\dQ,\rT_?)$ the natural map
$\rH^1(\dQ,\rT_?)\to\rH^1(\dQ,\bar\rT_?^n)$.

Suppose that $\dim_{\dQ_p}\rH^1_f(\dQ,\rV_?)=r$. Since $\rH^1_f(\dQ,\rT_?)$
is finitely generated, we may choose $s_1,\dots,s_r\in\rH^1_f(\dQ,\rT_?)$
such that they generate a saturated free $\dZ_p$-submodule. In particular,
the image of $s_1,\dots,s_r$ under the natural map
\[\rH^1_f(\dQ,\rT_?)\to\rH^1_f(\dQ,\rT_?)/p^n\rH^1_f(\dQ,\rT_?)\hookrightarrow
\rH^1(\dQ,\rT_?)/p^n\rH^1(\dQ,\rT_?)\hookrightarrow\rH^1(\dQ,\bar\rT_?^n)\]
span a free $\dZ/p^n$-submodule of rank $r$, which is contained in
$\rH^1_f(\dQ,\bar\rT_?^n)$ by definition.
\end{proof}

The following discussion will be under all of Assumption \ref{as:group_r},
until the end of this section.

\begin{lem}\label{le:localization}
Let $n\geq 1$ be an integer. We have that
\begin{enumerate}
  \item If $\wp(N^-)$ is even, then the class $\bar\Delta_{\sigma,\pi}^n$
      belongs to $\rH^1_f(\dQ,\bar\rT_{\sigma,\pi}^n)$.

  \item If $\wp(N^-)$ is even and $\ell_1,\ell_2$ are two distinct
      strongly $n$-admissible primes, then
      \[\loc_v(\bar\Delta^n_{\sigma,\pi\res\ell_1,\ell_2})\in\rH^1_\unr(\dQ_v,\bar\rT_{\sigma,\pi}^n)\]
      for all primes $v\nmid pqN^-\cD_{N^+M}\ell_1\ell_2$, and
      \[\loc_p(\bar\Delta^n_{\sigma,\pi\res\ell_1,\ell_2})\in\rH^1_f(\dQ_p,\bar\rT_{\sigma,\pi}^n).\]

  \item If $\wp(N^-)$ is odd and $\ell$ is a strongly $n$-admissible
      prime, then
      \[\loc_v(\bar\Delta^n_{\sigma,\pi\res\ell})\in\rH^1_\unr(\dQ_v,\bar\rT_{\sigma,\pi}^n)\]
      for all primes $v\nmid pqN^-\cD_{N^+M}\ell$, and
      \[\loc_p(\bar\Delta^n_{\sigma,\pi\res\ell})\in\rH^1_f(\dQ_p,\bar\rT_{\sigma,\pi}^n).\]
\end{enumerate}
\end{lem}

\begin{proof}
Part (1) follows by Proposition \ref{pr:belong_selmer}, Definition
\ref{de:selmer_group}, and the construction of $\bar\Delta_{\sigma,\pi}^n$.
For (2), since the $p$-adic representation
\[\rH^1_{\et}(Y_{N^+M,N^-\ell_1\ell_2;\dQ_v^\ac},\dQ_p(1))\otimes\rH^2_\cusp(X_{N^+M;\dQ_v^\ac},\dQ_p(1))\]
is unramified at all primes $v\nmid pqN^-\cD_{N^+M}\ell_1\ell_2$, we have
$\loc_v(\bar\Delta^n_{\sigma,\pi\res\ell_1,\ell_2})\in\rH^1_\unr(\dQ_v,\bar{\rT}_{\sigma,\pi}^n)$
by Lemma \ref{le:unramified_torsion} and the construction of
$\bar\Delta_{\sigma,\pi\res\ell_1,\ell_2}^n$. For the localization at $p$,
the argument is the same as the last part of the proof of Proposition
\ref{pr:belong_selmer}. Part (3) is the same as (2).
\end{proof}

\begin{notation}\label{no:divisibility_density}~
\begin{enumerate}
  \item If $\wp(N^-)$ is even, denote by $n_{\r{div}}$ the
      $p$-divisibility $\ord_p(\bar\Delta_{\sigma,\pi}^\infty)$ of
      $\bar\Delta_{\sigma,\pi}^\infty$. If $\wp(N^-)$ is odd, denote by
      $n_{\r{div}}$ the $p$-divisibility of \eqref{eq:ichino} as a
      non-zero number in $\dZ_p$. Then $n_{\r{div}}\geq 0$ is an integer.

  \item By Remark \ref{re:strongly_admissible} (1) and (R4), the density
      of strongly $(n,-)$-admissible primes among all $n$-admissible
      primes is strictly positive and independent of $n$. We denote it by
      $\fd$. Let $n_{\r{den}}\geq 0$ be the smallest integer such that
      \[p^{n_{\r{den}}+1}>\fd^{-1}.\]
\end{enumerate}
\end{notation}

For our convenience, we record the following corollary of Theorems
\ref{th:congruence} and \ref{th:congruence_bis}.

\begin{proposition}\label{pr:divisibility}
Let $n\geq1$ be an integer.
\begin{enumerate}
  \item Suppose that $\wp(N^-)$ is even. Let $\ell_1,\ell_2$ be two
      distinct strongly $n$-admissible primes. Then
      \[\ord_p(\loc_{\ell_1}(\bar\Delta^n_{\sigma,\pi}))=\ord_p(\partial_{\ell_2}\loc_{\ell_2}(\bar\Delta^n_{\sigma,\pi\res\ell_1,\ell_2})),\]
      where
      \begin{align*}
      \loc_{\ell_1}(\bar\Delta^n_{\sigma,\pi})&\in\rH^1_\unr(\dQ_{\ell_1},\bar\rT_{\sigma,\pi}^n)\simeq\dZ/p^n,\\
      \partial_{\ell_2}\loc_{\ell_2}(\bar\Delta^n_{\sigma,\pi\res\ell_1,\ell_2})&\in\rH^1_\sing(\dQ_{\ell_1},\bar\rT_{\sigma,\pi}^n)\simeq\dZ/p^n
      \end{align*}
      by Proposition \ref{pr:cohomology_mixed}.

  \item Suppose that $\wp(N^-)$ is odd. Let $\ell$ be a strongly
      $n$-admissible prime. Then
      \[\ord_p(\partial_{\ell}\loc_{\ell}(\bar\Delta^n_{\sigma,\pi\res\ell}))=n_{\r{div}},\]
      where $n_{\r{div}}$ is the $p$-divisibility of \eqref{eq:ichino} as
      in Notation \ref{no:divisibility_density}.
\end{enumerate}
\end{proposition}

\subsubsection{Proof under Assumption \ref{as:group_r} (R6a)}

In this case, the representation $(\rho_{\sigma,\pi},\rV_{\sigma,\pi})$ is
absolutely irreducible. Denote by $n_{\r{red}}=\fn(\rT_{\sigma,\pi})$ the
integer in Lemma \ref{le:reducibility_depth}.

First, consider the case where $\wp(N^-)$ is even. We prove by contradiction,
and hence assume that $\dim_{\dQ_p}\rH^1_f(\dQ,\rV_{\sigma,\pi})\geq 2$. Let
$n$ be a sufficiently large integer which will be determined later. By Lemma
\ref{le:selmer_reduction}, we may take a free $\dZ/p^n$-submodule $S$ of
$\rH^1_f(\dQ,\bar\rT_{\sigma,\pi}^n)$ of rank $2$. Moreover, we may assume
that $S$ has a basis $\{s_1,s_2\}$ such that
$\bar\Delta_{\sigma,\pi}^n=p^{n_{\r{div}}}s_1$.

By (R4), there exists a strongly $n$-admissible prime. We denote the image of
its Frobenius substitution under $\bar\rho_{\sigma,\pi}^n$ by
$\varrho\in\rG$. By raising a sufficiently large $p$-power, we may assume
that $\varrho$ has order $l$ coprime to $p$. We would like to apply the
discussion in \Sec \ref{ss:some_galois} to $\rT=\bar\rT_{\sigma,\pi}^n$ and
$\varrho$ as above. Assumption \ref{as:group_r} (R6a) implies that $\rG$
contains a non-trivial scalar element of order coprime to $p$. Let $\Sigma$
be the set of primes dividing $2pqN^-\cD_{N^+M}$ and those ramified in $L_S$.

\begin{lem}
If $w$ is a $(\Sigma,\varrho)$-admissible place of $L$, then its underlying
prime $\ell(w)$ is strongly $n$-admissible.
\end{lem}

\begin{proof}
Since the kernel of the homomorphism
$\GL(\bar\rT_\sigma^n)\times\rO(\pres{\sharp}{\bar\rT}_\pi^n)\xrightarrow{\otimes}\GL(\bar\rT_{\sigma,\pi}^n)$
consists of $(\mu^{-1},\mu)$ for some $\mu\in(\dZ/p^n)^\times$ of order $4$,
the lemma follows by Definition \ref{de:group_s} (S2) and Remark
\ref{re:strongly_admissible}.
\end{proof}

By Lemmas \ref{le:image} and \ref{le:admissible}, we may choose two
$(\Sigma,\varrho)$-admissible places $w_1,w_2$ of $L$ such that
\[\Psi_{w_1}(s_2)=0,\quad\Psi_{w_1}(s_1)=t_1,\quad\Psi_{w_2}(s_2)=t_2\]
with $t_1,t_2\in(\bar\rT_{\sigma,\pi}^n)^{\langle\varrho\rangle}$ such that
$\ord_p(t_1)=\ord_p(t_2)=4n_{\r{red}}$, and $\ell_1\colonequals
\ell(w_1),\ell_2\colonequals\ell(w_2)$ are distinct. By Lemma
\ref{le:galois_pairing}, we have
\[\loc_{\ell_1}(s_2)=0,\quad\loc_{\ell_1}(\bar\Delta_{\sigma,\pi}^n)(\Fr_{\ell_1})=p^{n_{\r{div}}}t_1/l.\]
By Proposition \ref{pr:divisibility} (1), the $p$-divisibility of
$\partial_{\ell_2}\loc_{\ell_2}(\bar\Delta_{\sigma,\pi\res\ell_1,\ell_2}^n)$
in $\rH^1_\sing(\dQ_{\ell_2},\bar\rT_{\sigma,\pi}^n)\simeq\dZ/p^n$ is
$n_{\r{div}}+4n_{\r{red}}$.

By Lemmas \ref{le:localization}, \ref{le:unramified_torsion},
\ref{le:trivial_unramified} (2) and \ref{le:trivial_p}, we have
\begin{align}\label{eq:divisibility1}
\langle s_2,\bar\Delta_{\sigma,\pi\res\ell_1,\ell_2}^n\rangle_v=0,
\end{align}
for all primes $v\nmid qN^-\cD_{N^+M}\ell_1\ell_2$. We also have
\begin{align}\label{eq:divisibility2}
\langle s_2,\bar\Delta_{\sigma,\pi\res\ell_1,\ell_2}^n\rangle_{\ell_1}=0,
\end{align}
since $\loc_{\ell_1}(s_2)=0$. For every prime $v$ divides $qN^-\cD_{N^+M}$,
we have
\begin{align}\label{eq:divisibility3}
p^{n_{\r{bad}}}\langle s_2,\bar\Delta_{\sigma,\pi\res\ell_1,\ell_2}^n\rangle_v=0,
\end{align}
where $n_{\r{bad}}$ is introduced in Notation \ref{no:bad}. By our choice of
$w_2$ and the fact that $\langle\;,\;\rangle_{\ell_2}$ induces a perfect
pairing between $\rH^1_\unr(\dQ_{\ell_2},\bar\rT_{\sigma,\pi}^n)$ and
$\rH^1_\sing(\dQ_{\ell_2},\bar\rT_{\sigma,\pi}^n)$, we have $\ord_p(\langle
s_2,\bar\Delta_{\sigma,\pi\res\ell_1,\ell_2}^n\rangle_{\ell_2})=n_{\r{div}}+8n_{\r{red}}$.
If $n>n_{\r{div}}+n_{\r{bad}}+8n_{\r{red}}$, then together with
\eqref{eq:divisibility1}, \eqref{eq:divisibility2}, \eqref{eq:divisibility3},
we have $\sum_v\langle
s_2,\bar\Delta_{\sigma,\pi\res\ell_1,\ell_2}^n\rangle_v\neq 0$, which
contradicts Lemma \ref{le:trivial_global}. Therefore, Theorem \ref{th:selmer}
(2) follows under Assumption \ref{as:group_r} (R6a).

Next, consider the case where $\wp(N^-)$ is odd. We prove by contradiction,
and hence assume that $\dim_{\dQ_p}\rH^1_f(\dQ,\rV_{\sigma,\pi})\geq 1$. Let
$n$ be a sufficiently large integer which will be determined later. By Lemma
\ref{le:selmer_reduction}, we may take a free $\dZ/p^n$-submodule $S$ of
$\rH^1_f(\dQ,\bar\rT_{\sigma,\pi}^n)$ of rank $1$ with a basis $\{s\}$.

By (R4), there exists a strongly $n$-admissible prime. We denote its image
under $\bar\rho_{\sigma,\pi}^n$ by $\varrho\in\rG$. By raising a sufficiently
large $p$-power, we may assume that $\varrho$ has order $l$ coprime to $p$.
We would like to apply the discussion in \Sec \ref{ss:some_galois} to
$\rT=\bar\rT_{\sigma,\pi}^n$ and $\varrho$ as above. Assumption
\ref{as:group_r} (R6a) implies that $\rG$ contains a non-trivial scalar
element of order coprime to $p$. Let $\Sigma$ be the set of primes dividing
$2pqN^-\cD_{N^+M}$ and those ramified in $L_S$. If $w$ is a
$(\Sigma,\varrho)$-admissible place of $L$, then its underlying prime
$\ell(w)$ is strongly $n$-admissible.

By Lemma \ref{le:admissible}, we may choose a $(\Sigma,\varrho)$-admissible
place $w$ of $L$ such that $\ord_p(\Phi_w(s))=n_{\r{red}}$. By Proposition
\ref{pr:divisibility} (2), the $p$-divisibility of
$\partial_{\ell}\loc_{\ell}(\bar\Delta_{\sigma,\pi\res\ell}^n)$ in
$\rH^1_\sing(\dQ_{\ell},\bar\rT_{\sigma,\pi}^n)\simeq\dZ/p^n$ is
$n_{\r{div}}$. We will obtain a contradiction when
$n>n_{\r{div}}+n_{\r{red}}+n_{\r{bad}}$. Therefore, Theorem \ref{th:selmer}
(1) follows under Assumption \ref{as:group_r} (R6a).

\subsubsection{Proof under Assumption \ref{as:group_r} (R6b)}

In this case, we have
$\rT_{\sigma,\pi}=\rT_{\sigma\times\breve\pi}\oplus\rT_{\sigma\otimes\breve\eta}$.
Denote by $\bar\Delta_{\sigma\times\breve\pi}^n$ (resp.\
$\bar\Delta_{\sigma\otimes\breve\eta}^n$) the projection of
$\bar\Delta_{\sigma,\pi}^n$ to $\rH^1(\dQ,\bar\rT_{\sigma\times\breve\pi}^n)$
(resp.\ $\rH^1(\dQ,\bar\rT_{\sigma\otimes\breve\eta}^n)$) for $1\leq
n\leq\infty$. Since the representation
$(\rho_{\sigma\times\breve\pi},\rV_{\sigma\times\breve\pi})$ is absolutely
irreducible, we have the integer $n_{\r{red}}\colonequals
\fn(\rT_{\sigma\times\breve\pi})$ from Lemma \ref{le:reducibility_depth}.
Moreover, we may take $\fn(\rT_{\sigma\otimes\breve\eta})=0$ by Assumption
\ref{as:group_r} (R2).

In what follows, we only prove for Theorem \ref{th:selmer} (2), since the
proof for Theorem \ref{th:selmer} (1) is very similar and easier, as we can
see from the proof under Assumption \ref{as:group_r} (R6a). The proof for
Theorem \ref{th:selmer} (1) will be left to the reader.

We divide the proof of Theorem \ref{th:selmer} (2) under Assumption
\ref{as:group_r} (R6b) into the following four lemmas.

\begin{lem}\label{le:r6b1}
If $\bar\Delta_{\sigma\times\breve\pi}^\infty$ is non-torsion, then
$\dim_{\dQ_p}\rH^1(\dQ,\rV_{\sigma\times\breve\pi})=1$.
\end{lem}

For this statement, we do not need to consider $\fd$ and $n_{\r{den}}$.

\begin{proof}
We prove by contradiction, and hence assume that
$\dim_{\dQ_p}\rH^1_f(\dQ,\rV_{\sigma\times\breve\pi})\geq 2$. Let $n$ be a
sufficiently large integer which will be determined later. By Lemma
\ref{le:selmer_reduction}, we may take a free $\dZ/p^n$-submodule $S$ of
$\rH^1_f(\dQ,\bar\rT_{\sigma\times\breve\pi}^n)$ of rank $2$. Moreover, we
may assume that $S$ has a basis $\{s_1,s_2\}$ such that
$\bar\Delta_{\sigma\times\breve\pi}^n=p^{n_{\r{div}}}s_1$.

By (R4), there exists a strongly $(n,+)$-admissible prime. We denote its
image under $\bar\rho_{\sigma\times\breve\pi}^n$ by $\varrho\in\rG$. By
raising a sufficiently large $p$-power, we may assume that $\varrho$ has
order coprime to $p$. We apply the discussion in \Sec \ref{ss:some_galois} to
$\rT=\bar\rT_{\sigma\times\breve\pi}^n$ and $\varrho$ as above. Assumption
\ref{as:group_r} (R6b) implies that $\rG$ contains a non-trivial scalar
element of order coprime to $p$. Let $\Sigma$ be the set of primes dividing
$2pqN^-\cD_{N^+M}$ and those ramified in $L_S$. Note that
$\rT^{\langle\varrho\rangle}$ is free of rank $1$. Let $w$ be a
$(\Sigma,\varrho)$-admissible place of $L$. We claim that its underlying
prime $\ell(w)$ is strongly $(n,+)$-admissible. This follows from the
observation that the kernel of the homomorphism
$\GL(\bar\rT_\sigma^n)\times\GL(\bar\rT_{\breve\pi}^n(-1))\xrightarrow{\otimes}\GL(\bar\rT_{\sigma\times\breve\pi}^n)$
consists of $(\mu^{-1},\mu)$ for some $\mu\in(\dZ/p^n)^\times$ of order $3$.
However by Remark \ref{re:strongly_admissible} (2), the element
$\mu\bar\rho_{\breve\pi}^n(-1)(\Fr_{\ell(w)})$ belongs to the image of
$\bar\rho_{\breve\pi}^n(-1)$ (here, $\rho_{\breve\pi}$ is the underlying
representation of $\rT_{\breve\pi}$) only when $\mu=1$.

The rest follows in the same way as the case under (R6a).
\end{proof}

\begin{lem}\label{le:r6b2}
If $\bar\Delta_{\sigma\times\breve\pi}^\infty$ is non-torsion, then
$\dim_{\dQ_p}\rH^1(\dQ,\rV_{\sigma\otimes\breve\eta})=0$.
\end{lem}

\begin{proof}
We prove by contradiction, and hence assume that
$\dim_{\dQ_p}\rH^1_f(\dQ,\rV_{\sigma\otimes\breve\eta})\geq 1$. Let $n$ be a
sufficiently large integer which will be determined later. By Lemma
\ref{le:selmer_reduction}, we may take a free $\dZ/p^n$-submodule $S_1$
(resp.\ $S_2$) of $\rH^1_f(\dQ,\bar\rT_{\sigma\times\breve\pi}^n)$ (resp.\
$\rH^1_f(\dQ,\bar\rT_{\sigma\otimes\breve\eta}^n)$) of rank $1$ with basis
$\{s_1\}$ (resp.\ $\{s_2\}$). Moreover, we may assume that
$\bar\Delta_{\sigma\times\breve\pi}^n=p^{n_{\r{div}}}s_1$. A similar argument
from the above lemma provides us with a strongly $(n,+)$-admissible prime
$\ell_1$ such that
$\ord_p(\loc_{\ell_1}(\bar\Delta_{\sigma,\pi}^n)(\Fr_{\ell_1}))=n_{\r{div}}+n_{\r{red}}$.

We apply the discussion in \Sec \ref{ss:some_galois} to
$\rT=\bar\rT_{\sigma\otimes\breve\eta}^n$. Note that (R3) implies that $\rG$
contains a non-trivial scalar element of order coprime to $p$. Let $\Sigma$
be the set of primes dividing $2pqN^-\cD_{N^+M}$ and those ramified in
$L_{S_2}$. If we put
$\varrho_1=\bar\rho_{\sigma\otimes\breve\eta}^n(\Fr_{\ell_1})\in\rG$, then it
has order coprime to $p$ and
$(\bar\rT_{\sigma\otimes\breve\eta}^n)^{\langle\varrho_1\rangle}=0$. This
implies that $\loc_{\ell_1}(s_2)=0$.

Consider the composite homomorphism
\[\chi\colon\Gamma_\dQ\xrightarrow{\chi_1}\Gal(L_{S_2}/\dQ)\times\Gal(\breve{F}/\dQ)\to\rG\times\Gal(\breve{F}/\dQ),\]
in which the second factor of $\chi_1$ is the natural projection. Since
$\breve{F}$ is unramified outside $N$, $\chi$ is surjective. Since the kernel
of the second homomorphism is a $p$-group (and $p$ is odd) and the image of
$\chi_1$ has the index at most $2$, the homomorphism $\chi_1$ is surjective
as well.

Let $\ell$ be an $n$-admissible prime that is unramified in $L_{S_2}$ such
that $\epsilon_\sigma(\ell)=\breve\eta(\dF_\ell)$. Take an arbitrary place
$w$ of $L$ above $\ell$ and denote its Frobenius substitution in $\rG$ by
$\varrho$. The $p$-divisibility $\ord_p(\Psi_w(s_2))$ of
$\Psi_w(s_2)\in(\bar\rT_{\sigma\otimes\breve\eta}^n)^{\langle\varrho\rangle}\simeq\dZ/p^n$
is independent of the choice of $w$. By the Chebotarev Density Theorem and
the fact that $\chi_1$ is surjective, the density of those $\ell$ among all
$n$-admissible primes satisfying $\epsilon_\sigma(\ell)=\breve\eta(\dF_\ell)$
and such that $\ord_p(\Psi_w(s_2))\leq n'$ is $1-1/{p^{n'+1}}$ for $0\leq
n'<n$. Then by Notation \ref{no:divisibility_density} (2), there is a
strongly $(n,-)$-admissible prime $\ell_2$ such that
$\ord_p(\loc_{\ell_2}(s_2))\leq n_{\r{den}}$ as we see in Lemma
\ref{le:galois_pairing}.

If we proceed as for the case under (6Ra), then we will obtain a
contradiction when $n>n_{\r{div}}+n_{\r{red}}+n_{\r{den}}+n_{\r{bad}}$.
\end{proof}

\begin{lem}\label{le:r6b3}
If $\bar\Delta_{\sigma\otimes\breve\eta}^n$ is non-torsion, then
$\dim_{\dQ_p}\rH^1(\dQ,\rV_{\sigma\times\breve\pi})=0$.
\end{lem}

\begin{proof}
The proof is very similar to Lemma \ref{le:r6b2}. We only sketch the
procedure and omit the justification since they are close. Let $n$ be a
sufficiently large integer which will be determined later.

We prove by contradiction, and hence assume that
$\dim_{\dQ_p}\rH^1(\dQ,\rV_{\sigma\times\breve\pi})\geq 1$. By Lemma
\ref{le:selmer_reduction}, we may take a free $\dZ/p^n$-submodule $S_1$
(resp.\ $S_2$) of $\rH^1_f(\dQ,\bar\rT_{\sigma\otimes\breve\eta}^n)$ (resp.\
$\rH^1_f(\dQ,\bar\rT_{\sigma\times\breve\pi}^n)$) of rank $1$ with basis
$\{s_1\}$ (resp.\ $\{s_2\}$). Moreover, we may assume that
$\bar\Delta_{\sigma\otimes\breve\eta}^n=p^{n_{\r{div}}}s_1$.

We start from $\bar\rT_{\sigma\otimes\breve\eta}^n$ and choose a strongly
$(n,-)$-admissible prime $\ell_1$ such that
$\bar\rho_{\sigma\times\breve\pi}^n(\Fr_{\ell_1})$ has order coprime to $p$,
and $\ord_p(\loc_{\ell_1}(\bar\Delta_{\sigma\otimes\breve\eta}^n))\leq
n_{\r{den}}+n_{\r{div}}$. We also have $\loc_{\ell_1}(s_2)=0$. Now choose a
strongly $(n,+)$-admissible prime $\ell_2$ such that
$\ord_p(\loc_{\ell_2}(s_2))=n_{\r{red}}$. If we proceed as for the case under
(6Ra), then we will obtain a contradiction when
$n>n_{\r{div}}+n_{\r{den}}+n_{\r{red}}+n_{\r{bad}}$.
\end{proof}

\begin{lem}\label{le:r6b4}
If $\bar\Delta_{\sigma\otimes\breve\eta}^n$ is non-torsion, then
$\dim_{\dQ_p}\rH^1(\dQ,\rV_{\sigma\otimes\breve\eta})=1$.
\end{lem}

\begin{proof}
We prove by contradiction, and hence assume that
$\dim_{\dQ_p}\rH^1(\dQ,\rV_{\sigma\otimes\breve\eta})\geq 2$. Let $n$ be a
sufficiently large integer which will be determined later. By Lemma
\ref{le:selmer_reduction}, we may take a free $\dZ/p^n$-submodule $S$ of
$\rH^1_f(\dQ,\bar\rT_{\sigma\otimes\breve\eta}^n)$ of rank $2$. Moreover, we
may assume that $S$ has a basis $\{s_1,s_2\}$ such that
$\bar\Delta_{\sigma\otimes\breve\eta}^n=p^{n_{\r{div}}}s_1$. As in Lemma
\ref{le:r6b3}, we may choose a strongly $(n,-)$-admissible prime $\ell_1$
such that $\ord_p(\loc_{\ell_1}(s_1))\leq n_{\r{den}}$. Taking a suitable
linear combination of $s_1,s_2$, we obtain an element $s'_2\in S$ with
$\ord_p(s'_2)=0$ and $\loc_{\ell_1}(s'_2)=0$. Choose another strongly
$(n,-)$-admissible prime $\ell_2$ such that $\ord_p(\loc_{\ell_2}(s'_2))\leq
n_{\r{den}}$. The same argument for the case under (6Ra) will produce a
contradiction when $n>n_{\r{div}}+2n_{\r{den}}+n_{\r{bad}}$.
\end{proof}

\subsection{Application to elliptic curves}
\label{ss:application_elliptic}

Now we go back to deduce both Theorem \ref{th:main_even} and Theorem
\ref{th:main_odd} from Theorem \ref{th:selmer}. Let $\sigma$ (resp.\ $\pi$)
be the cuspidal automorphic representation of $B_{N^-}^\times(\dA)$ (resp.\
$\Res_{F/\dQ}\GL_2(\dA)$) associated to $E$ (resp.\ $A$) as in \Sec
\ref{ss:main_results}. Then both $\sigma$ and $\pi$ have the trivial central
character, and are non-dihedral by Assumption \ref{as:group_e} (E1).

For the elliptic curve $A$, we have three different types.
\begin{description}
  \item[Type AI] $A$ and $A^\theta$ are not geometrically isogenous;

  \item[Type AII] $A$ and $A^\theta$ are geometrically isogenous, but
      $A^\theta$ is not isogenous to a (possibly trivial) quadratic twist
      of $A$;

  \item[Type B] $A^\theta$ is isogenous to a (possibly trivial) quadratic
      twist of $A$.
\end{description}

Note that $\pi$ is Asai-decomposable (Definition \ref{de:asai_decomposable})
if and only if $A$ is of Type B. In this case, we have an extra number field
$\breve{F}$.

\begin{definition}[Group \textbf{P}]\label{de:group_p}
We say that a prime $p$ is \emph{good} (with respect to the pair $(E,A)$) if
\begin{description}
  \item[(P1)] $p\geq 11$ and $p\neq 13$;

  \item[(P2)] $p$ is coprime to $N\cdot\Nm_{F/\dQ}\fM\cdot\disc(F)$;

  \item[(P3)] the $\Gamma_\dQ$-representation $E[p](\dQ^\ac)$ is
      surjective;

  \item[(P4)] the $\Gamma_\dQ$-representation $E[p](\dQ^\ac)$ is ramified
      at all primes $\ell\mid N^-$ with $p\mid \ell^2-1$;

  \item[(P5)] the induced projective $\Gamma_\dQ$-representation of
      $E[p](\dQ^\ac)$ is ramified somewhere outside $p$;

  \item[(P6)] we have
  \begin{itemize}
    \item if $A$ is of Type AI, then the $\Gamma_F$-representation
        $A[p](\dQ^\ac)\oplus A^\theta[p](\dQ^\ac)$ has the largest
        possible image, that is,
        $\rG(\SL(A[p](\dQ^\ac))\times\SL(A^\theta[p](\dQ^\ac)))$;

    \item if $A$ is of Type AII, then there exists an intermediate
        number field $F\subset K\subset\dQ^\ac$ over which $A$ and
        $A^\theta$ are isogenous, such that the
        $\Gamma_K$-representation $A[p](\dQ^\ac)$ is surjective;

    \item if $A$ is of Type B, then the
        $\Gamma_{\breve{F}F}$-representation $A[p](\dQ^\ac)$ is
        surjective.
  \end{itemize}
\end{description}
\end{definition}

\begin{lem}
For a given pair of elliptic curves $(E,A)$ satisfying Assumption
\ref{as:group_e}, all but finitely many primes $p$ are good.
\end{lem}

\begin{proof}
It is clear that for all but finitely many primes $p$, (P1, P2, P4) are
satisfied.

By \cite{Ser72}*{Th\'{e}or\`{e}me 2}, (P3) is also satisfied for all but
finitely many primes $p$ under Assumption \ref{as:group_e} (E1, E2).

(P6) is satisfied for all but finitely many primes $p$ by
\cite{Ser72}*{Th\'{e}or\`{e}me 6} (resp.\ \cite{Ser72}*{Th\'{e}or\`{e}me 2})
if $A$ is of Type AI (resp.\ Type AI and Type B).

We prove (P5) by contradiction. Assume that (P5) does not hold for almost all
primes $p$. Then the elliptic curve $E$ will have potentially good reduction
at every prime $v\mid N$. For every prime $v\mid N$, let $\rJ_v$ be the
smallest quotient of $\rI_v$ such that the representation $\rV_q(E)$
restricted to $\rI_v$ factorizes through $\rJ_v$ for every prime $q\nmid N$.
Then $\#\rJ_v=2$. Denote by $\rI'_v$ the kernel of the projection
$\rI_v\to\rJ_v$. We may find a quadratic extension $F'/\dQ$ ramified at every
prime $v\mid N$ with the prescribed inertia subgroup $\rI'_v$. It follows
that $E$ acquires everywhere good reduction after base change to $F'$, which
violates \cite{Kid02}*{Theorem 1.1}.
\end{proof}

Let $p$ be a good prime. We first check several conditions in Assumption
\ref{as:group_r}. By enlarging the ideal $\fM$ to a positive integer $M$, but
remaining coprime to $pN$, (R1) is satisfied. (R2) is equivalent to (P3).
(R3) is equivalent to (P4). (R5) is weaker than (P1). The rest of this
section will take care of (R4) and (R6).

Recall that we have the following Galois representations
\[\rho_\sigma\colon\Gamma_\dQ\to\GL(\rT_\sigma),\quad
\rho_\pi\colon\Gamma_F\to\GL(\rT_\pi),\quad
\pres{\sharp}\rho_\pi\colon\Gamma_\dQ\to\GL(\pres{\sharp}\rT_\pi).\] For each
prime $\ell\nmid 2pM\disc(F)$, the element $\pres{\sharp}\rho_\pi(\Fr_\ell)$
has determinant $1$ (resp.\ $-1$) if and only if $\ell$ is split (resp.\
inert) in $F$. Denote by $L_1$ (resp.\ $L_2$) the splitting field of the
residue representation $\bar\rho_\sigma$ (resp.
$\pres{\sharp}{\bar\rho}_\pi$). Put $L_0=L_1\cap L_2$, in which the three
fields are all number fields contained in $\dQ^\ac$. Denote by
$\GL(\bar\rT_\sigma)^+$ the subgroup of $\GL(\bar\rT_\sigma)$ of elements
whose determinant belongs to $(\dF_p^\times)^2$.

\begin{lem}\label{le:large_image}
The Galois group $\Gal(L_1/L_0)$ contains $\GL(\bar\rT_\sigma)^+$.
\end{lem}

\begin{proof}
We first show that $\Gal(L_1/L_0)$ contains $\SL(\bar\rT_\sigma)$. Otherwise,
the subgroup $\Gal(L_1/L_0)$ is contained in the center of
$\GL(\bar\rT_\sigma)$ since it is normal. Because $\pres{\sharp}\rho_\pi$ is
unramified outside $pM\disc(F)$, the fields $L_2$ and hence $L_0$ are
unramified outside $pM\disc(F)$. In particular, for $v\mid N$, the element
$\rho_\sigma(\rI_v)$ is contained in the center of $\GL(\bar\rT_\sigma)$.
This is contradictory to (P5).

By (P6), for every $a\in\dF_p^\times$, the element $(a\b1_2,a\b1_2)$ is
contained in the image of $\bar\rho_\pi\times\bar\rho_\pi^\theta$. Suppose
that the previous element is the image of some Frobenius substitution $\Fr$
of $\dQ^\ac$, whose underlying prime is coprime to $pNM\disc(F)$ and split in
$F$. Then $\pres{\sharp}{\bar\rho}_\pi(\Fr)=\b1_4$ and
$\det\bar\rho_\sigma(\Fr)=a^2$. This means that $\Gal(L_1/L_0)$ contains an
element of determinant $a^2$. Since $a$ is arbitrary, the lemma follows.
\end{proof}

\begin{remark}
Lemma \ref{le:large_image} implies the following
\begin{itemize}
  \item If $A$ is of Type AI and $p$ is good, then $\rT_{\sigma,\pi}$ is
      residually absolutely irreducible.

  \item If $A$ is of Type B and $p$ is good, then
      $\rT_{\sigma\times\breve\pi}$ is residually absolutely irreducible.
\end{itemize}
Now I claim that if $A$ is of Type AII, then $\rT_{\sigma,\pi}$ is residually
absolutely irreducible for all but finitely many (good) primes $p$. To
indicate the relevance on $p$, we write $\rT_{\pi,p}$ instead of $\rT_\pi$.
Assuming the contrary, for all but finitely many primes $p$, the
representation $\bar\rT_{\pi,p}$ is isomorphic to
$\bar\rT_{\pi,p}^\theta\otimes\breve\eta_p$ for a quadratic character
$\breve\eta_p$, which a priori depends on $p$. Let $w$ be a place of $F$
whose underlying prime is split in $F$. Then for all but finitely many primes
$p$,
\[\r{Sat}(\pi_w)^2\equiv\r{Sat}(\pi_w^\theta)^2 \mod p,\]
where $\r{Sat}$ indicates the Satake parameter. Therefore, the character
$\breve\eta\colonequals\breve\eta_p$ is independent of $p$, and
$\pi^\theta\simeq\pi\times\breve\eta$, which contradicts that $A$ is of Type
AII. In particular, for all types of $A$, the integer $\fn(\rT)$ in Lemma
\ref{le:reducibility_depth} may be taken as $0$ for $\rT=\rT_{\sigma,\pi}$ or
$\rT_{\sigma\times\breve\pi}$ for all but finitely many primes $p$.
\end{remark}

Put $\dF_p^{\r{bad}}=\{\mu\in\dF_p^\times\res \mu^r=1\text{ for some
}r\in\{1,2,3,4,6\}\}$.

\begin{lem}\label{le:good_image}
The subgroup $\Gal(L_2/\dQ)$ of $\GL(\pres{\sharp}{\bar\rT}_\pi)$ contains a
(semisimple) element whose geometric eigenvalues are $\{1,-1,\mu,\mu^{-1}\}$
for some $\mu\not\in\dF_p^{\r{bad}}$.
\end{lem}

\begin{proof}
For this statement, we may replace $\pres{\sharp}{\bar\rT}_\pi$ by the tensor
induction (Asai representation) of $\bar\rT_\pi$ since they have the same
semisimplification. Choose a lifting $\vartheta\in\Gamma_\dQ$ of
$\theta\in\Gal(F/\dQ)$. We will use $\{1,\vartheta\}$ as the coset
representatives to define the tensor induction, which is realized on the
$\dF_p$-vector space $\bar\rT_\pi^{\otimes 2}$. Suppose that
\[\bar\rho_\pi(\vartheta^2)=\left(
                              \begin{array}{cc}
                                a & b \\
                                c & d \\
                              \end{array}
                            \right)
\] under a basis $\{e_1,e_2\}$ of $\bar\rT_\pi$. Then
\[\pres{\sharp}{\bar\rho_\pi}(\vartheta)=\varepsilon(\vartheta)^{-1}\left(
                                           \begin{array}{cccc}
                                             a & b &  &  \\
                                              &  & a & b \\
                                             c & d &  &  \\
                                              &  & c & d \\
                                           \end{array}
                                         \right)
\] under the basis $\{e_1\otimes e_1,e_1\otimes e_2,e_2\otimes e_1,e_2\otimes e_2\}$,
where $\varepsilon(\vartheta)$ is an element in $\dF_p^\times$ whose square
equals $\det\bar\rho_\pi(\vartheta^2)$. If
$\varepsilon(\vartheta)^{-1}\bar\rho_\pi(\vartheta^2)$ has geometric
eigenvalues $\{\mu,\mu^{-1}\}$ for some $\mu\not\in\dF_p^{\r{bad}}$, then we
are done. Otherwise, we may modify $\{e_1,e_2\}$ such that
\[\varepsilon(\vartheta)^{-1}\bar\rho_\pi(\vartheta^2)\in
\left\{\left(
         \begin{array}{cc}
           \mu &  \\
            & \mu^{-1} \\
         \end{array}
       \right),\pm\left(
         \begin{array}{cc}
           1 & b \\
           0 & 1 \\
         \end{array}
       \right)\res \mu\in\dF_p^{\r{bad}}, b\in\dF_p
\right\}.\] By (P1), we may choose some $\alpha\in\dF_p^\times$ whose order
is not in $\{1,2,3,4,6,8,12,24\}$. Then $\mu\alpha^2\not\in\dF_p^{\r{bad}}$
for every $\mu\in\dF_p^{\r{bad}}$. If we can show that $\Gal(L_2/\dQ)$
contains the element
\[\left(
         \begin{array}{cc}
           \alpha & 0 \\
           0 & \alpha^{-1} \\
         \end{array}
       \right)^{\otimes 2},
\] then we are done. The rest of the proof is dedicated to this point,
according to types of $A$.
\begin{itemize}
  \item If $A$ is of Type AI, then it follows from (P6).

  \item If $A$ is of Type AII, then we choose a number field $K$ as in
      (P6). Note that
      $\bar\rho_\pi^\theta(\gamma)=\bar\rho_\pi(\vartheta\gamma\vartheta^{-1})$.
      Since $A_K$ is isogenous to $A^\theta_K$, the representations
      $\bar\rho_\pi^\theta\res_{\Gamma_K}$ and
      $\bar\rho_\pi\res_{\Gamma_K}$ are conjugate. By (P6), we may
      replace $\vartheta$ by its conjugation by some element in
      $\Gamma_K$, such that
      $\bar\rho_\pi(\vartheta\gamma\vartheta^{-1})=\bar\rho_\pi(\gamma)$
      for all $\gamma\in\Gamma_K$. Note that for $\gamma\in\Gamma_K$, we
      have
      $\pres{\sharp}{\bar\rho}_\pi(\gamma)=\bar\rho_\pi(\gamma)\otimes\bar\rho_\pi(\vartheta\gamma\vartheta^{-1})=
      \bar\rho_\pi(\gamma)\otimes\bar\rho_\pi(\gamma)$. The lemma
      follows.

  \item If $A$ is of Type B, then the proof is the same as above by
      letting $K=\breve{F}F$.
\end{itemize}
\end{proof}

\begin{lem}
Suppose that $p$ is a good prime as before. Then Assumption \ref{as:group_r}
(R6) is satisfied.
\end{lem}

\begin{proof}
Suppose that $A$ is of Type AI or AII. By Lemma \ref{le:large_image}, the
geometric irreducibility of $\pres{\sharp}\rV_\pi$ implies that of
$\rV_{\sigma,\pi}$. But
$\pres{\sharp}\rV_\pi\res_{\Gamma_F}\simeq\rV_\pi\otimes\rV_\pi^\theta$ is
absolutely reducible only when $\pi$ is Asai-decomposable, which is not the
case.

Suppose that $A$ is of Type B. By Lemma \ref{le:large_image}, the geometric
irreducibility of $\rV_{\breve\pi}$ implies that of
$\rV_{\sigma\times\breve\pi}$. Note that $\rV_{\breve\pi}\res_{\Gamma_F}$ is
isomorphic to a quadratic twist of $\Sym^2\rV_\pi$, and the latter is
absolutely irreducible since $\pi$ is non-dihedral.

Again by Lemma \ref{le:large_image}, the image of $\bar\rho_{\sigma,\pi}$
contains a non-trivial scalar matrix. Then its Teichm\"{u}ller lift will be
the non-trivial scalar matrix of order dividing $p-1$ contained in the image
of $\rho_{\sigma,\pi}$ or $\rho_{\sigma\times\breve\pi}$.
\end{proof}

\begin{lem}
Suppose that $p$ is a good prime as before. Then Assumption \ref{as:group_r}
(R4) is satisfied.
\end{lem}

\begin{proof}
For this statement, we may again replace $\pres{\sharp}{\bar\rT}_\pi$ by the
tensor induction of $\bar\rT_\pi$. By Lemma \ref{le:good_image}, we fix an
element $\gamma$ in the image of $\pres{\sharp}{\bar\rho}_\pi$ whose
geometric eigenvalues are $\{1,-1,\mu,\mu^{-1}\}$ for some
$\mu\not\in\dF_p^{\r{bad}}$. By (P1), it is elementary to check that there is
an element $b\in\dF_p^\times$ such that $b^4\not\in\{1,-1,\mu-1,-\mu-1\}$.

Suppose that $A$ is of Type AI or AII. Then $\breve{F}=\dQ$ in this case. By
Lemma \ref{le:large_image}, the image of
$\bar\rho_\sigma\times\pres{\sharp}{\bar\rho}_\pi$ contains (some conjugate
of)
\[\left(\left(
          \begin{array}{cc}
            \epsilon 1 &  \\
             & \epsilon b^2 \\
          \end{array}
        \right)
,\gamma\right)\] for both $\epsilon=\pm$. Then every prime $\ell\nmid
2pqN^-\cD_{N^+M}$ such that
$(\bar\rho_\sigma\times\pres{\sharp}{\bar\rho}_\pi)(\Fr_\ell)$ is conjugate
to the above element will be a strongly $(1,\epsilon)$-admissible prime.

Suppose that $A$ is of Type B. Since $\breve{F}$ is unramified outside
$M\disc(F)$, we have a similar result for $\bar\rT_{\sigma\otimes\breve\eta}$
as in Lemma \ref{le:large_image}. Then the remaining argument is the same as
the previous case.
\end{proof}

Theorems \ref{th:main_even} and \ref{th:main_odd} are finally proved!

\end{document}